\documentclass[12pt, twoside,reqno]{amsart}

\usepackage{times,amssymb,amsmath,rotating,multirow,color,enumerate,url,mathrsfs,mathrsfs}
\usepackage{arydshln}
\usepackage{bbm}
\usepackage{esint}
\usepackage{gensymb}
\usepackage{epstopdf}
\usepackage{comment}
\usepackage{leftidx}
\usepackage{amsthm}
\usepackage{mathtools}
\usepackage{enumerate}
\usepackage{subfig}
\usepackage{amsaddr}
\usepackage{accents}
\newtheorem{theorem}{Theorem}[section]
\newtheorem{corollary}[theorem]{Corollary}
\newtheorem{lemma}[theorem]{Lemma}
\newtheorem{proposition}[theorem]{Proposition}

\theoremstyle{definition}
\newtheorem{definition}[theorem]{Definition}
\newtheorem{remark}[theorem]{Remark}
\newtheorem{example}[theorem]{Example}

\numberwithin{equation}{section}

\newcommand{\ubar}[1]{\underaccent{\bar}{#1}}
\newcommand{\R}{\mathbb{R}}
\newcommand{\Rd}{\R^d}
\newcommand{\Om}{{\Omega}}
\newcommand{\Ob}{{\overline{\Omega}}}
\newcommand{\hk}{{H}}
\newcommand{\hf}{{\mathscr{C}}}
\newcommand{\Hf}{{\mathscr{H}_0}}
\newcommand{\Hnc}{{K}}

\newcommand{\Hs}{{\mathscr{H}}}
\newcommand{\Hc}{{\mathscr{H}_1}}
\newcommand{\Hch}{{\bar{\mathscr{H}}_1}}
\newcommand{\Hcc}{{\ubar{\mathscr{H}}_1}}
\newcommand{\HM}{{\mathrm{conv}(\Hs_{axial})}}
\newcommand{\Hax}{{\mathscr{H}_{axial}}}
\newcommand{\argu}{{\,\cdot\,}}
\newcommand{\LSdd}{{\mathscr{L}(\Sdd)}}
\newcommand{\Mes}{{\mathcal{M}}}
\newcommand{\Jlam}[1]{{J_{#1}}}
\newcommand{\Dd}{{\bigl(\D(\Rd)\bigr)^d}}
\newcommand{\MesH}{{\Mes\bigl(\Ob;\LSdd \bigr)}}
\newcommand{\MesHH}{{\Mes\bigl(\Ob;\Hs \bigr)}}
\newcommand{\MesF}{{\Mes\bigl(\Ob;\Rd \bigr)}}
\newcommand{\MesT}{{\Mes\bigl(\Ob;\Sdd \bigr)}}
\newcommand{\sig}{{\sigma}}
\newcommand{\Fl}{{F}}
\newcommand{\TAU}{{\tau}}
\newcommand{\Comp}{{\mathcal{C}}}
\newcommand{\cost}{{c}}
\newcommand{\tr}{{\mathrm{Tr}}}
\newcommand{\FMD}{{(\mathrm{FMD})}}
\newcommand{\Cmin}{{\mathcal{C}_{\mathrm{min}}}}
\newcommand{\Totc}{{C_0}}
\newcommand{\ro}{{\rho}}
\newcommand{\dro}{{\rho^0}}
\newcommand{\jh}{{\bar{j}}}
\newcommand{\jc}{{\bar{j}^{\,*}}}
\newcommand{\Prob}{{(\mathcal{P})}}
\newcommand{\relProb}{{(\overline{\mathcal{P}})}}
\newcommand{\dProb}{{(\mathcal{P}^*)}}

\newcommand{\Uc}{{\overline{\mathcal{U}}_1}}

\newcommand{\Leb}{{\mathcal{L}}}

\newcommand{\U}{{\mathcal{U}}}

\newcommand{\V}{{\mathcal{V}}}

\newcommand{\pairing}[1]{{\left \langle #1 \right \rangle}}
\newcommand{\norm}[1]{\Arrowvert #1 \Arrowvert}
\newcommand{\abs}[1]{{\left \lvert #1 \right \rvert}}

\newcommand{\eps}{\varepsilon}
\newcommand{\Rb}{\overline{\mathbb{R}}}

\newcommand{\D}{\mathcal{D}}

\newcommand{\Ha}{\mathcal{H}}
\newcommand{\DIV}{\mathrm{div}}

\newcommand{\Ker}{\mathrm{Ker}}

\newcommand{\sign}[1]{\text{sign} \,#1}

\newcommand{\mres}{\mathbin{\vrule height 1.6ex depth 0pt width
		0.13ex\vrule height 0.13ex depth 0pt width 1.3ex}}

\newcommand{\Sdd}{{\mathcal{S}^{d \times d}}}

\begin{document}

	\title[On the Free Material Design problem]{Setting the Free Material Design problem through the methods of optimal mass distribution}

	
	\author{Karol Bo{\l}botowski}

	\address{Department of Structural Mechanics and Computer Aided Engineering, Faculty of Civil Engineering, Warsaw University of Technology, 16 Armii Ludowej Street, 00-637 Warsaw;\\
	College of Inter-Faculty Individual Studies in Mathematics and Natural Sciences, University of Warsaw, 2C Stefana Banacha St., 02-097 Warsaw
	}
	\email{k.bolbotowski@il.pw.edu.pl}

	\author{Tomasz Lewi\'{n}ski}
	
	\address{Department of Structural Mechanics and Computer Aided Engineering, Faculty of Civil Engineering, Warsaw University of Technology, 16 Armii Ludowej Street, 00-637 Warsaw
	}
	\email{t.lewinski@il.pw.edu.pl}

	
	\date{\today}
	
	\dedicatory{}
	
	\begin{abstract}
		The paper deals with the Free Material Design (FMD) problem aimed at constructing the least compliant structures from an elastic material the constitutive field of which play the role of the design variable in the form of a tensor valued measure $\lambda$ supported in the design domain. Point-wise the constitutive tensor is referred to a given anisotropy class $\mathscr{H}$ while the integral of a cost $c(\lambda)$ is bounded from above. The convex $p$-homogeneous elastic potential $j$ is parameterized by the constitutive tensor. The work  puts forward the existence result and shows that the original problem can be reduced to the Linear Constrained Problem (LCP) known from the theory of optimal mass distribution by G. Bouchitt\'{e} and G. Buttazzo. A theorem linking solutions of (FMD) and (LCP) allows to effectively solve the original problem. The developed theory encompasses several optimal anisotropy design problems known in the literature as well as it unlocks new optimization problems including design of structures made of a material whose elastic response is dissymmetric in tension and compression. By employing the explicitly derived optimality conditions we give several analytical examples of optimal designs.
	\end{abstract}
	
	\maketitle

\section{Introduction}

Under the term  compliance  of an elastic structure we understand the value of the elastic energy stored in the structure subjected to a given static load $F$. In the present paper we consider optimum design of the field of constitutive tensor of a prescribed class of anisotropy. The aim is to find within a domain $\Omega \subset \Rd$ a distribution of the constitutive tensor that minimizes the compliance. Our attention is focused on materials with elastic potential $j=j(\hk, \xi)$ whose arguments are: the 4th-order constitutive  positive semi-definite tensor $\hk$ of suitable symmetries, that henceforward will be shortly called a \textit{Hooke tensor}, and the 2nd-order strain tensor $\xi$, defined as the symmetric part of the gradient of the displacement vector function $u$. The Hooke tensor field, point-wise restricted to a closed convex cone $\Hs$ of our choosing, will be the design variable while imposing a bound $\Totc$ on its total cost being integral of  a norm $\cost = \cost(\hk)$. This problem will be referred to as the \textit{Free Material Design} problem (FMD) in general (also known in the literature under the name \textit{Free Material Optimization}), and as the \textit{Anisotropic Material Design} (AMD) if the anisotropy is not subject to any constraints, namely $\Hs$ is the whole set of Hooke tensors.

In the context of the linear theory of elasticity in which $j(\hk,\xi)=\frac{1}{2} \pairing{\hk\,\xi , \xi}$ and with the unit cost function $\cost(\hk)=\tr \, \hk$ the above problem in the AMD setting has been for the first time put forward in \cite{ringertz1993}, where also a direct numerical method of solving this problem has been proposed. Soon then in \cite{bendsoe1994} the AMD was formally reformulated to a form in which only one scalar variable is involved: $ m :=\tr\, \hk$. There has also been shown that the optimal tensor assumes the singular form: $\check\hk = m\, \tilde{\xi} \otimes \tilde{\xi}$ where point-wise $\tilde{\xi}$ is the normalized strain tensor. Consequently, one eigenvalue of the optimal $\check\hk$ is positive, while the other five eigenvalues vanish. Due to reduction of the number of scalar design variables from 21 (in three dimensions) to 1 an efficient numerical method could be developed, cf. Section 5 in \cite{bendsoe1994}.

The analytical method of paper \cite{bendsoe1994} has been applied in \cite{bendsoe1996} concerning minimum compliance of softening structures. This time the analytical work has been done one step forward showing, at the formal level, how to eliminate the design variable $m$, but due to the necessity of using the optimization tools for smooth optimization problems this reduction had not been used in the next steps, e.g. within the numerical tools, at the cost of increasing the number of design variables. Thus, the mentioned papers: \cite{bendsoe1994},\cite{bendsoe1996} did not make use of possibility of  elimination of all the design variables in the AMD problem. Such elimination leads to the minimization problem of a functional of  linear growth with respect to the stress tensor field running over the set of all stresses satisfying the equilibrium equations. In the present paper this problem is a particular case of the more general problem $\dProb$ if one assumes $\dro$ to be the Euclidean norm on the space of matrices, cf. \eqref{eq:dProb_intro} below. 

To the present authors' knowledge one of the first contributions that puts the AMD problem in rigorous mathematical frames is \cite{werner2000} where a variant of existence result is given. The author used a variational formula for the compliance thus rewriting the original problem as a min-max problem in terms of the Hooke tensor field $\hk$ and the vector displacement function $u$. In order to gain compactness in some functional space of Hooke tensor fields a uniform upper bound $\tr\, \hk(x) \leq m_{max}$ was additionally enforced in \cite{werner2000}, which allowed to establish existence of a solution in the form of a tensor-valued $L^\infty$ function $x \mapsto \check\hk(x)$. Based on a saddle-point result the author also proved that there exists a displacement function $u \in W^{1,2}(\Omega;\Rd)$ solving the linear elasticity problem in the optimally designed body characterized by $\check{\hk} \in L^\infty(\Omega;\Hs)$. Bounding point-wise the trace of Hooke tensor has an advantage of preventing the material density blowing up in the vicinity of singularity sets (e.g. the re-entrant corners of $\Omega$), which should potentially render the optimal design more practical. Contrarily, the extra constraint deprives us of some vital mathematical features: it is no longer possible to reduce the original formulation AMD to the problem $\dProb$ of minimizing the functional of linear growth. It is also notable that combining the local constraint $\tr\, \hk(x) \leq m_{\max}$ with the global one $\int_\Omega \tr\, \hk(x) \, dx \leq C_0$ must surely produce results that depend on the ratio $(m_{\max} \, \abs{\Omega})/\Totc$; in particular once it is below one the bound on the global cost is never sharp and thus may be disposed of. Furthermore, due to the uniform bound on the optimal Hooke tensor field $\check{\hk}$, we should not \textit{a priori} expect the regularity of the fields solving the corresponding elasticity problem to be higher than in the classical case: the displacement $u$ will in general lie in the Sobolev space $W^{1,2}(\Omega;\Rd)$ (discontinuities possible) and the stress tensor function $\sigma\in L^2(\Omega;\Sdd)$ may blow up to infinity. The local upper bound on the trace of Hooke tensor is also kept throughout the papers \cite{zowe1997} or \cite{kocvara2008} that concentrate rather on the numerical treatment. 

Another work that offers an existence result in a FMD-adjacent problem is \cite{haslinger2010} where authors put a special emphasis on controlling the displacement function $u$ in the optimally designed body -- therein a more general design problem is considered that includes additional constraints on both displacement $u$ and the the stress $\sigma$. In order to arrive at a well-posed problem a relaxation is proposed where, apart from initially considered uniform upper bound $\tr\, \hk(x) \leq m_{\max}$, a lower bound $\eps\, \mathrm{Id} \leq \hk(x)$ is imposed as well for some small $\eps>0$; the inequality ought to be understood in the sense of comparing the induced quadratic forms, whilst $\mathrm{Id}:\Sdd \mapsto \Sdd$ is the identity operator. 

As outlined above, the reformulation of AMD problem to the problem $\dProb$, the one of minimizing a functional of linear growth proposed first in \cite{bendsoe1996}, was not utilized in the subsequent works in years 1996-2010 keeping the uniform upper bound $\tr\, \hk(x) \leq m_{max}$ (that guaranteed compactness in $L^\infty$) and applying more direct numerical approaches. This matter was revisited in \cite{czarnecki2012} where the passage from AMD to $\dProb$ played a central role: a detailed, yet still formal, derivation of $\dProb$ via the optimality conditions is therein given. In the same work the problem $\dProb$ was treated numerically. Next, the paper \cite{czarnecki2014} formally put forward a problem dual to $\dProb$ where the virtual work of the load is maximized over displacement functions $u$ that produce strain $e(u)$ point-wise contained in a unit Euclidean ball. In the present work this dual problem may be recovered as a particular case of the problem $\Prob$ by choosing $\rho$ to be again Euclidean norm, cf. \eqref{eq:Prob_intro} below. The idea of reformulating an optimal design problem by a pair of mutually dual problems $\Prob$ and $\dProb$ was inspired by the theory of Michell structures where a pair of this form can be employed to obtain both analytical and highly accurate numerical solutions, cf. \cite{Lewinski2019} or \cite{bouchitte2008}.

As stressed above the solution to the AMD problem is highly singular: only one eigenvalue of the elastic moduli tensor turns out to be positive, the other five vanish. A way of remedying this is by imposing some additional local condition on the type of material's anisotropy: in \cite{czarnecki2015a}, \cite{czarnecki2015b} and later in \cite{czarnecki2017b} the \textit{Isotropic Material Design} problem (IMD) was proposed as another setting of the family of FMD problems. Essentially IMD problem boils down to seeking two scalar fields $K$ and $G$ being, respectively, bulk and shear moduli that fully determine the field of isotropic Hooke tensor $\hk$ for which (in 3D setting) $\tr \,\hk = 3 K + 10 G$. Analogously to the AMD setting, the IMD problem was reformulated to a pair of mutually dual problems of the form $\Prob$, $\dProb$, only this time the functions $\rho, \dro$ are not the Euclidean norms but a certain pair of mutually dual norms on the space of symmetric matrices. A similar effort was made in \cite{czubacki2015} where the \textit{Cubic Material Design} problem (CMD) was approached: the cubic symmetry was imposed on the unknown Hooke tensor field $\hk$ reducing the CMD problem to minimizing over three moduli fields as well as directions of anisotropy. Once again reformulation to a pair $\Prob$, $\dProb$ proved to be feasible with $\rho, \dro$ chosen specifically to CMD problem. Finally, along with isotropy symmetry the Poisson ratio $\nu$ may be fixed as well leading to the \textit{Young Modulus Design} problem (YMD), where only one field of Young moduli $E$ is the design variable, see \cite{czarnecki2017a}.

In summary, throughout years 2012-2017 a family of Free Material Design problems: AMD, CMD, IMD, YMD has been proposed and rewritten as pairs of mutually dual problems $\Prob$ and $\dProb$, specified for each problem via different functions $\rho, \dro$. The present contribution is aimed at mathematically rigorous unification of the theory of FMD family including showing existence results as well as the equivalence with the pair $\Prob$, $\dProb$. The latter issue excludes the possibility of imposing the uniform bound $\tr\,\hk(x) \leq m_{\max}$ hence compactness of the set of admissible Hooke tensor fields must be established in topology other than the one of $L^\infty(\Omega;\Hs)$. The global constraint $\int_\Omega \tr\,\hk\, dx \leq \Totc$ yields merely boundedness of $\hk$ in $L^1(\Omega;\Hs)$. Naturally, due to lack of reflexivity of $L^1(\Omega;\Hs)$, the compactness in in this space is impossible to obtain.

Almost in parallel to the mathematical work \cite{werner2000} on the Free Material Design problem the so-called \textit{Mass Optimization Problem} (MOP) was developed in \cite{bouchitte2001}. In MOP we seek a mass distribution, being a non-negative scalar field $m$, that minimizes the compliance. Roughly speaking MOP is equivalent to a particular case of the FMD problem with the set of admissible Hooke tensors chosen as $\Hs = \bigl\{m\, \hk_0 \, : \, m \geq 0 \bigr \}$ where $\hk_0$ is a fixed strictly positive definite Hooke tensor (once $\hk_0$ is isotropic then MOP is equivalent to YMD problem). Consequently the only constraint is global and it reads $\int_\Omega m \,dx \leq M_0$ for some $M_0>0$. Similarly as in FMD problem the compactness of the set of admissible mass fields in any Lebesgue space $L^q(\Omega)$ is beyond reach. Therefore the authors of \cite{bouchitte2001} depart from the relaxed MOP from the beginning where the design variable is a positive Radon measure supported in the closure of the design domain: $\mu \in \Mes_+(\Ob)$ represents the mass distribution whereas the constraint is simply $\int d\mu \leq M_0$. According to Lebesgue decomposition theorem $\mu = \mu_{ac} +\mu_s$ where $\mu_{ac} = m \,\Leb^d \mres \Omega$ with $m \in L^1(\Omega)$ and $\mu_s$ is the singular part of $\mu$. In contrast to the FMD problem with the uniform bound $\tr \, \hk(x) \leq m_{\max}$ assumed in works like \cite{werner2000} or \cite{haslinger2010} it may happen in MOP that the optimal $\check{m}$ blows up to infinity, which is debatable in terms of manufacturability. On the other hand, however, the optimal singular part $\check{\mu}_s$ could concentrate e.g. on some curve $C \subset \Ob$, namely $\check{\mu}_s = m_C \,\Ha^1\mres C$, which would mean that via MOP we recognize a need for one-dimensional reinforcement of a $d$-dimensional body/structure. This feature of an optimal design problem is well-established in the theory of Michell structures.

In the paper \cite{bouchitte2001} we find a rigorous reformulation of MOP to a pair of mutually dual problems $\Prob$, $\dProb$ mentioned above, which now, in virtue of the measure theoretic setting may be readily written down:
\begin{alignat}{2}
	\label{eq:Prob_intro}
	&\Prob \qquad \qquad Z &&= \sup\biggl\{ \int \pairing{u,F} \ : \ u \in C^1(\Ob;\Rd), \ \rho\bigl( e(u) \bigr) \leq 1 \ \text{ in } \Omega \biggr\} \qquad \qquad\\
	\label{eq:dProb_intro}
	&\dProb\qquad \qquad &&=\min \biggl\{ \int \dro(\TAU) \ : \ \TAU \in \Mes\bigl( \Ob;\Sdd \bigr), \ -\DIV\,  \TAU = F \biggr\}.  \qquad \qquad
\end{alignat}
It must be noted that the applied load is a vector valued measure $F \in \MesF$ which accounts for e.g. point loads, whilst the equilibrium equation $-\DIV\,  \TAU = F$ must be understood in the sense of distributions. The variable $\tau$ in $\dProb$, being a tensor valued measure, seems to play a role of the stress field yet it is not the case. In the present work $\tau$ will be referred to as the \textit{force flux}: indeed $\dProb$ resembles the problem of optimally transporting parts of the vector source $F$ to its other parts; optimal $\TAU$ may be diffused over some subdomain of non-zero Lebesgue measure or rather concentrate on some curve. Once $F$ is balanced the problem $\dProb$ attains a solution $\hat{\TAU}$ while there exists a continuous displacement function $\hat{u} \in C(\Ob;\Rd)$ that solves a version $\relProb$ of the problem $\Prob$ where the differentiability condition is relaxed. One of the main theorems in \cite{bouchitte2001} allows to recover a solution of the original MOP based on solutions $\hat{u}, \hat{\tau}$: the optimal mass reads $\check{\mu} = \frac{M_0}{Z} \dro(\hat{\TAU})$ while the displacement function $\check{u} =  \frac{Z}{M_0} \hat{u}$ and the stress function $\check{\sig} = \frac{d \hat{\TAU}}{d\check{\mu}}$ solve the underlying elasticity problem of the body given by mass $\check{\mu}$ and subject to the load $F$. It is remarkable that $\check{u}$ and $\check{\sig}$ gain extra regularity in comparison to classical elasticity: the function $\check{u}$ is differentiable $\Leb^d$-a.e. in $\Omega$ with $e(u) \in L^\infty(\Omega;\Sdd)$ (see Lemma 2.1 in \cite{bouchitte2008}) while $\check{\sig} \in L_{\check{\mu}}^\infty(\Ob;\Sdd)$ or more precisely the stress $\check{\sig}$ is uniform in the optimal body in the sense that  $\dro(\check{\sig}) \equiv \frac{Z}{M_0}\ $ $\check{\mu}$-a.e. One may think of $\check\sig$ as "micro-stress" referred to elastic medium described by $\check{\mu}$; then the resulting force flux $\hat\TAU = \check{\sig} \check{\mu}$ could be thought of as "macro-stress". Then, although $\check{\sig}$ is bounded and uniform, $\hat{\TAU}$ is in some sense "point-wise unbounded", which is difficult to put in a precise way with $\hat{\TAU} \in \Mes(\Ob;\Sdd)$ being a tensor valued measure. Similar distinction between such "micro-stress" and "macro-stress" (called therein \textit{Hemp's forces}) occurs in theory of Michell structures, see \cite{Lewinski2019}.

The theory of MOP was further developed and generalized in \cite{bouchitte2007}. The present work essentially adapts the methods in \cite{bouchitte2007} to provide a rigorous mathematical framework of the family of Free Material Design problems in the setting of the papers by Czarnecki et al (2012-2017). The choice of the class of anisotropy we are designing, i.e. whether we are in the setting of AMD, IMD etc., shall be determined by a set $\Hs$ being any closed convex cone contained in the set of Hooke tensors. The elastic potential that furnishes the constitutive law of elasticity and, at the same time, describes the dependence on the design variable shall be a two-argument real non-negative function $j:\Hs \times \Sdd \rightarrow \R$. The unit cost of the Hooke tensor at a point $\cost:\Hs \rightarrow \R$ may be picked as restriction of any norm on the space of 4th-order tensors to the set $\Hs$; the standard cost $c = \tr$ is a particular choice. The family of the FMD problems is thus parameterized by $\Hs$, $j$, $c$, which, considering assumptions (H1)-(H5) given in Section \ref{sec:elasticity_problem}, offers a wide variety of optimal design problems. In analogy to \cite{bouchitte2001} or \cite{bouchitte2007} the design variable shall be the tensor valued measure $\lambda \in \Mes(\Ob;\Hs)$ representing the Hooke tensor field; the constraint on the total cost shall read $\int c(\lambda) \leq \Totc$. The measure $\lambda$ may be decomposed to $\lambda = \hf \mu$ where $\hf \in L^\infty_\mu(\Ob;\Hs)$ with $c(\hf) = 1\ $ $\mu$-a.e.; the positive Radon measure $\mu$ again plays the role of the "mass" of the body. The objective is to minimize the elastic compliance $\Comp = \Comp(\lambda)$ under a balanced load $F \in \Mes(\Ob;\Rd)$ that is expressed variationally: $\Comp(\lambda) = \sup_{u\in C^1(\Ob;\Rd)}\bigl\{ \int \pairing{u,F} -\int j\bigl(\lambda,e(u)\bigr) \bigr\}$. The hereby proposed FMD problem falls into class of \textit{structural topology optimization} problems for it determines:
\begin{enumerate}[(i)]
	\item the topology and shape of the optimal body via the closed set $\mathrm{spt} \, \check{\mu}$ which, in general, can be strictly contained in the design domain $\Ob$ (cutting-out property of FMD);
	\item variation of dimension of the optimal structure from point to point in $\Ob$ (solids, shells or bars may appear altogether) and this information is encoded in the geometric properties of $\check{\mu}$, see paper \cite{bouchitte1997} on the space tangent to measure at a point;
	\item the anisotropy $\check\hf(x)$ at $\check{\mu}$-a.e. $x$ that may imply certain singularities of the material, e.g. in the AMD setting the solution $\check\hf$ degenerates $\check{\mu}$-a.e. to have a single positive eigenvalue, whilst for IMD problem it appears typical for the optimal material to be auxetic, see \cite{czarnecki2015b}. 
\end{enumerate}

After checking in Section \ref{sec:elasticity_problem} the well-posedness of the FMD problem by showing weak-* upper semi-continuity of the functional $\lambda \mapsto \int j\bigl(\lambda,e(u) \bigr)$, in Section \ref{sec:from_FMD_to_LCP} we move on to reduce the original problem to the pair of problems $\Prob$, $\dProb$ of identical form as in the work \cite{bouchitte2007} on MOP. One of the paramount differences between the two design problems is the following: for MOP the functions $\rho,\dro$ are data that can be inferred from the given constitutive law whereas here the gauge function $\rho$ is to be computed such that $\frac{1}{p}\bigl(\rho(\xi)\bigr)^p = \max_{\hk \in \Hs, \ \cost(\hk) \leq 1} j(\hk,\xi)$ where $p$ is the exponent of homogeneity of $j(\hk,\argu)$; then $\dro$ is defined as the polar of $\rho$. We stress that $\rho$ depends on all the parameters $\Hs, j, c$ of the family of FMD problems hence the pair $\rho,\dro$ in $\Prob,\dProb$ encodes the actual setting of the FMD problem (e.g. AMD, IMD etc.). The finite dimensional programming problem that gives the pair $\rho, \dro$ is thoroughly studied in Section \ref{sec:anisotropy_at_point}, where in particular we learn that $\frac{1}{p'}\bigl(\dro(\sig)\bigr)^{p'} = \min_{\hk \in \Hs, \ \cost(\hk) \leq 1} j^*(\hk,\sig)$. The final method of solving the FMD problem follows from Theorem  \ref{thm:FMD_LCP}: we learn that for any solutions $\hat{u}$ and $\hat{\TAU}$ of, respectively, $\relProb$ and $\dProb$ the measure $\check{\mu} = \frac{\Totc}{Z}\,\dro(\hat{\TAU})$ is an optimal mass distribution in the FMD problem while the displacement $\check{u} = \bigl(\frac{Z}{\Totc}\bigr)^{p'/p} \hat{u}$ and the stress $\check{\sig} = \frac{d\hat{\TAU}}{d\check{\mu}}$ solve the elasticity problem in the optimal structure. Finally, the optimal function of Hooke tensor $\check{\hf}$ is the one that point-wise solves the problem $j^*\bigl(\check{\hf}(x),\check{\sig}(x)\bigr)=\min_{\hk \in \Hs, \ \cost(\hk) \leq 1} j^*\bigl(\hk,\check\sig(x)\bigr)$; Lemma \ref{lem:measurable_selection} guarantees that there exists such $\check{\mu}$-measurable function $\check{\hf}$. In Section \ref{sec:optimality_conditions} we again build upon \cite{bouchitte2007} to arrive at the optimality conditions for a quadruple $(u,\mu,\sig,\hf)$ in Theorem \ref{thm:optimality_conditions}. The optimality conditions are then employed in Section \nolinebreak \ref{sec:examples} to give analytical solutions of some simple example of FMD problem in its different settings. This includes the settings of AMD and IMD, but we also propose the new \textit{Fibrous Material Design} setting (FibMD) where the set $\Hs$ is chosen as convex hull of the (non-convex) cone of uni-axial Hooke tensors $\hk =  a \ \eta \otimes \eta \otimes \eta \otimes \eta$ with $a \geq 0$ and $\eta$ being a unit vector. Remarkably, in the FibMD setting the pair $\Prob$, $\dProb$ represents exactly the Michell problem as $\rho$ turns out to be the spectral norm, see \cite{strang1983} and \cite{bouchitte2008}. The assumptions (H1)-(H5) allow to take potentials $j$ beyond the classical $j(\hk,\xi)= \frac{1}{2} \pairing{\hk\,\xi , \xi}$ and in Example \ref{ex:dissymetru_tension_compresion} we demonstrate this possibility by proposing a potential $j_\pm$ that is dissymetric for tension and compression while the dependence on $\hk$ is non-linear.

The goal of the concluding Section \ref{sec:outlook} is to show that the theory developed in this work finds its application outside elasticity. The framework of the paper \cite{bouchitte2007} is very general and it applies e.g. to Kirchhoff plates. Section \ref{sec:FMD_for_plates} offers a sketch of adaptation of the FMD theory to Kirchhoff plates, cf. the work \cite{weldeyesus2016} on numerical methods for FMD in plates and shells. Treating the FMD problem in elasticity as vectorial one (the function $u$ is vector valued) in Section \ref{sec:FMD_for_heat_cond} we outline the theory of the scalar FMD problem that, in turn, furnishes optimal field of conductivity tensor. In analogy to \cite{bouchitte2001} we recognize connection of the scalar FMD problem to the \textit{Optimal Transport Problem} (cf. \cite{villani2003topics}), which allows to characterize the optimal conductivity field by means of the optimal transportation plan.

\section{Elasticity framework}
\label{sec:elasticity_problem}

\subsection{Hooke tensor fields and constitutive law. Strain formulation of elasticity theory}
By $\Omega$ we shall understand a bounded open set with Lipschitz boundary, contained in a $d$-dimensional space $\Rd$ (in this work $d = 2$ or $d =3$). The space $\Sdd$ will consist of symmetric 2nd-order tensors representing either the stress or the strain at a point $\Ob$, being the domain of an elastic body: a plate in case of $d=2$ and a solid for $d=3$. Naturally $\Sdd$ is isomorphic to the space of symmetric $d \times d$ matrices. We will use the symbol $\pairing{\argu,\argu}$ to denote the Euclidean scalar product in any finite dimensional space.

In classical elasticity, point-wise in $\Omega$, the anisotropy of the body is characterized by a \textit{Hooke tensor}: a 4-th order tensor that enjoys certain symmetries and is positive semi-definite. In fact the set of Hooke tensors is isomorphic to the subset $\Hf = \big\{ \hk \in \LSdd  :  \hk \text{ is positive semi-definite} \bigr\}$ with $\LSdd$ standing for the space of symmetric operators from $\Sdd$ to $\Sdd$. We thus agree that henceforward we shall (slightly abusing the terminology) speak of Hooke tensors as elements of $\Hf$ being a closed convex cone in $\LSdd$. For a Hooke tensor $\hk \in \Hf$ the notions of eigenvalues $\lambda_i(\hk)$ or the trace $\tr\,\hk$ are thus well established; similarly we may speak of identity element $\mathrm{Id} \in \Hf$ or a tensor product $A \otimes A \in \Hf$ for $A \in \Sdd$. 

In the sequel we will restrict the admissible class of anisotropy by admitting Hooke tensors in a chosen subcone of $\Hf$, i.e.
\begin{equation*}
	\Hs  \text{ is an aribtrary non-trivial closed convex cone contained  in } \Hf.
\end{equation*}
We now display some cases of the cones $\Hs$ that will be of interest to us:
\begin{example}
	\label{ex:Hs_symmetry}
	The subcone $\Hs$ may be chosen so that the condition $\hk \in \Hs$ implies a certain type of anisotropy symmetry, for instance $\Hs = \Hs_{iso}$ will denote the set of isotropic Hooke tensors; we have the characterization
	\begin{equation}
		\label{eq:iso_K_G}
		\Hs_{iso} = \left\{\hk \in \Hf \, : \, \hk = d K \biggl( \frac{1}{d}\, \mathrm{I} \otimes \mathrm{I} \biggr) + 2 G\, \biggl( \mathrm{Id}- \frac{1}{d}\, \mathrm{I} \otimes \mathrm{I} \biggr), \ K,G\geq 0  \right\}
	\end{equation}
	where by $\mathrm{I} \in \Sdd$ we denote the identity matrix, while $\mathrm{Id}$ is the identity operator in $\mathscr{L}(\Sdd)$. The non-negative numbers $K, G$ are the so-called bulk and shear moduli respectively. In case of plane elasticity, i.e. $d=2$, for later purposes we give the relation between the moduli and the pair Young modulus $E$, Poisson ratio $\nu$:
	\begin{equation}
		\label{eq:Young_and_Poisson}
		E =2 \left(\frac{1}{2K}+\frac{1}{2G} \right)^{-1} =  \frac{4 K G}{K+G}, \qquad \nu = \frac{K-G}{K+G}.
	\end{equation}
	It must be stressed, however, that some symmetry classes of the Hooke tensor generate cones that are non-convex. This is the case with classes that distinguishes directions, e.g. orthotropy, cubic symmetry.	   
\end{example}

\begin{example}.
	\label{ex:Hs_Michell}
	Let us denote by $\Hs_{axial}$ the set of uni-axial Hooke tensors, i.e.
	\begin{equation*}
		\Hs_{axial} = \bigl\{ \hk \in \Hf : \hk = a \ \eta \otimes \eta \otimes \eta \otimes \eta, \ a \geq 0, \, \eta \in S^{d-1} \bigr\}.
	\end{equation*}
	where by $S^{d-1}$ we mean the unit sphere in $\Rd$. The set $\Hs_{axial}$ is clearly a cone yet it is non-convex for $d>1$ and thus a natural step is to consider the smallest closed convex cone containing $\Hs_{axial}$, i.e. its closed convex hull:
	\begin{equation*}
		\Hs = \overline{\mathrm{conv}(\Hs_{axial})} = \mathrm{conv}(\Hs_{axial}),
	\end{equation*}
	where we have used the fact that in the finite dimensional space the convex hull of a closed cone is closed.	This family of Hooke tensors relates to materials that are made of one-dimensional fibres. Although $\mathrm{conv}(\Hs_{axial})$ is properly contained in $\Hf$ it contains non-trivial isotropic Hooke tensors.

\end{example}

We proceed to narrow down the class of constitutive laws of elasticity that shall be herein considered, i.e. point-wise relation between the stress tensor $\sig \in \Sdd$ and the strain tensor $\xi \in \Sdd$. We will deal with a family of constitutive relations parametrized by a Hooke tensor $\hk \in \Hs$, therefore the elastic energy potential will depend on two variables:
\begin{equation*}
	j: \Hs \times \Sdd \rightarrow \R;
\end{equation*}
note that we assume that $j$ cannot admit infinite values. It is important that there is no explicit dependence on the spatial variable $x$. Below we state our assumptions on the function $j$. Throughout the rest of the paper we fix an exponent $p \in (1,\infty)$, while $p' = \frac{p}{p-1}$ will stand for its H\"{o}lder conjugate.

We assume that for each $\hk \in \Hs$ there hold:
\begin{enumerate}[(H1)]
	\item \label{as:convex} the function $j(\hk,\argu)$ is real-valued, non-negative and convex on $\Sdd$;
	\item \label{as:p-hom} the function $j(\hk,\argu)$ is positively $p$-homogeneous on $\Sdd$;
\end{enumerate}
whilst for each $\xi \in \Sdd$ there hold:
\begin{enumerate}[(H1)]
	\setcounter{enumi}{2}
	\item \label{as:concave} the function $j(\argu,\xi)$ is concave and upper semi-contiunous on the closed convex cone $\Hs$;
	\item \label{as:1-hom} the function $j(\argu,\xi)$ is one-homogeneous on $\Hs$;
	\item \label{as:elip} there exists $\hk \in \Hs$ such that $j(\hk,\xi) >0$.
\end{enumerate}
It is worth to stress that the condition (H\ref{as:elip}) that gives a kind of ellipticity is weak as it allows \textit{degenerate} tensors $\hk \in \Hs$ for which there exists non-zero $\xi \in \Sdd$ such that $j(\hk,\xi) =0$.

We shall say that a stress tensor $\sigma \in \Sdd$ and a strain tensor $\xi \in \Sdd$ satisfy the constitutive law of elasticity with respect to a Hooke tensor $\hk \in \Hs$ whenever
\begin{equation*}
	\sigma \in \partial j(\hk,\xi),
\end{equation*}
where we agree that henceforward  the subdifferential $\partial j(\hk,\xi)$ will be intended with respect to the second variable; similarly we shall later understand the Fenchel transform $j^*$. This way the constitutive law above may be rewritten as the equality $\pairing{\xi,\sigma} = j(\hk,\xi) + j^*(\hk,\sigma)$.
\begin{example}
	The simplest case of a function $j$ in case when $p=2$ is the one from linear elasticity:
	\begin{equation*}
		j(\hk,\xi) = \frac{1}{2} \pairing{\hk \,\xi,\xi}.
	\end{equation*}
	It is trivial to see that the assumptions (H\ref{as:convex})-(H\ref{as:1-hom}) are satisfied by the function above. The assumption (H\ref{as:elip}) is virtually put on the set $\Hs$ as it has to contain "enough" Hooke tensors.
\end{example}
Next we state several results that will be useful in terms of integral functionals with $j$ as the integrand.
\begin{proposition}
	\label{prop:Carath}
	For a given Radon measure $\mu \in \Mes_+(\Ob)$ let $\hf: \Ob \rightarrow \Hs$ be a $\mu$-measurable tensor valued function. Then the function $j\bigl(\hf(\argu),\argu\bigr) : \Ob \times \Sdd \rightarrow \R$ is a Carath\'{e}odory function, i.e.
	\begin{enumerate}[(i)]
		\item 	for $\mu$-a.e. $x\in \Ob$ the function $j\bigl(\hf(x),\argu\bigr) $ is continuous;
		\item  for every $\xi \in \Sdd$ the function $j\bigl(\hf(\argu),\xi\bigr)$ is $\mu$-measurable.
	\end{enumerate}
\end{proposition}
\begin{proof}
	The statement (i) follows easily from the assumption (H\ref{as:convex}) since every convex function that is finite on the whole finite dimensional space (here $\Sdd$) is automatically continuous.
	
	We fix $\xi \in \Sdd$; to see that (ii) holds it is enough to show that for arbitrary $\alpha \in \R$ the set $\{x\in \nolinebreak \Ob \,:\, j\bigl(\hf(x),\xi\bigr) <\alpha \} $ is $\mu$-measurable. Due to the upper semi-continuity assumption (H\ref{as:concave}) the set $A =\{\hk \in \Hs \,: \, j\bigl(\hk,\xi\bigr) <\alpha \}$ is open in topology of $\mathscr{L}(\Sdd)$ relative to $\Hs$. Since $\hf$ is $\mu$-measurable we obtain that  $\hf^{-1}(A)$ is $\mu$-measurable and the thesis follows.
\end{proof}

	In compliance with convex analysis a convex function restricted to convex subset of a linear space can be equivalently treated as a function defined on the whole space if extended by $+\infty$. Since the real function $j:\Hs \times \Sdd \rightarrow \R$ is concave with respect to first variable $\hk$ we can by analogy speak of an extended real function $j:\mathscr{L}(\Sdd) \times \Sdd \rightarrow \Rb = [-\infty,\infty]$ such that $j(\hk,\xi) = -\infty$ for any $\xi \in \Sdd$ and any $\hk \in \mathscr{L}(\Sdd)\backslash \Hs$.

\begin{proposition}
	\label{prop:usc_j}
	The function $j$ is upper semi-continuous on the product $\LSdd \times \Sdd$, i.e. jointly in variables $\hk$ and $\xi$.
	\begin{proof}
		W fix a pair $(\breve{\hk},\breve{\xi}) \in \LSdd \times \Sdd$. We may assume that $\breve{\hk} \in \Hs$ since otherwise $j(\breve{\hk},\breve{\xi}) = -\infty$ and the thesis follows trivially.
		Let us take any ball $U \subset \LSdd$ centred at $\breve{\hk}$ and introduce a compact set $K = \overline{U} \cap \Hs$. We observe that for any fixed $\xi \in \Sdd$ the set $\{j(\hk,\xi) : \hk \in \nolinebreak K \}$ is bounded in $\R$. The zero lower bound follows from non-negativity of $j \vert_\Hs$ whereas, since $j(\argu,\xi)$ is real-valued concave and upper semi-continuous on $\Hs$, it achieves its finite maximum on $K$. According to Theorem 10.6 in \cite{rockafellar1970} the shown point-wise boundedness combined with convexity of every $j(\hk,\argu)$ imply that the family of functions $\{j(\hk,\argu) : \hk \in K \}$ is equi-continuous on any bounded subset of $\Sdd$. Upon fixing $\eps>0$ we may thus choose $\delta_1>0$ such that
		\begin{equation*}
			\abs{j(\hk,\xi) - j(\hk,\breve{\xi})} <\frac{\eps}{2} \qquad \forall\, \xi \in B(\breve{\xi},\delta_1) \subset \Sdd, \quad \forall\, \hk \in K \subset \Hs,
		\end{equation*}
		where it must be stressed that $K$ does not depend on $\eps$. Due to the upper semi-continuity of $j(\argu,\breve{\xi})$ we can also choose $\delta_2>0$ for which $B(\breve{\hk},\delta_2) \subset U$ and
		\begin{equation*}
			j(\hk,\breve{\xi}) < j(\breve{\hk},\breve{\xi}) + \frac{\eps}{2} \qquad \forall\, \hk \in B(\breve{\hk},\delta_2).
		\end{equation*}
		For any pair $(\hk,\xi) \in \bigl(B(\breve{\hk},\delta_2)\cap\Hs \bigr) \times B(\breve{\xi},\delta_1)$ we therefore obtain
		\begin{equation*}
			j(\hk,\xi) = j(\hk,\breve{\xi}) + \bigl( j(\hk,\xi) - j(\hk,\breve{\xi}) \bigr) < j(\breve{\hk},\breve{\xi}) +\frac{\eps}{2} + \frac{\eps}{2},
		\end{equation*}
		which proves that $j$ is upper semi-continuous on $\Hs \times \Sdd$ being a convex and closed subset of $\LSdd \times \Sdd$. Extending $j$ by $-\infty$  guarantees its upper semi-continuity on $\LSdd \times \Sdd$.
	\end{proof}
\end{proposition}

The elastic properties of a $d$-dimensional body contained in $\Ob$ and, in fact, the shape of the body itself shall be fully determined by a constitutive field or a \textit{Hooke tensor field} represented by a $\LSdd$-valued measure $\lambda \in \Mes\bigl(\Ob;\LSdd \bigr)$; we note that $\lambda$ is compactly supported in $\Rd$. Let $f$ be any norm on the space $\LSdd$ (chosen in the sequel as the cost function), then according to Radon-Nikodym theorem $\lambda$ can be decomposed as follows
\begin{equation}
	\label{eq:lambda_decomp}
	\lambda = \hf \, \mu, \qquad \mu \in \Mes_+(\Ob), \quad \hf \in L^\infty_\mu\bigl(\Ob;\LSdd \bigr), \quad f(\hf) = 1 \ \ \mu\text{-a.e.},
\end{equation}
that is $\mu$ can be computed as variation measure $\abs{\lambda}$ with respect to the norm $f$ while $\hf$ is the Radon-Nikodym derivative $\frac{d \lambda}{ d \abs{\lambda}}$. Unless any confusion may arise, henceforward $\hf, \mu$ shall always denote the unique decomposition of $\lambda$ as above. This way the information $\lambda$ about the Hooke tensor field has been split into two: i) information on the distribution of elastic material $\mu$ that after \cite{bouchitte2001} shall be called \textit{mass distribution}; ii) information on the anisotropy $\hf$.

Displacement of the body shall be expressed by a vector valued function and although the body is essentially contained in the support of $\mu$ it is convenient to start with displacement fields $u$ determined in the whole $\Rd$, more precisely $u \in \Dd$ where $\D(\Rd)$ stands for the standard test space of smooth functions, while $\Dd = \D(\Rd;\Rd)$ denotes its $d$ copies. Next, the strain $\eps$ will be a tensor valued function being the symmetric part of the gradient of $u$:
\begin{equation*}
	\eps = e(u) := \frac{1}{2} \left( \nabla u + (\nabla u)^\top \right) \quad \in \quad \D(\Rd;\Sdd) \ \text{ for } \ u\in \Dd.
\end{equation*}

With a Hooke field $\lambda$ fixed the total strain energy $\Jlam{\lambda}$ of an elastic body is a convex functional on a space of strain functions, more accurately $\Jlam{\lambda}: L^p_\mu\bigl(\Ob;\Sdd \bigr) \rightarrow \Rb$ and it is defined as follows
\begin{equation}
	\label{eq:J_lambda}
	\Jlam{\lambda}(\eps) := \int j\bigl(\lambda,\eps\bigr) = \int j\bigl(\hf(x),\eps(x) \bigr) \mu(dx),
\end{equation}
where we have utilised one-homogeneity of $j$ with respect to the first argument. We note that $J_\lambda$ is proper (does not admit $-\infty$ anywhere and is not constantly $\infty$) and in fact non-negative if and only if $\hf(x) \in \Hs$ for $\mu$-a.e. $x \in \Ob$. Indeed, for any $\hk \notin\Hs$ and arbitrary $\xi \in \Sdd$ one obtains $j(\hk,\xi) = - \infty$. Therefore, although formally $\lambda \in \Mes\bigl(\Ob;\LSdd \bigr)$, the condition on $\Jlam{\lambda}$ being finite will virtually force that the Hooke function $\hf$ point-wise lies in $\Hs$ as desired.

It is elementary that the convex functional $J_\lambda$ is weakly lower semi-continuous on $L^p_\mu\bigl(\Ob;\Sdd \bigr)$. In the process of optimization the Hooke field $\lambda$ will play a role of the design variable and hence we must examine the weak-* upper semi-continuity of a concave functional $J_{(\argu)}(\eps): \MesH \rightarrow \Rb$ for a fixed continuous function $\eps$. It is natural to take the convex functional $-J_{(\argu)}(\eps)$ instead, yet the issue with utilizing the classical theorems (see e.g. \cite{reshetnyak1968}) to show its lower semi-continuity is that $-J_{(\argu)}(\eps)$ admits negative values.

\begin{proposition}	
	\label{prop:usc_J}
	Let us take any $\eps \in C\bigl(\Ob;\Sdd \bigr)$, then the functional $J_{(\argu)}(\eps)$ is weakly-* upper semi-continuous in the space $\MesH$.
	\begin{proof}
		The idea is to show that there exists a continuous function $G: \Sdd \rightarrow \left(\LSdd\right)^* \equiv \LSdd$ such that for every $\xi \in \Sdd$ we obtain a majorization $\pairing{G(\xi),\argu} \geq j(\argu,\xi)$ on $\LSdd$. Once this is established we define $g:\Ob \times \LSdd \rightarrow \Rb$ by
		\begin{equation*}
		g(x,\hk) := \pairing{G\bigl(\eps(x) \bigr),\hk} - j(\hk,\eps(x)). 
		\end{equation*}
		Since $G$ is continuous and $j$ is upper semi-continuous jointly on $\LSdd \times \Sdd$ by Proposition \ref{prop:usc_j}, we see by uniform continuity of $\eps$ that the function $g$ is lower semi-continuous jointly on $\Ob \times \nolinebreak \LSdd$. Then, due to non-negativity of $g$ and its convexity together with positive one-homogeneity with respect to the second variable, it is a classical result (see e.g. \cite{reshetnyak1968} or \cite{bouchitte1988}) that the functional $\lambda \mapsto \int g\bigl(x,\lambda(dx)\bigr)$ is weakly-* lower semi-continuous on $\MesH$. We observe that for any $\eps \in C(\Ob;\Sdd)$
		\begin{equation*}
			J_\lambda(\eps) = \int  \pairing{G\bigl(\eps(x) \bigr),\lambda(x)} -  \int g\bigl(x,\lambda(x)\bigr),
		\end{equation*}
		hence the functional $J_{(\argu)}(\eps)$ is a difference of a continuous linear functional (the function $G \circ \nolinebreak \eps:\Ob \rightarrow \LSdd$ is uniformly continuous) and weakly-* lower semi-continuous functional on $\MesH$, which furnishes the thesis.
		
		To conclude the proof we must therefore show existence of the function $G$. We will work with function $j^- := - j$ (convex and proper in the first variable) instead, while the function $G:\Sdd \rightarrow \LSdd$ must be its linear minorant in the sense displayed above for the majorant (we keep the symbol $G$ nevertheless). First we show that for every $\xi \in \Sdd$ the proper, convex and l.s.c. function $j^-(\argu,\xi):\LSdd \rightarrow \Rb$ is subdifferentiable at the origin, i.e. $\partial_1 j^-(0,\xi) \neq \varnothing$ where in this proof by $\partial_1 j^-$ we shall understand the subdifferential with respect to the first argument. By Theorem 23.3 in \cite{rockafellar1970} the scenario  $\partial_1 j^-(0,\xi) = \varnothing$ can occur only if there exists a direction $\Delta\hk \in \LSdd$ such that the directional derivative with respect to the first argument $j^-_{\Delta\hk}(0,\xi)$ equals $-\infty$. Since $j^-(\argu,\xi)$ is positively homogeneous our argument for subdifferentiability at the origin amounts to verifying that $j^-(\hk,\xi)> -\infty$ for every $\hk$ in a unit sphere in $\LSdd$. But this is trivial since we know that $j^-(\argu,\xi)$ is proper for each $\xi$.
		
		We have thus arrived at a multifunction $\Gamma: \Sdd \ni \xi \mapsto \partial_1 j^-(0,\xi) \in \bigl(2^{\LSdd} \backslash \varnothing\bigr)$ that is convex and closed valued. According to Theorem 3.2" in \cite{michael1956} in order to show that there exists a continuous selection $G$ of $\Gamma$ it suffices to show that $\Gamma$ is l.s.c (in the sense of theory of multifunctions). In turn, Lemma A2 in the appendix of \cite{bouchitte1988} states that $\Gamma$ is l.s.c. if and only if the function $(\Delta \hk,\xi) \mapsto \delta^*\bigl( \Delta \hk \,\vert\, \partial_1 j^-(0,\xi)  \bigr)$ is l.s.c. on $\LSdd \times \Sdd$ where $\delta^*$ denotes the support function. The Theorem 23.2 in \cite{rockafellar1970} says that $\delta^*\left( \Delta \hk \,\vert\, \partial_1 j^-(0,\xi)  \right)$ is exactly the directional derivative of $j^-(\argu ,\xi)$ at $\hk=0$ in the direction $\Delta \hk$, but due to homogeneity of $j^-(\argu ,\xi)$ this derivative is precisely $j^-(\Delta \hk,\xi)$ and all boils down to showing lower semi-continuity of $j^- = -j$, which is guaranteed by Proposition \ref{prop:usc_j} in this work.
	\end{proof}
\end{proposition}

A load that may be applied to an elastic body shall be modelled by a vector-valued measure $\Fl \in \Mes(\Ob;\Rd)$. We will assume that our body is not kinematically supported (fixed), e.g. on a boundary of $\Omega$, so in order to have equilibrium the load $\Fl$ must be balanced (see the next subsection for details). We give a definition of \textit{elastic compliance} of elastic body represented by a Hooke field $\lambda \in \MesH$ or, as we shall henceforward write, $\lambda \in \MesHH$:
\begin{equation}
	\label{eq:compliance_def}
	\Comp(\lambda) := \sup \left\{ \int \pairing{u,F} - \int j\bigl(\lambda,e(u)\bigr) \ : \ u \in \Dd \right\}.
\end{equation}
The maximization problem in \eqref{eq:compliance_def} may be viewed as a strain formulation of elasticity problem. The compliance $\Comp(\lambda)$ is always non-negative and it obviously can equal $\infty$ in case when $\lambda$ is not suitably adjusted to $\Fl$. Naturally, even if $\Comp(\lambda)<\infty$, the maximization problem does not, in general, have a solution in the space of smooth functions. Once $j$ satisfies a suitable ellipticity condition, the relaxed solution may be found in a Sobolev space with respect to measure $\mu = \abs{\lambda}$ denoted by $W^{1,p}_\mu$ which was proposed in \cite{bouchitte1997} and then developed in e.g. \cite{bouchitte2007}. In this paper it is crucial that $j$ may be degenerate in the sense that, in particular, $j\bigl(\hf(x),\eps(x)\bigr)$ may vanish for some non-zero $\eps \in L^p_\mu(\Ob;\Sdd)$ on a set of non-zero measure $\mu$. The discussed theory cannot therefore be applied to every pair of measures $\Fl$ and $\lambda$. However, it will appear in Section \ref{sec:FMD_LCP} that the situation is better if $\lambda$ is optimally chosen for $\Fl$.

\subsection{Stress formulation of elasticity theory}

We begin this subsection with a definition of a field that we shall call a \textit{force flux}. By a force flux that equilibrates a load $\Fl \in \Mes\bigl(\Ob;\Rd \bigr)$ in a closed domain $\Ob$ we shall understand a tensor valued measure $\TAU \in \MesT$ that satisfies
\begin{equation}
	\label{eq:eqeq}
	\DIV \,\tau + F = 0
\end{equation}
in sense of distributions on the whole space $\Rd$. Naturally, the above equation can be equivalently written in the form of \textit{virtual work principle}:
\begin{equation*}
	\int \pairing{e(\varphi),\TAU}  = \int \pairing{\varphi,\Fl} \qquad \forall\, \varphi \in \Dd,
\end{equation*}
which is almost by definition up to using the fact that $\pairing{\nabla \varphi(x),  \sigma} = \pairing{ e(\varphi)(x),  \sigma}$ for all $\sigma \in \Sdd$. It is important to note that $\varphi$ above may not vanish on the boundary $\partial\Omega$ and therefore a Neumann boundary condition is accounted for in \eqref{eq:eqeq}, possibly non-homogeneous once $\Fl$ charges $\partial\Omega$.

For existence of a force flux $\TAU$ that equilibrates a load $\Fl$, an assumption on this load is needed: we say that $\Fl$ is \textit{balanced} when one of the two equivalent conditions is satisfied:
\begin{enumerate}[(i)]
	
	\item the virtual work of $\Fl$ is zero on \textit{the space of rigid body displacement functions} $\mathcal{U}_0$:
	\begin{equation*}
		\int \pairing{u_0,F} = 0 \qquad \forall\, u_0 \in \mathcal{U}_0 := \left \{ u \in C^1\bigl(\Ob;\Rd\bigr) \ : \ e(u) = 0 \right\};
	\end{equation*}
	\item $\Fl$ has zero resultant force and moment:
	\begin{equation*}
		\int \Fl = 0 \quad \text{in} \quad \Rd \qquad \text{and} \qquad   \int \bigl( x_i\, F_j - x_j \, F_i\bigr) = 0 \quad \forall \, i,j \in \{1,\ldots,d\}.
	\end{equation*}
\end{enumerate}
A proof that solution $\TAU$ of \eqref{eq:eqeq} exists if and only if $\Fl$ is balanced may be found in \cite{bouchitte2008}.\textit{ Henceforward we shall assume that the load $\Fl$ is indeed balanced.}

For an elastic body represented by a Hooke field $\lambda \in \MesHH$ and subjected to a balanced load $\Fl \in \MesF$ we derive the dual problem to \eqref{eq:compliance_def} with one of the Fenchel transformations performed in duality pairing $L^p_\mu(\Ob;\Sdd) \, ,\, L^{p'}_\mu(\Ob;\Sdd)$ where $\mu = \abs{\lambda}$ and $\hf = \frac{d\lambda}{d\mu}$:
\begin{equation}
	\label{eq:dual_comp}
	\Comp(\lambda) = \min \left\{ \int j^*\bigl(\hf(x),\sigma(x) \bigr) \, d\mu \ : \ \sigma \in L^{p'}_\mu(\Ob;\Sdd), \ -\DIV ( \sigma \mu ) =\Fl \right\}
\end{equation}
where $j^*$ denotes the Fenchel conjugate with respect to the second variable. Upon acknowledging Proposition \ref{prop:Carath} and the fact that the functional $\int j(\lambda,\argu)\, d\mu : L^p_\mu(\Ob;\Sdd) \rightarrow \R$ is continuous we find that the duality argument is a use of a standard algorithm from Chapter III in \cite{ekeland1999} hence we decide not to display the details. We note that as a part of \eqref{eq:dual_comp} we claim that $\Comp(\lambda) < \infty $ and that the minimizer exists, which is true for balanced $\Fl$. We observe that \eqref{eq:dual_comp} may be considered a dual definition of compliance while the minimization problem itself is a stress-based formulation of the elasticity problem.

\section{The Free Material Design problem}
\label{sec:FMD_problem}

\subsection{Formulation of the optimal design problem}
In the optimization problem herein considered the Hooke field $\lambda \in \MesHH$ will play a role of the design variable. The natural constraint on $\lambda$ will be the bound on the total cost, therefore we must choose a cost integrand $\cost:\Hs \rightarrow \R_+$ that satisfies essential properties: convexity, positive homogeneity, lower semi-continuity on $\Hs$ and $\cost(\hk) = 0 \Leftrightarrow \hk=0$. Since $\Hs$ is a closed convex cone consisting of positively semi-definite tensors, for every non-zero $\hk \in \nolinebreak \Hs$ necessarily $-\hk \notin \Hs$. Then it is easily seen that every such function $\cost$ extends to a norm on the whole space $\LSdd$. It is thus suitable that
\begin{equation*}
	\text{we choose the cost function } \cost\text{ as restriction of any norm on } \LSdd \text{ to } \Hs.
\end{equation*} 
\begin{example}
	In the pioneering work on the Free Material Design problem \cite{ringertz1993} the cost function $\cost$ was proposed as the trace function, i.e.
	\begin{equation*}
		\cost(\hk) = \tr \, \hk = \sum_{i=1}^{N(d)} \lambda_i(\hk) \qquad \forall\, \hk \in \Hs
	\end{equation*}
	where $\lambda_i(\hk)$ denotes $i$-th eigenvalue of tensor (in fact a symmetric operator) $\hk$; $N(d) =\frac{1}{2}\, d\, (1 + d)$ is the dimension of the space of symmetric tensors $\Sdd$ (e.g. $N(2)=3$ and $N(3)=6$). Note that $\tr:\Hs \rightarrow \R_+$ may be extended to the whole space $\LSdd$ by $ \sum_{i=1}^{N(d)} \abs{\lambda_i(\hk)}$ being a norm dual to the spectral one. This is an exceptional example of a cost function $\cost$ for it is linear on $\Hs$.
\end{example} 

Our problem of designing in a domain $\Ob$ an optimal elastic body which equilibrates a balanced load $\Fl$  can be readily posed as a compliance minimization problem:
\vspace{5mm}
\begin{equation}
	\label{eq:FMD_problem}
	\FMD \qquad \qquad \qquad \Cmin = \min \biggl\{ \Comp(\lambda) \ : \ \lambda \in \MesHH, \ \int \cost(\lambda) \leq \Totc \biggr\} \qquad \qquad \qquad
\end{equation}
\vspace{0mm}

\noindent which, due to the point-wise free choice of anisotropy $\hf(x) = \frac{d\lambda}{d\abs{\lambda}}(x) \in \Hs$, receives the name \textit{Free Material Design problem} (FMD). The positive number $\Totc$ is the maximal cost of an elastic body. In the decomposition \eqref{eq:lambda_decomp} the function $f$ could be any norm on $\LSdd$ therefore at this point it is convenient to assume $f=c$ and henceforward by a pair $\hf$, $\mu$ we shall always understand the decomposition $\lambda = \hf \mu$ with $\cost(\hf) = 1\ $ $\mu$-a.e. This way the constraint can be rewritten as $\int \cost(\lambda) = \int \cost(\hf) \, d\mu = \int d\mu \leq \Totc$ which is a constraint on the total mass of the elastic body, cf. \cite{bouchitte2001}.

By recalling the definition \eqref{eq:compliance_def} of the compliance we find that, as a point-wise supremum of a family of convex and weakly-* lower semi-continuous functionals on $\MesH$ (see Proposition \ref{prop:usc_J}), $\Comp$ \nolinebreak is itself convex and weakly-* l.s.c. Since $c$ is a norm, in $\FMD$ we are actually performing minimization over a bounded and thus weakly-* compact set in $\MesH$. We infer that our problem has a solution $\check\lambda$ whenever $\Cmin$ is finite, which is indeed the case for a balanced load $\Fl$.

\subsection{Reduction of the Free Material Design problem to a Linear Constrained Problem}
\label{sec:from_FMD_to_LCP}

With the definition \eqref{eq:compliance_def} of $\Comp(\lambda)$ plugged explicitly into $\FMD$ problem we arrive at a min-max problem:
\begin{equation}
	\label{eq:min-max}
	\Cmin = \inf\limits_{\substack{\lambda \in \Mes(\Ob;\Hs), \\ \int \cost(\lambda) \leq \Totc^{}} } \sup\limits_{\ u\in \mathcal{D}(\Rd;\Rd)} \  \biggl\{ \int \pairing{u,F} - \int j\bigl(\lambda,e(u)\bigr) \biggr\}.
\end{equation}
By acknowledging  Proposition \ref{prop:usc_J} from this work we easily verify the assumptions of Proposition 1 in \cite{bouchitte2007} which allows to interchange $\inf$ and $\sup$ above. We may thus formulate a variant of Theorem 1 from \cite{bouchitte2007}, but first we introduce some additional notions. The function $\jh :\Sdd \rightarrow \R$ shall represent a strain energy that is maximal with respect to admissible anisotropy represented by Hooke tensor $\hk \in \Hs$ of a unit $\cost$-cost:  
\begin{equation}
	\label{eq:jh_def}
	\jh(\xi) := \sup\limits_{\hk \in \Hc} j(\hk,\xi), \qquad \Hc := \bigl\{\hk \in \Hs \ : \ \cost(\hk) \leq 1  \bigr\}.
\end{equation}
As a point-wise supremum of a family of convex functions $\left\{j(\hk,\argu) : \hk \in \Hc \right\}$ the function $\jh$ is convex as well. Furthermore, since each $j(\hk,\argu)$ is positively $p$-homogeneous by assumption (H\ref{as:p-hom}), the function $\jh$ inherits this property. Next, due to concavity and upper semi-continuity of $j(\argu,\xi)$ together with compactness of $\Hc$, we see that $\jh(\xi) = \max\limits_{\hk \in \Hc} j(\hk,\xi) = j(\bar{\hk}_\xi,\xi)$ for some $\bar{\hk}_\xi \in \Hc$ and in particular $\jh$ is finite on $\Sdd$. It is natural to define
\begin{equation*}
	\Hch(\xi) := \bigl\{ \hk \in \Hc : \jh(\xi) = j(\hk,\xi) \bigr\}
\end{equation*}
being  a non-empty, convex and compact subset of $\Hc$ for every $\xi\in \Sdd$. The short over-bar $\bar{\argu}$ will be consistently used in the sequel to stress maximization with respect to Hooke tensor or field and should not be confused with long over-bar $\overline{\argu}$ denoting e.g. the closure of a set.

We have just showed that $\jh$ is a convex, continuous and positively $p$-homogeneous function and it is well-known (see e.g. Corollary 15.3.1 in \cite{rockafellar1970}) that it can be written as
\begin{equation}
	\label{eq:jh_rho}
	\jh(\xi) = \frac{1}{p} \bigl( \ro(\xi) \bigr)^p,
\end{equation}
where $\ro:\Sdd \rightarrow \R_+$ is a positively one-homogeneous function.

From the ellipticity assumption (H5) follows that $\jh(\xi) = 0$ if and only if $\xi = 0$ and the same holds for $\rho$. It is thus straightforward that:
\begin{proposition}
	\label{prop:rho}
	The function $\ro: \Sdd \rightarrow \R_+$ is a finite, continuous, convex positively one-homogeneous function that satisfies for some positive constants $C_1,C_2$
	\begin{equation*}
		C_1 \abs{\xi} \leq \ro(\xi) \leq C_2 \abs{\xi} \qquad \forall\, \xi\in \Sdd.
	\end{equation*} 
\end{proposition}

Our theorem can be readily stated:

\begin{theorem}
	\label{thm:problem_P}
	For a balanced load $\Fl \in \MesF$ the minimum value of compliance in $\FMD$ problem \eqref{eq:FMD_problem} equals
	\begin{equation}
		\label{eq:Cmin}
		\Cmin = \frac{1}{p' \, \Totc^{p'-1}}\ Z^{\,p'}
	\end{equation}
	where we introduce an auxiliary variational problem with a linear objective
	\vspace{5mm}
	\begin{equation}
		\label{eq:Prob}
		\Prob \qquad \qquad Z := \sup \biggl\{ \int \pairing{u,F} \ : \ u\in \Dd, \ \ro\bigl( e(u) \bigr) \leq 1 \text { point-wise in } \Omega \biggr\}. \qquad \qquad
	\end{equation}
	\vspace{0mm}
	\begin{proof}
		Upon swapping $\inf$ and $\sup$ in \eqref{eq:min-max} the latter may be rewritten as
		\begin{equation*}
			\Cmin = \sup\limits_{\ u\in \mathcal{D}(\Rd;\Rd)} \  \biggl\{ \int \pairing{u,F} - \bar{J}\bigl(e(u) \bigr) \biggr\}
		\end{equation*}
		where for any continuous stress field $\eps \in C(\Ob;\Sdd)$ we have
		\begin{equation*}
			\bar{J}(\eps) = \sup\limits_{\lambda \in \Mes(\Ob;\Hs), \ \int \cost(\lambda) \leq \Totc^{} \ } J_\lambda(\eps) = \sup\limits_{\substack{\mu \in \Mes_+(\Ob), \ \int d\mu \,\leq \Totc^{} \\ \hf \in L^1_\mu(\Ob;\Hs), \ \cost(\hf) = 1} } \int j(\hf,\eps)\, d\mu,
		\end{equation*}
		where we decomposed $\lambda$ to $\hf \mu$ with $c(\hf) = 1\ $ $\mu$-a.e. (the symbol $\bar{J}$ is not to be confused with l.s.c. regularization of some functional $J$). Further we fix the strain field $\eps$. For any pair $\hf,\mu$ admissible above we easily find an estimate
		\begin{equation*}
			\int j(\hf,\eps)\, d\mu \leq \int \jh(\eps) \, d\mu \leq \, \norm{\jh(\eps)}_{L^\infty(\Ob)} \int d\mu\, \leq \Totc \,\norm{\jh(\eps)}_{L^\infty(\Ob)}
		\end{equation*}
		which yields $\bar{J}(\eps) \leq  \Totc \,\norm{\jh(\eps)}_{L^\infty(\Ob)}$. We shall show that the RHS of this inequality is attainable for a certain competitor $\bar{\lambda}_\eps$.
		
		By Proposition \ref{prop:rho} we see that due to continuity of $\eps$ on $\Ob$ the function $\jh\bigl(\eps(\argu) \bigr)$ is continuous on a compact set $\Ob$ as well and thus there exists $\bar{x} \in \Ob$ such that  $\norm{\jh(\eps)}_{L^\infty(\Ob)} = \jh\bigl(\eps(\bar{x}) \bigr)$. With a strain function $\eps$ fixed we propose $\bar{\lambda}_\eps = \Totc\,  \bar{\hk}_{\eps(\bar{x})} \, \delta_{\bar{x}}$ where $\bar{\hk}_{\eps(\bar{x})}$ is any Hooke tensor from the non-empty set $\Hch\bigl(\eps(\bar{x}) \bigr)$ and $\delta_{\bar{x}}$ is a Dirac delta measure at $\bar{x}$. It is trivial to check that  $\int \cost(\bar{\lambda}_\eps) =\Totc$ whilst
		\begin{equation*}
			J_{\bar{\lambda}_\eps}(\eps)  = \Totc \int  j\bigl(\bar{\hk}_{\eps(\bar{x})}, \eps(x) \bigr) \, \delta_{\bar{x}}(x) = \Totc\,  j\bigl(\bar{\hk}_{\eps(\bar{x})}, \eps(\bar{x}) \bigr) = \Totc \, \jh\bigl(\eps(\bar{x}) \bigr) = \Totc \, \norm{\jh(\eps)}_{L^\infty(\Ob)},
		\end{equation*}
		which proves that indeed $\bar{J}(\eps) =\Totc \, \norm{\jh(\eps)}_{L^\infty(\Ob)}$ and further that also $\bar{J}(\eps) =\frac{\Totc}{p} \, \left(\norm{\ro(\eps)}_{L^\infty(\Ob)} \right)^p$.

		Next we use a technique that was already applied in \cite{golay2001}: by substitution $u = t \,u_1$ we obtain
		\begin{alignat*}{1}
		\Cmin &= \sup\limits_{\ u\in \mathcal{D}(\Rd;\Rd)} \  \biggl\{ \int \pairing{u,F} - \frac{\Totc}{p} \, \left(\norm{\ro\bigl(e(u)\bigr)}_{L^\infty(\Ob)} \right)^p \biggr\}\\
		& = \sup\limits_{\substack{ u_1\in \mathcal{D}(\Rd;\Rd),\, t\geq 0}}   \biggl\{ \biggl(\int \pairing{u_1,F}\biggr) \,t - \frac{\Totc}{p} \, t^p \ : \ \norm{\ro\bigl(e(u_1)\bigr)}_{L^\infty(\Ob)} = 1\biggr\}\\
		&= \sup\limits_{u_1\in \mathcal{D}(\Rd;\Rd)}   \biggl\{\frac{1}{p' \, \Totc^{p'-1}}\ \biggl(\int \pairing{u_1,F}\biggr)^{\,p'} \ : \ \norm{\ro\bigl(e(u_1)\bigr)}_{L^\infty(\Ob)} \leq 1\biggr\},
		\end{alignat*}
		where, under the assumption that $\int \pairing{u_1,F}$ is non-negative, in the last step we have computed the maximum with respect to $t$ which was attained for $\bar{t} = \bigl(\frac{\int \pairing{u_1,F}}{\Totc}\bigr)^{p'-1}$. Since the power function $(\argu)^{p'}$ is increasing for non-negative arguments the thesis follows.
	\end{proof}
\end{theorem}

Following the contribution \cite{bouchitte2007} we move on by deriving the problem dual to $\Prob$; in contrary to the duality applied in \eqref{eq:dual_comp} the natural duality pairing here is $C(\Ob;\Sdd),\mathcal{M}(\Ob;\Sdd)$. Again the duality argument is standard up to noting that for any $\TAU \in \nolinebreak \MesT$
\begin{equation}
	\label{eq:duality_of_integral}
	\int \dro(\TAU) = \sup \biggl\{ \int \pairing{\eps,\TAU} \ : \ \eps \in C(\Ob;\Sdd),\ \ro(\eps) \leq 1 \text{ in } \Ob \biggr\};
\end{equation}
the reader is referred to e.g. \cite{bouchitte1988} for the proof. The function $\dro:\Sdd \rightarrow \R_+$ represents the function polar to $\ro$, namely for the stress tensor $\sig \in \Sdd$
\begin{equation*}
	\dro(\sig) = \sup\limits_{\xi \in \Sdd}\biggl\{ \pairing{\xi,\sigma} \ : \ \ro(\xi) \leq 1 \biggr\}
\end{equation*}
where we recall that $\pairing{\argu,\argu}$ stands for the Euclidean scalar product in $\Sdd$. With the use of \eqref{eq:duality_of_integral} a standard duality argument (cf. Chapter III in \cite{ekeland1999}) readily produces the dual to the problem $\Prob$:
\vspace{5mm}
\begin{equation}
	\label{eq:dProb}
	\dProb \qquad \qquad \qquad   Z = \min \biggl\{ \int \dro(\TAU) \ : \ \TAU \in \MesT, \ -\DIV \, \TAU = \Fl \biggr\}  \qquad \qquad
\end{equation}
\vspace{0mm}

\noindent and we emphasize that existence of a solution $\hat{\TAU}$ is part of the duality result ($\Fl$ is assumed to be balanced). After \cite{bouchitte2007} the pair of mutually dual problems $\Prob$ in \eqref{eq:Prob} and $\dProb$ in \eqref{eq:dProb} will be named a \textit{Linear Constrained Problem} (LCP).

\subsection{Designing the anisotropy at a point}
\label{sec:anisotropy_at_point}

The function $\jh$ and therefore also the function $\ro$ (see \eqref{eq:jh_rho}) are expressed via finite dimensional programming problem \eqref{eq:jh_def} where function $j$ enters. It is thus a natural step to examine how the polar $\dro$ depends on $j$ or, as it will appear, on $j^*$. By definition of polar $\dro$ for any pair $(\xi,\sig) \in \Sdd\times\Sdd$ there always holds $\pairing{\xi,\sig} \leq \ro(\xi) \, \dro(\sig)$; we shall say that such a pair satisfies the \textit{extremality condition} for $\rho$ and its polar whenever this inequality is sharp. One of the main results of this subsection will state that this extremality condition is equivalent to satisfying the constitutive law $\sig/\dro(\sig) \in \partial j(\breve{\hk},\xi)$ for some $\breve{\hk} \in \Hc$ optimally chosen for $\sigma$. This link will be utilized while formulating the general optimality conditions for $\FMD$ problem in Section \ref{sec:optimality_conditions}. 

Beforehand we investigate the properties of the Fenchel conjugate $j^*$; by its definition, for a fixed $\hk \in \Hs$ we get a function $j^*(\hk,\argu) : \Sdd \rightarrow \Rb$ expressed by the formula
\begin{equation}
	\label{eq:Fenchel_conj_def}
	j^*(\hk,\sig) = \sup\limits_{\zeta \in \Sdd} \bigl\{ \pairing{\zeta,\sig} - j(\hk,\zeta) \bigr\}.
\end{equation}
It is well-established that $j^*(\hk,\argu)$ is convex and l.s.c. on $\Sdd$, it is also proper and non-negative for each $\hk \in \Hs$ since $j(\hk,\argu)$ is real-valued and equals zero at the origin. For a given $\hk \in \Hs$, however, $j^*(\hk,\argu)$ may admit infinite values: take for instance $\hk\in \Hs$ that is a singular tensor and $j(\hk,\xi) = \frac{1}{2} \pairing{\hk\,\xi,\xi}$, then $j^*(\hk,\sig) = \infty$ for any $\sig \not\perp \Ker\, \hk$.
Furthermore it is well-established that $j^*(\hk,\argu)$ is positively $p'$-homogeneous. 
Next, again for a fixed $\hk \in \Hs$, we look at the subdifferential $\partial j(\hk,\argu) : \Sdd \rightarrow 2^\Sdd$. Almost by definition for $\xi,\sig \in \Sdd$
\begin{equation}
	\label{eq:subdiff_Fenchel}
	\sig \in \partial j(\hk,\xi) \qquad \Leftrightarrow \qquad \pairing{\xi,\sig} \geq j(\hk,\xi) + j^*(\hk,\sig),
\end{equation}
while the opposite inequality on the RHS, known as Fenchel's inequality, holds always. By recalling positive $p$-homogeneity of $j(\hk,\argu)$ it is well established that the following repartition of energy holds (see e.g. \cite{rockafellar1970})
\begin{equation}
	\label{eq:repartition}
	\sig \in \partial j(\hk,\xi) \qquad \Leftrightarrow \qquad \left\{
	\begin{array}{l}
		\pairing{\xi,\sig} = p \ \ j(\hk,\xi),\\
		\pairing{\xi,\sig} = p'\, j^*(\hk,\sig).
	\end{array}
	\right.
\end{equation}

We can infer more about the function $j^*$. Since $j(\argu,\zeta)$ is concave and u.s.c. for every $\zeta\in \Sdd$ the mapping $(\hk,\sig) \mapsto \pairing{\zeta,\sig} - j(\hk,\zeta)$ is for each $\zeta$ convex and l.s.c. jointly in $\hk$ and $\sig$. As a result the function $j^*:\LSdd \times \Sdd \rightarrow \Rb$ is also jointly convex and l.s.c. as a point-wise supremum with respect to $\zeta$. It is, although, not so clear at this point whether the function $j^*(\argu,\sig)$ is proper for arbitrary $\sig \in \Sdd$, i.e. we question the strength of the assumption (H\ref{as:elip}). A positive answer to this question shall be a part of the theorem that we will state below. Beforehand we give another property of functions $j$, $j^*$ that we shall utilize later: for any $\hk_1,\hk_2 \in \Hs$ and any $\xi,\sig \in\Sdd$ we have
\begin{alignat}{1}
	\label{eq:super_additiviy_j}
	j(\hk_1+\hk_2,\xi) &\geq j(\hk_1,\xi) + j(\hk_2,\xi), \\
	\label{eq:sub_additivity_j_star}
	j^*(\hk_1+\hk_2,\sig) &\leq \min\bigl\{ j^*(\hk_1,\sig) , j^*(\hk_2,\sig) \bigr\}.
\end{alignat}
The first inequality can be obtained by combining concavity and 1-homogenity of $j(\argu,\xi)$. Next we see that $j^*(\hk_1 + \hk_2,\sig) = \sup_{\zeta \in \Sdd} \bigl\{ \pairing{\zeta,\sig_+} - j(\hk_1 + \hk_2,\zeta) \bigr\}$, which, with the use of \eqref{eq:super_additiviy_j} and non-negativity of $j$, furnishes \eqref{eq:sub_additivity_j_star}.

\begin{theorem}
	\label{thm:rho_drho}
	Let $\ro:\Sdd \rightarrow \R_+$ be the real gauge function defined by \eqref{eq:jh_def} and \eqref{eq:jh_rho} and by $\dro$ denote its polar function. Then the polar function is another real gauge function satisfying  $\tilde{C}_1 \abs{\sig} \leq \dro(\sig) \leq \tilde{C}_2 \abs{\sigma}$ for some positive $\tilde{C}_1,\tilde{C}_2$ and the following statements hold:
	\begin{enumerate}[(i)]
		\item for every stress tensor $\sigma \in \Sdd$
		\begin{equation}
			\label{eq:conjugate_of_jhat}
			\min\limits_{\hk \in \Hc} j^*(\hk,\sig) = \jc(\sig) = \frac{1}{p'} \bigl(\dro(\sig) \bigr)^{p'}
		\end{equation}
		where the continuous function $\jc:\Sdd \rightarrow \R_+$ is the Fenchel conjugate of $\jh =  \max_{\hk \in \Hc}j(\hk,\argu)$;
		\item for a strain tensor $\xi \in \Sdd$ satisfying $\ro(\xi) \leq 1$, an arbitrary non-zero stress tensor $\sig \in \Sdd$ and a Hooke tensor $\breve{\hk} \in \Hc$ the following conditions are equivalent:
		\begin{enumerate}[(1)]
			\item there hold the extremality conditions:
			\begin{equation*}
				\pairing{\xi,\sig} = \dro(\sig) \qquad \text{and} \qquad \breve{\hk} \in \Hcc(\sig)
			\end{equation*}
			where we introduce a non-empty convex compact set of Hooke tensors optimally chosen for $\sig$
			\begin{equation*}
				\Hcc(\sig) :=  \biggl\{ \hk \in \Hc : j^*(\hk,\sig)= \jc(\sig) = \min\limits_{\tilde\hk \in \Hc} j^*\bigl(\tilde\hk,\sig \bigr)  \biggr\};
			\end{equation*}
			\item the constitutive law is satisfied: 
			\begin{equation}
				\label{eq:extreme_const_law}
				\frac{1}{\dro(\sig)} \, \sig \in \partial j\bigl(\breve{\hk},\xi \bigr).
			\end{equation}
		\end{enumerate}
		Moreover for each of the two conditions (1), (2) to be true it is necessary to have $\ro(\xi)=1$; 
		
		\item the following implication is true for every non-zero $\xi,\sig \in \Sdd$
		\begin{equation*}
			\pairing{\xi,\sig} = \ro(\xi) \, \dro(\sig) \qquad \Rightarrow \qquad \Hcc(\sig) \subset \Hch(\xi),
		\end{equation*}
		while in general $\Hcc(\sig)$ may be a proper subset of $\Hch(\xi)$. 
	\end{enumerate}
	\begin{proof}
		The lower and upper bounds on $\dro$ are a straightforward corollary from the analogous property for $\ro$ stated in Proposition \ref{prop:rho}. For the proof of statement (i) we fix a non-zero tensor $\sig \in \Sdd$, then directly by definition of the Fenchel transform \eqref{eq:Fenchel_conj_def} we obtain
		\begin{equation*}
			\inf\limits_{\hk \in \Hc} j^*(\hk,\sig) = \inf\limits_{\hk \in \Hc} \sup\limits_{\zeta \in \Sdd} \bigl\{ \pairing{\zeta,\sig} - j(\hk,\zeta) \bigr\}
		\end{equation*}
		thus arriving at a min-max problem of a very analogous (yet finite dimensional) form to \eqref{eq:min-max}. Again by Proposition 1 from \cite{bouchitte2007} we may swap the order of infimum and supremum and we arrive at
		\begin{equation*}
			\inf\limits_{\hk \in \Hc} j^*(\hk,\sig) = \sup\limits_{\zeta \in \Sdd} \left\{ \pairing{\zeta,\sig} - \sup\limits_{\hk \in \Hc} j(\hk,\zeta) \right\} = \sup\limits_{\zeta \in \Sdd} \biggl\{ \pairing{\zeta,\sig} - \jh(\zeta) \biggr\} = \jc(\sig)
		\end{equation*}
		or, by acknowledging that $\jh(\zeta) = \frac{1}{p} \bigl(\ro(\zeta)\bigr)^p$, we recover the well known result:
		\begin{alignat}{1}
			\label{eq:chain_sup}
			\jc(\sig)=\inf\limits_{\hk \in \Hc} j^*(\hk,\sig) &= \sup\limits_{\zeta \in \Sdd} \biggl\{ \pairing{\zeta,\sig} - \frac{1}{p} \bigl(\ro(\zeta)\bigr)^p \biggr\} = \sup\limits_{\substack{\zeta_1 \in \Sdd\\ t\geq 0}} \biggl\{ t \pairing{\zeta_1,\sig} - \frac{t^p}{p} \ : \ \ro(\zeta_1) = 1 \biggr\}\nonumber \\ 
			& = \sup\limits_{\zeta_1 \in \Sdd} \biggl\{ \frac{1}{p'}\abs{\pairing{\zeta_1,\sig}}^{p'}   \ : \ \ro(\zeta_1) \leq 1  \biggr\} = \frac{1}{p'} \bigl(\dro(\sig) \bigr)^{p'},
		\end{alignat}
		where the maximal problem with respect to $t$ were solved with $\bar{t}_{\zeta_1} = \abs{\pairing{\zeta_1,\sig}}^{p'-1}$. Since the function $\dro$ is real-valued we have actually showed that $\inf_{\hk \in \Hc} j^*(\hk,\sig)$ is finite proving that the function $j^*(\argu,\sig)$ is proper for any $\sig$ and thus, due to its convexity and l.s.c, we know that it admits its minimum on the compact set $\Hc$, hence point (i) is proved.
		
		We move on to prove statement (ii); we fix $\xi,\sig \in \Sdd$ with $\rho(\xi) \leq 1$ and $\breve{\hk} \in \Hc$; it is not restrictive to assume that $\dro(\sig) = 1$. Let us first assume that (1) holds, i.e. that $\pairing{\xi,\sig} = \dro(\sig) = 1$ and $\breve{\hk}$ is an element of the non-empty set $\Hcc(\sig)$, so that $j^*(\breve{\hk},\sig) = \jc(\sig) = \frac{1}{p'} \bigl(\dro(\sig) \bigr)^{p'}$. Since $\pairing{\xi,\sig} \geq \ro(\xi) \dro(\sig)$ we necessarily have $\rho(\xi) = 1$ together with
		\begin{equation*}
			\pairing{\xi,\sig} = \sup\limits_{\zeta_1 \in \Sdd}  \biggl\{ \pairing{\zeta_1,\sig}   \ : \ \ro(\zeta_1) \leq 1  \biggr\}
		\end{equation*}
		and, since $\bar{t}_\xi= \abs{\pairing{\xi,\sig}}^{p'-1}=1$, we infer that $\bar{t}_\xi \, \xi  = \xi $ solves all the maximization problems with respect to $\zeta$ or $\zeta_1$ in the chain \eqref{eq:chain_sup}, the first one in particular. Together with $\breve{\hk} \in \Hcc(\sig)$ this means that $(\breve{\hk},\xi)$ is a saddle point for the functional $(\hk,\zeta) \mapsto  \pairing{\zeta,\sig} - j(\hk,\zeta)$, i.e.
		\begin{alignat*}{2}
			\pairing{\xi,\sig} - j(\breve{\hk},\xi) &=  \max\limits_{\zeta \in \Sdd} \min\limits_{\hk \in \Hc} \biggl\{ \pairing{\zeta,\sig} - j(\hk,\zeta) \biggr\} = \max\limits_{\zeta \in \Sdd} \biggl\{ \pairing{\zeta,\sig} - j(\breve{\hk},\zeta) \biggr\} = j^*(\breve{\hk},\sig)\\
			 &=\min\limits_{\hk \in \Hc} \max\limits_{\zeta \in \Sdd} \biggl\{ \pairing{\zeta,\sig} - j(\hk,\zeta) \biggr\}= \min\limits_{\hk \in \Hc} \biggl\{ \pairing{\xi,\sig} - j(\hk,\xi) \biggr\} =  \pairing{\xi,\sig}  - \jh(\xi),
		\end{alignat*}
		which furnishes two equalities: $\pairing{\xi,\sig} - j(\breve{\hk},\xi) = j^*(\breve{\hk},\sig)$ and $\pairing{\xi,\sig} - j(\breve{\hk},\xi) = \pairing{\xi,\sig}  - \jh(\xi)$ from which we infer that $\sig \in \partial j(\breve{\hk},\xi)$ and, respectively, $\breve{\hk} \in \Hch(\xi)$. The former conclusion gives the implication (1) $\Rightarrow$ (2) while the latter, since $\breve{\hk}$ was arbitrary element of $\Hcc(\sig)$ and $\pairing{\xi,\sig} = \dro(\sig)$, establishes the point (iii) for the case when $\rho(\xi)=1$; then the general setting of (iii) follows by the fact that $\Hch(\xi) = \Hch(t\, \xi)$ for any $t > 0$.
		
		For the implication (2) $\Rightarrow$ (1) we assume that for the triple $\xi,\sig,\breve{\hk}$ with $\rho(\xi)\leq 1$, $\dro(\sig) = 1$, $\breve{\hk} \in \Hc$ the constitutive law \eqref{eq:extreme_const_law} is satisfied. Then, by \eqref{eq:repartition} there holds the repartition of energy: $\pairing{\xi,\sig} = p \, j(\breve{\hk},\xi)$ and $\pairing{\xi,\sig} = p' \, j^*(\breve{\hk},\sig)$. The following chain can be written down:
		\begin{equation*}
			1 = \dro(\sig) = p' \biggl( \frac{1}{p'} \bigl(\dro(\sig)\bigr)^{p'} \biggr) \leq p'\, j^*(\breve{\hk},\sig) = \pairing{\xi,\sig} = p \, j(\breve{\hk},\xi) \leq p\, \biggl( \frac{1}{p} \bigl(\ro(\xi)\bigr)^{p} \biggr) \leq 1
		\end{equation*}
		and therefore all the inequalities above are in fact equalities; in particular we have
		\begin{equation*}
			\pairing{\xi,\sig} = \dro(\sig) =  \rho(\xi) = 1, \qquad \jc(\sig) = \frac{1}{p'} \bigl(\dro(\sig)\bigr)^{p'} = j^*(\breve{\hk},\sig) \quad \Rightarrow \quad \breve{\hk} \in \Hcc(\sig),
		\end{equation*}
		which proves the implication (2) $\Rightarrow$ (1) and the "moreover part" of point (ii), concluding the proof.
	\end{proof}	
\end{theorem}

\begin{example}[\textbf{Anisotropic Material Design problem}]
	\label{ex:rho_drho_AMD}
	We shall compute the functions $\ro$, $\dro$ together with the sets $\Hch(\xi)$, $\Hcc(\sig)$ in the setting of $\FMD$ problem mostly discussed in the literature: the \textit{Anisotropic Material Design} (AMD) setting for the linearly elastic body, more precisely we choose
	\begin{equation}
		\label{eq:AMD_setting}
		\Hs = \Hf, \qquad j(\hk,\xi) =   \frac{1}{2} \pairing{\hk \,\xi , \xi} \quad (p=2), \qquad \cost(\hk) = \tr\, \hk,
	\end{equation}
	i.e. $\Hs$ contains all possible Hooke tensors. Upon recalling that here $\Hc = \left\{\hk \in \Hf \ : \ \tr \, \hk \leq 1  \right\}$ for each $\hk \in \Hc$ we may write down an estimate
	\begin{equation}
		\label{eq:est_upper_AMD}
		j(\hk,\xi) = \frac{1}{2} \pairing{\hk \,\xi,\xi} \leq \frac{1}{2} \left(\max\limits_{i \in \{1,\ldots,N(d)\}} \lambda_i(\hk) \right) \abs{\xi}^2 \leq \frac{1}{2}\, \bigl(\tr\,\hk \bigr)\, \abs{\xi}^2 \leq \frac{1}{2} \, \abs{\xi}^2
	\end{equation}
	($\abs{\xi}= \pairing{\xi,\xi}^{1/2}$ denotes the Euclidean norm of $\xi$) and therefore $\frac{1}{2}\bigl(\ro(\xi) \bigr)^2 = \jh(\xi) \leq \frac{1}{2} \, \abs{\xi}^2$. On the other hand, we may define for a fixed non-zero $\xi\in \Sdd$
	\begin{equation}
		\label{eq:H_xi_AMD}
		\bar{\hk}_\xi =  \frac{\xi}{\abs{\xi}} \otimes \frac{\xi}{\abs{\xi}}
	\end{equation}
	that is a tensor with only one non-zero eigenvalue being equal to 1 and the corresponding unit eigenvector $\xi / \abs{\xi}$ (in fact a symmetric tensor); obviously we have $\tr \, \bar{\hk}_\xi = 1$ and $j(\bar{\hk}_\xi,\xi) = \frac{1}{2} \pairing{\bar{\hk}_\xi \,\xi,\xi} = \frac{1}{2} \, \abs{\xi}^2$,
	which shows that in fact $\jh(\xi) = j(\bar{\hk}_\xi,\xi) = \frac{1}{2} \, \abs{\xi}^2$ and hence
	\begin{equation*}
		\ro(\xi) = \abs{\xi}, \qquad \frac{\xi}{\abs{\xi}} \otimes \frac{\xi}{\abs{\xi}} \in \Hch(\xi).
	\end{equation*}
	It is easy to observe that for a non-zero $\xi$ the tensor $\bar{\hk}_\xi$ is the unique element of the set $\Hch(\xi)$. Indeed, if $\xi /\abs{\xi}$ is kept as one of the eigenvectors of a chosen $\hk \in \Hc$ and $\lambda_1(\hk)$ is the corresponding eigenvalue, then $\hk \neq \bar{\hk}_\xi$ means that $\lambda_1(\hk) <1$ yielding $j(\hk,\xi) = \frac{1}{2} \langle\hk,\bar{\hk}_\xi \rangle \abs{\xi}^2< \frac{1}{2}\abs{\xi}^2$. One may easily verify that we obtain a similar result whenever $\xi /\abs{\xi}$ is not one of the eigenvectors of $\hk$.
	
	It is well known that the polar $\dro$ to the Euclidean norm $\ro = \abs{\argu}$ is again this very norm. Furthermore it is obvious that
	\begin{equation*}
		\pairing{\xi,\sig} = \ro(\xi)\, \dro(\sig) = \abs{\xi} \, \abs{\sig} \qquad \Leftrightarrow \qquad t_1 \xi = t_2 \sig \ \text{ for some }\ t_1,t_2 \geq 0.
	\end{equation*}
	Next, since $\Hch(\xi)$ is a singleton for non-zero $\xi$, for non-zero $\sig$ the point (iii) of Theorem \ref{thm:rho_drho} furnishes:
	\begin{equation*}
		\xi_\sig = t \,\sigma \ \text{ for some } \ t>0 \qquad \Rightarrow \qquad \Hcc(\sig) = \Hch(\xi_\sig) = \left\{ \frac{\xi_\sig}{\abs{\xi_\sig}} \otimes \frac{\xi_\sig}{\abs{\xi_\sig}}\right\} = \left\{ \frac{\sig}{\abs{\sig}} \otimes \frac{\sig}{\abs{\sig}}\right\}.
	\end{equation*}
	Therefore we have obtained
	\begin{equation}
		\label{eq:Hcc_char_AMD}
		\dro(\sig) = \abs{\sig}, \qquad \Hcc(\sig) =  \left\{ \frac{\sig}{\abs{\sig}} \otimes \frac{\sig}{\abs{\sig}}\right\}.
	\end{equation}
	The latter results were given in \cite{czarnecki2012} and were obtained by solving the problem  $\min_{\hk \in \Hc} j^*(\hk,\sig)$ directly.
\end{example}

\begin{remark}
	It is worth noting that in general neither of the sets $\Hch(\xi)$ or even $\Hcc(\sig)$ is a singleton, see Examples \ref{ex:rho_drho_FibMD} and \ref{ex:rho_drho_IMD}  in Section \ref{sec:example_material_settings}.
\end{remark}

\section{The link between solutions of the Free Material Design problem and\\ the Linear Constrained Problem}

\label{sec:FMD_LCP}

In Section \ref{sec:from_FMD_to_LCP} we have expressed the value of minimum compliance $\Cmin$ by value $Z$ being the supremum and infimum in problems $\Prob$ and $\dProb$ respectively. Next we ought to show how solution of the original Free Material Design problem, the optimal Hooke field $\lambda$ in particular, may be recovered from the solution of much simpler Linear Constrained Problem, being the mutually dual pair  $\Prob$, $\dProb$ exactly. Since the form of (LCP) is identical to the one discussed in the paper \cite{bouchitte2007} the present section may be recognized as a variant of the argument in the former work: herein we must additionally retrieve the optimal Hooke function $\hf$. Since, in contrast to $\dProb$, the problem $\Prob$ in general does not attain a solution, the condition on smoothness of $u$ in $\Prob$ must be relaxed. This was already done in \cite{bouchitte2001} or \cite{bouchitte2007} therefore we repeat the result by only sketching the proof of the compactness result:

\begin{proposition}
	Let $\Omega$ be a bounded domain with Lipschitz boundary. Then the set
	\begin{equation}
		\label{eq:U1_def}
		\U_1 := \left\{ u \in \Dd : \ro\bigl(e(u)\bigr) \leq \nolinebreak 1 \text{ in } \Omega \right\}
	\end{equation}
	is precompact in the quotient space $C(\Ob;\Rd) / \U_0$ where $\mathcal{U}_0 = \left \{ u \in C^1\bigl(\Ob;\Rd\bigr) \, : \, e(u) = 0 \right\}$.
	\begin{proof}[Outline of the proof]
		Using Korn's inequality we infer that, up to a function in $\U_0$ (a rigid body displacement function), the set $\U_1$ is a bounded subset of $W^{1,q}(\Omega)$ for any $1 \leq q <\infty$. By taking any $q > d$ we employ Morrey's embedding theorem (which uses Lipschitz regularity of the boundary $\partial \Omega$) to conclude that the H\"{o}lder seminorm $\abs{\argu}_{C^{0,\alpha}}$ is uniformly bounded in $\U_1$ for any exponent $\alpha$ ranging in $(0,1)$ and the thesis follows.
	\end{proof}
\end{proposition}
The load $\Fl$, throughout assumed to be balanced, satisfies the condition $\int\pairing{u_0,\Fl} = 0$ for any $u_0 \in \U_0$. The results above justifies the following relaxation of the problem $\Prob$:
\begin{equation}
	\label{eq:relProb}
	\relProb \qquad \qquad Z = \max \biggl\{ \int \pairing{u,F} \ : \ u \in \Uc \biggr\} \qquad \qquad
\end{equation} 
where $\Uc$ stands for the closure of $\U_1$ in the topology of uniform convergence. The problem $\relProb$ attains a solution that is unique up to a rigid body displacement function $u_0 \in \U_0$. It is worth noting that each $u \in \Uc$ is $\Leb^d$-a.e. differentiable with $e(u) \in L^\infty(\Omega;\Sdd)$, yet still there are $u \notin \mathrm{Lip}(\Omega;\Rd)$ belonging to $\Uc$, which is possible due to the lack of Korn's inequality for $q = \infty$, cf. \cite{bouchitte2008} for details.

It is left to relax the another displacement-based problem appearing in this work, i.e. the one of elasticity \eqref{eq:compliance_def}. After \cite{bouchitte2007} we shall say that for the Hooke field $\lambda \in \MesHH$ the function $\check{u} \in C(\Ob;\Rd)$ is a relaxed solution of \eqref{eq:compliance_def} when $\Comp(\lambda)$ admits a maximizing sequence in $u_n \in \Dd$ with $u_n \rightrightarrows u$ in $\Ob$ (uniformly) and $\ro\bigl(e(u_n)\bigr) \leq (Z/\Totc)^{p'/p}$ in $\Ob$. Simple adaptation of the result in point (ii) of Proposition 2 in \cite{bouchitte2007} allows to infer that such a maximizing sequence exists provided that $\lambda$ is the optimal solution for $\FMD$, which justifies the definition of the relaxed solution.

For any force flux $\TAU \in \MesT$  by $\dro(\TAU)$ we understand a positive Radon measure that for any Borel set $B \subset \Rd$ gives $\dro(\TAU)\,(B) := \int_B \dro(\TAU)$, where the integral is indented in the sense of convex functional on measures (see \eqref{eq:duality_of_integral}). Since $\TAU$ is absolutely continuous with respect to $\dro(\TAU)$ the Radon-Nikodym theorem gives $\TAU = \sig \,\mu$ where $\mu = \dro(\TAU)$ and $\sig = \frac{d\TAU}{d\mu} \in  L^1_\mu(\Ob;\Sdd)$; obviously there must hold $\dro(\sig) = 1\ $ $\mu$-a.e. so in fact $\sig \in  L^\infty_\mu(\Ob;\Sdd)$.

In \cite{bouchitte2007} the authors defined solution of the Linear Constrained Problem as a triple $u,\mu,\sig$ where $u$ solves $\relProb$ and $\tau = \sig\mu$ solves $\dProb$ with $\dro(\sig)=1\ $ $\mu$-a.e. For our purpose, i.e. in order to recover the full solution of the Free Material Design problem, we must speak of optimal quadruples $u,\mu,\sig,\hf$ where for $\mu$-a.e. $x$ the Hooke tensor $\hf(x) \in \Hc$ (of unit $\cost$-cost) is optimally chosen for $\sig(x)$, namely $\hf(x) \in \Hcc\bigl(\sig(x)\bigr)$. Beforehand we must make sure  there always exists such a function $\hf$ that is $\mu$-measurable:

\begin{lemma}
	\label{lem:measurable_selection}
	For a given Radon measure $\mu \in \Mes_+(\Ob)$ let $\gamma: \Ob \rightarrow \Sdd$ be a $\mu$-measurable function. We consider a closed and convex-valued multifunction $\Gamma_\gamma: \Ob \rightarrow 2^\Hc\backslash \varnothing$ as below
	\begin{equation*}
	\Gamma_\gamma(x) := \Hcc\bigl(\gamma(x) \bigr) = \left\{ \hk \in \Hc : \jc\bigl(\gamma(x)\bigr) = j^*\bigl(\hk,\gamma(x)\bigr) \right\}.
	\end{equation*}
	Then there exists a $\mu$-measurable selection $\hf_\gamma: \Ob\rightarrow \Hc$ of the multifunction  $\Gamma_\gamma$, namely
	\begin{equation*}
	\hf_\gamma (x) \in \Hcc\bigl(\gamma(x)\bigr) \qquad  \text{for } \mu\text{-a.e. } x \in \Ob.
	\end{equation*}
	\begin{proof}
		It suffices to prove that the multifunction $\sig \mapsto \Hcc(\sig)$ is upper semi-continuous on $\Sdd$. Then it is also a measurable multifunction and thus there exsits a Borel measurable selection $\bar{\hk}:\Sdd \rightarrow \Hc$, i.e. $\bar{\hk}(\sig) \in \Hcc(\sig)$ for every $\sig \in \Sdd$, see Corollary III.3 and Theorem III.6 in \cite{castaing1977}. Then $\hf_\gamma : = \bar{\hk} \circ \gamma: \Ob \rightarrow \Hc$ is $\mu$-measurable as a composition of Borel measurable and $\mu$-measurable functions.
		
		By definition of the upper semi-continuity of multifunctions we must show that for any open set $U \subset \Hc$ (open in the relative topology of the compact set $\Hc\subset \LSdd$) the set
		\begin{equation*}
			V =  \bigl\{ \sig \in \Sdd \ : \ \Hcc(\sig) \subset  U \bigr\}
		\end{equation*}
		is open in $\Sdd$. Below the set $U$ is fixed; assuming that $V$ is non-empty we choose arbitrary $\breve{\sig} \in V$. We must show that there exists $\delta >0$ such that $B(\breve{\sig},\delta) \subset V$, which may be rewritten as
		\begin{equation}
		\label{eq:lemma_proof_thesis}
		\text{for every } \sig \in B(\breve{\sig},\delta) \text{ there holds:} \qquad  j^*(\hk,\sig) > \jc(\sig) \quad \forall\, \hk \in \Hc \backslash U.
		\end{equation} 
		
		We start proving \eqref{eq:lemma_proof_thesis} by observing that there exists $\eps>0$ such that
		\begin{equation}
		\label{eq:3eps}
		\inf_{\tilde\hk \in \Hc \backslash U} j^*(\tilde\hk,\breve{\sig}) > \jc(\breve{\sig}) + 3 \eps.
		\end{equation}
		Indeed, the compact set $\Hcc(\breve{\sig})$ is a subset of the open set $U$ and therefore lower semi-continuity of $j^*(\argu,\breve{\sig})$ implies that $\inf_{\tilde\hk \in \Hc \backslash U} j^*(\tilde\hk,\breve{\sig})$ must be greater than $ \jc(\breve{\sig}) = \min_{\tilde\hk \in \Hc} j^*(\tilde\hk,\breve{\sig})$ because otherwise the minimum would be attained in the compact set $\Hc \backslash U$, which is in contradiction with $\Hcc(\breve{\sig}) \subset  U$.
		
		For a fixed $\eps$ satisfying \eqref{eq:3eps} we shall choose $\delta>0$ so that for every $\sig \in B(\breve{\sig},\delta)$ there hold
		\begin{equation}
		\label{eq:cont_jc}
		\abs{\,\jc(\sig) - \jc(\breve{\sig})} < \eps,
		\end{equation}
		\begin{equation}
		\label{eq:uniform_l.s.c.}
		j^*(\hk,\sig) \geq \inf_{\tilde\hk \in \Hc \backslash U} j^*(\tilde\hk,\breve{\sig}) - \eps \qquad \forall\, \hk \in \Hc \backslash U.
		\end{equation}
		Possibility of choosing $\delta = \delta_1$ so that \eqref{eq:cont_jc} holds follows from continuity of $\jc$ (see point (i) of Theorem \ref{thm:rho_drho}), while estimate \eqref{eq:uniform_l.s.c.}, being uniform in $\Hc \backslash U$, is more involved. Since $j^*: \Hc \times \Sdd\rightarrow \Rb$ is lower semi-continuous (jointly in both arguments), for every $\tilde{\hk} \in \Hc$ we may pick $\tilde{\delta}=\tilde{\delta}(\tilde{\hk})>0$ such that
		\begin{equation*}
			j^*(\hk,\sig) \geq j^*(\tilde{\hk},\breve{\sig}) - \eps \qquad \forall\, (\hk,\sig) \in B\bigl(\tilde{\hk},\tilde{\delta}(\tilde{\hk})\bigr) \times B\bigl(\breve{\sig},\tilde{\delta}(\tilde{\hk})\bigr).
		\end{equation*}
		Since $\Hc\backslash U$ is compact one may choose its finite subset $\{\tilde{\hk}_i\}_{i=1}^m$ such that $\Hc \backslash U \subset \bigcup_{i=1}^m B\bigl(\tilde{\hk}_i,\tilde{\delta}(\tilde{\hk}_i) \bigr)$. By putting $\delta_2 = \min_{i=1}^m \tilde{\delta}(\tilde{\hk}_i)$ we find that
		\begin{equation*}
			j^*(\hk,\sig) \geq j^*(\tilde{\hk}_i,\breve{\sig}) - \eps \qquad \forall\, (\hk,\sig) \in B\bigl(\tilde{\hk}_i,\tilde{\delta}(\tilde{\hk}_i)\bigr) \times B\bigl(\breve{\sig},\delta_2\bigr) \qquad \forall\, i\in\{1,\ldots,m\},
		\end{equation*}
		and thus, since the finite family of balls covers $\Hc \backslash U$
		\begin{equation*}
			j^*(\hk,\sig) \geq \min\limits_{ i\in\{1,\ldots,m\}}j^*(\tilde{\hk}_i,\breve{\sig}) - \eps \qquad \forall\, (\hk,\sig) \in \bigl(\Hc\backslash U\bigr) \times B\bigl(\breve{\sig},\delta_2\bigr),
		\end{equation*}
		which, by the fact that $\tilde{\hk}_i \in \Hc\backslash U$ for all $i$, furnishes \eqref{eq:uniform_l.s.c.} for any $\sig \in B(\breve{\sig},\delta_2)$.
		
		We fix $\delta = \min\{\delta_1,\delta_2\}$ to have \eqref{eq:cont_jc} and \eqref{eq:uniform_l.s.c.} all together, which, combined with \eqref{eq:3eps}, give for any $\sig \in B(\breve{\sig},\delta)$ and any $\hk \in \Hc \backslash U$
		\begin{equation*}
		j^*(\hk,\sig) \geq \inf_{\tilde\hk \in \Hc \backslash U} j^*(\tilde\hk,\breve{\sig}) - \eps > \left(\jc(\breve{\sig}) + 3 \eps \right)  -\eps = \jc(\breve{\sig})+2\eps > \left(\jc(\sig) - \eps \right) + 2 \eps = \jc(\sig) +\eps,
		\end{equation*}
		which establishes \eqref{eq:lemma_proof_thesis} and thus concludes the proof.
	\end{proof}
\end{lemma}

The definition of a quadruple solving (LCP) may readily be given:

\begin{definition}
	\label{def:LCP_solution}
	By a solution of (LCP) we will understand a quadruple: $\hat{u}\in C(\Ob;\Rd),\ \hat{\mu} \in \Mes_+(\Ob),\ \hat{\sig} \in L^\infty_{\hat{\mu}}(\Ob;\Sdd)$ and $\hat{\hf} \in L^\infty_{\hat{\mu}}(\Ob;\Hs)$ such that:
	$\hat{u}$ solves $\relProb$; $\hat{\TAU} = \hat{\sig} \hat{\mu} \in \MesT$ solves $\dProb$; $\dro(\hat{\sig}) = \nolinebreak 1\ $ $\hat{\mu}$-a.e.; $\hat{\hf}$ is any measurable selection of the multifunction $x \mapsto \Hcc\bigl( \hat{\sig}(x) \bigr)$ which exists by virtue of Lemma \ref{lem:measurable_selection}.
\end{definition}
Then we define a solution of the Free Material Design problem, yet, apart from the Hooke field $\lambda$ being the design variable, we also speak of the stress and the displacement function in the optimal body:
\begin{definition}
	\label{def:FMD_solution}
	By a solution of (FMD) we will understand a quadruple: $\check{u}\in C(\Ob;\Rd),\ \check{\mu} \in \Mes_+(\Ob),\ \check{\sig} \in L^p_{\check\mu}(\Ob;\Sdd)$ and $\check{\hf} \in L^\infty_{\check{\mu}}(\Ob;\Hs)$ such that:
	$\check{\lambda} = \check{\hf} \check{\mu} \in \MesHH$ solves the compliance minimization problem $\Cmin$ with $\cost(\check{\hf}) = 1\ $ $\check\mu$-a.e.; $\check{\sig}$ solves the stress-based elasticity problem \eqref{eq:dual_comp} for $\lambda = \check{\lambda}$; $\check{u}$ is a relaxed solution of the displacement based elasticity problem \eqref{eq:compliance_def} for $\lambda = \check{\lambda}$.
\end{definition}

We give a theorem that links the two solutions defined above:

\begin{theorem}
	\label{thm:FMD_LCP}
	Let us choose a quadruple $\hat{u}\in C(\Ob;\Rd), \hat{\mu} \in \Mes_+(\Ob), \hat{\sig} \in L^1_{\hat{\mu}}(\Ob;\Sdd)$ and $\hat{\hf} \in L^1_{\hat{\mu}}(\Ob;\Hs)$ and define
	\begin{equation}
		\label{eq:link_FMD_LCP}
		\check{\hf} = \hat{\hf}, \qquad \check{\mu} = \frac{\Totc}{Z}\, \hat{\mu}, \qquad \check{\sig} = \frac{Z}{\Totc}\, \hat{\sig}, \qquad \check{u} = \left(\frac{Z}{\Totc}\right)^{p'/p} \hat{u}.
	\end{equation}
	Then the quadruple $\hat{u},\hat{\mu},\hat{\hf},\hat{\sig}$ is a solution of (LCP) if and only if the quadruple $\check{u},\check{\mu},\check{\hf},\check{\sig}$ is a solution of (FMD) problem.
\end{theorem}
Before giving a proof we make an observation that is relevant from the mechanical perspective:
\begin{corollary}
	The stress $\check{\sig}$ that due to the load $\Fl$ occurs in the structure of the optimal Hooke tensor distribution $\check{\lambda} =\check{\hf} \check{\mu}$ is uniform in the sense that
	\begin{equation}
		\check{\sig} \in L^\infty_{\check{\mu}}(\Ob;\Sdd), \qquad \dro(\check{\sig})=\frac{Z}{\Totc} \quad \check{\mu}\text{-a.e.} 
	\end{equation} 
\end{corollary}
	\begin{proof}[Proof of Theorem \ref{thm:FMD_LCP}]
	Let us first assume that the quadruple $\hat{u},\hat{\mu},\hat{\hf},\hat{\sig}$ is a solution of (LCP) and the quadruple $\check{u},\check{\mu},\check{\hf},\check{\sig}$ is defined through \eqref{eq:link_FMD_LCP}. By definition $\hat{\TAU} = \hat{\sig} \hat{\mu}$ is a solution of the problem $\dProb$ and $\dro(\hat\TAU) = \dro(\hat\sig) \,\hat \mu = \hat\mu$. Since $\dro(\hat{\sig}) = 1\ $ $\hat\mu$-a.e. it is straightforward that $\cost(\hat{\hf}) = 1\ $ $\hat\mu$-a.e. as well: indeed, $\hk \in \Hcc(\zeta)$ for non-zero $\zeta$ only if $\cost(\hk) = 1$. Obviously the same concerns $\check{\hf}$. We verify that $\check{\lambda} = \check{\hf} \check{\mu}$ is a feasible Hooke tensor field by computing the total cost:
	\begin{equation}
	\label{eq:feas_lambda}
	\int \cost(\check{\lambda}) = \int \cost(\check{\hf})\, d \check{\mu} = \int d \check{\mu} = \frac{\Totc}{Z} \int  d \hat{\mu} = \frac{\Totc}{Z} \int \dro(\hat{\TAU}) = \Totc,
	\end{equation}
	where we have used that $\hat{\TAU}$ is a minimizer for $\dProb$. In order to prove that $\check{\lambda}$ is a solution for $\Cmin$ it suffices to show that $\Comp(\check{\lambda}) \leq \Cmin$ where $\Cmin = \frac{1}{p' \, \Totc^{p'-1}}\ Z^{\,p'}$ by Theorem \ref{thm:problem_P}. We observe that $\hat\mu$-a.e. $\dro(\check{\sig}) = \frac{Z}{\Totc} \dro(\hat{\sig})  = \frac{Z}{\Totc}$. Since there holds $ \check{\sig} \check{\mu} = \bigl( \frac{Z}{\Totc} \hat{\sig}\bigr) \bigl(\frac{\Totc}{Z} \hat{\mu} \bigr) = \hat{\sig} \hat{\mu} = \hat{\TAU}$, obviously the equilibrium equation $-\DIV (\check{\sig} \check{\mu}) =\Fl$ is satisfied. Due to the assumption of $p$-homogeneity (H\ref{as:p-hom}) the field $\check{\hf} = \hat\hf$ is both a measurable selection for $x \mapsto \Hcc\bigl( \hat{\sig}(x) \bigr)$ and $x \mapsto \Hcc\bigl( \check{\sig}(x) \bigr)$. Then, by the dual stress-based version of the elasticity problem \eqref{eq:dual_comp}
	\begin{equation}
	\label{eq:min_lambda}
	\Comp(\check{\lambda}) \leq \int j^* \bigl(\check{\hf},\check{\sig} \bigr)\, d\check{\mu} = \int \jc\bigl(\check{\sig} \bigr) \, d\check{\mu} = \int \nonumber \frac{1}{p'}\biggl(\dro\bigl(\check{\sig}\bigr)\biggr)^{p'}\, d\check{\mu}
	= \int \frac{1}{p'}\biggl(\frac{Z}{\Totc}\biggr)^{p'}\, d\check{\mu} =\Cmin,
	\end{equation}
	where in the first equality we have used the fact that $\check{\hf}(x) \in \Hcc\bigl( \check{\sig}(x) \bigr)$ for $\hat\mu$-a.e. $x$; in the last equality we acknowledged that $\int d\check{\mu} = \Totc$, see \eqref{eq:feas_lambda}. This proves minimality of $\check{\lambda}$ and we have only equalities in the chain above, which shows that $\check{\sig}$ solves the dual elasticity problem \eqref{eq:dual_comp} for $\lambda = \check{\lambda}$.
	
	In order to complete the proof of the first implication we must show that $\check{u}$ is a relaxed solution for \eqref{eq:compliance_def}. Since $\hat{u}$ is a solution for $\relProb$ there exists a sequence $\hat{u}_n \in \U_1$ such that $\norm{\hat{u}_n - \hat{u}}_\infty \rightarrow 0$. By definition of $\U_1$ we have $\ro\bigl(e(u_n) \bigr) \leq 1$ and therefore by setting $\check{u}_n = \left(Z/\Totc\right)^{p'/p} \hat{u}_n$ we obtain $\ro\bigl( e(\check{u}_n) \bigr) \leq \left(Z/\Totc\right)^{p'/p}$ with  $\norm{\check{u}_n - \check{u}}_\infty \rightarrow 0$. In order to prove that $\check{u}$ is a relaxed solution it is thus left to show that $\check{u}_n$ is a maximizing sequence for \eqref{eq:compliance_def}. We see that
	\begin{alignat*}{1}
	&\liminf\limits_{n \rightarrow \infty} \left\{ \int \pairing{\check{u}_n,\Fl} - \int j\bigl(\check{\hf},e(\check{u}_n)\bigr) \, d\check{\mu}  \right\} \geq \liminf\limits_{n \rightarrow \infty} \left\{ \int \pairing{\check{u}_n,\Fl} - \int \frac{1}{p} \biggl(\ro\bigl(e(\check{u}_n) \bigr) \biggr)^p  d\check{\mu}  \right\} \\
	\geq  &\liminf\limits_{n \rightarrow \infty} \biggl\{ \int \pairing{\check{u}_n,\Fl} - \int \frac{1}{p} \biggl(\frac{Z}{\Totc} \biggr)^{p'}\! d\check{\mu}  \biggr\} =  \lim\limits_{n \rightarrow \infty} \biggl\{ \int \pairing{\check{u}_n,\Fl} \biggr\} - \frac{Z^{\,p'}}{p \, \Totc^{p'-1}}=  \Cmin= \Comp(\check{\lambda}),
	\end{alignat*}
	where we have used the fact that $\lim_{n\rightarrow \infty} \int\pairing{\check{u}_n,\Fl} = \left(Z/\Totc\right)^{p'/p} \lim_{n\rightarrow \infty} \int\pairing{\hat{u}_n,\Fl} = \left(Z/\Totc\right)^{p'/p} Z$. This shows that $\check{u}_n$ is a maximizing sequence for \eqref{eq:compliance_def} thus finishing the proof of the first implication.
	
	Conversely we assume that the quadruple $\check{u},\check{\mu},\check{\hf},\check{\sig}$ is a solution of the (FMD) problem (by definition we have $\cost(\check{\hf}) = \nolinebreak 1\ $ $\check{\mu}$-a.e.) and the quadruple $\hat{u},\hat{\mu},\hat{\hf},\hat{\sig}$ is defined via \eqref{eq:link_FMD_LCP}. The H\"{o}lder inequality furnishes
	\begin{equation}
	\label{eq:Holder}
	\int \dro(\check{\sig}) \, d\check{\mu} \leq \biggl(\int d\check{\mu}\biggr)^{1/p} \biggl( \int \bigl( \dro(\check{\sig}) \bigr)^{p'} \, d\check{\mu} \biggr)^{1/p'} \leq \Totc^{1/p} \biggl( \int \bigl( \dro(\check{\sig}) \bigr)^{p'} \, d\check{\mu} \biggr)^{1/p'}
	\end{equation}
	and the equalities hold only if $\dro(\check{\sig})$ is $\check{\mu}$-a.e. constant and only if either $\int d\check{\mu} = \Totc$ or $\check\sig$ is zero.
	Based on the fact that $\check{\lambda} =\check{\hf} \check{\mu}$ is a solution of $\FMD$ and $\check{\sig}$ is a minimizer in \eqref{eq:dual_comp} we may write a chain
	\begin{alignat*}{1}
	\Cmin = \Comp(\check{\lambda}) = \int j^*(\check{\hf},\check{\sig}) \,d\check{\mu} \geq \int \jc(\check{\sig}) \,d\check{\mu} =  \int \frac{1}{p'} \bigl(\dro(\check{\sig}) \bigr)^{p'} \,d\check{\mu} &\geq \frac{1}{p' \Totc ^{p'/p}}\biggl( 	\int \dro(\check{\sig}) \, d\check{\mu} \biggr)^{p'} \nonumber\\
	&\geq \frac{Z^{p'}}{p' \Totc ^{p'/p}} = \Cmin,
	\end{alignat*}
	where in the last inequality we use the fact that $\check{\TAU} = \check{\sig} \check{\mu}$ is a feasible force flux in $\dProb$. We see that above we have equalities everywhere, which, assuming that $\Cmin >0$ (otherwise the theorem becomes trivial), implies several facts. First, we have $Z= \int \dro(\check{\TAU})$, which shows that $\check{\TAU}$ is a solution for $\dProb$.
	Then, by H\"{o}lder inequality \eqref{eq:Holder} and the comment below it, we obtain that $\int d\check{\mu} = \Totc$ and $\dro(\check{\sig}) = t = \mathrm{const} $ $\check{\mu}$-a.e. Combining those three facts we have $\dro(\check{\sig}) = \frac{Z}{\Totc}$ since $Z = \int \dro(\check{\TAU}) = \int \dro(\check{\sig}) \, d\check{\mu} = t\, \Totc$. From this follows that $\hat{\sig} = \frac{\Totc}{Z} \check{\sig}$ and $\hat\mu = \frac{Z}{\Totc} \check{\mu}$ are solutions for (LCP). As the last information from the chain of equalities we take the point-wise equality $j^*(\check{\hf},\check{\sig}) = \jc(\check{\sig})\ $ $\check{\mu}$-a.e. implying that $\check{\hf}(x) \in \Hcc\bigl(\check{\sig}(x)\bigr)= \Hcc\bigl(\hat{\sig}(x)\bigr)$ for $\hat{\mu}$-a.e. $x$ and thus $\hat\hf = \check{\hf}$ together with the pair $\hat{\sig}, \hat{\mu}$ are solutions for (LCP).
	
	To finish the proof we have to show that $\hat{u} = \bigl(\frac{\Totc}{Z}\bigr)^{p'/p} \check{u}$ is a solution for $\relProb$. It is straightforward to show that $\hat{u} \in \Uc$ based on our definition of the relaxed solution for \eqref{eq:compliance_def} and thus we only have to verify whether $\int \pairing{\hat u,\Fl} = Z$. One can easily show that for $\check{u}$ being a relaxed solution for $\Comp(\check{\lambda})$ there holds the repartition of energy $\int\pairing{\check{u},\Fl} = p' \, \Comp(\check{\lambda})$ (see Proposition 3 in \cite{bouchitte2007}). Since $\Comp(\check{\lambda}) = \Cmin = \frac{Z^{p'}}{p' \Totc^{p'/p}}$ we indeed obtain $\int \pairing{\hat u,\Fl} = \bigl(\frac{\Totc}{Z}\bigr)^{p'/p} \int\pairing{\check{u},\Fl}  = Z$  and the proof ends here.
\end{proof}

\section{Optimality conditions for the Free Material Design problem}
\label{sec:optimality_conditions}

In order to efficiently verify whether a given quadruple $u,\mu,\sig,\hf$ is optimal for (FMD) problem we shall state the optimality conditions. Due to much simpler structure of the problem (LCP) and the link between the two problems in Theorem \ref{thm:FMD_LCP}, it is more natural to pose the optimality conditions for (LCP). Since the form of the latter problem is similar to the one from the paper \cite{bouchitte2007} we will build upon the concepts and results given therein: in addition we must somehow involve the Hooke tensor function $\hf$. We start by quickly reviewing elements of theory of space $T_\mu$ tangent to measure and its implications; for details the reader is referred to the pioneering work \cite{bouchitte1997} and further developments in \cite{bouchitte2003} or \cite{bouchitte2007}. This theory makes it possible to $\mu$-a.e. compute the tangent strain $e_\mu(u)$ for functions $u \in \overline{\U}_1$ that are not differentiable in the classical sense -- this will be essential when formulating point-wise relation between the stress $\sig(x)$ and the strain $e_\mu(u)(x)$.

For given $\mu \in \Mes_+(\Ob)$ we define $j_\mu:\Hs \times \Sdd \times \Ob \rightarrow \R$ such that for $\mu$-a.e. $x$
\begin{equation*}
	j_\mu(\hk,\xi,x) := \inf \biggl\{ j(\hk,\xi + \zeta) \ : \ \zeta \in \mathcal{S}^\perp_\mu(x) \biggr\}
\end{equation*}
where $\mathcal{S}^\perp_\mu(x)$ is the space of symmetric tensors orthogonal to measure $\mu$ at $x$. The characterization follows: $ \mathcal{S}^\perp_\mu(x) = \bigl(\mathcal{S}_\mu(x)\bigr)^\perp$ with $ \mathcal{S}_\mu(x) =  T_\mu(x) \otimes T_\mu(x)$ where $T_\mu(x) \subset \Rd$ is the space tangent to measure. We also introduce $\jh_\mu: \Sdd \times \Ob \rightarrow \R$ for $\mu$-a.e. $x$
\begin{equation}
	\label{eq:jh_mu}
	\jh_\mu(\xi,x) := \inf \biggl\{ \jh(\xi + \zeta) \ : \ \zeta \in \mathcal{S}^\perp_\mu(x) \biggr\}
\end{equation}
and, again by employing Proposition 1 in \cite{bouchitte2007} on interchanging $\inf$ and $\sup$, we observe that
\begin{equation*}
	\jh_\mu(\xi,x) = \inf\limits_{\zeta \in \mathcal{S}^\perp_\mu(x)} \sup_{\hk\in \Hc} \biggl\{ j(\hk,\xi + \zeta) \biggr\} =  \sup_{\hk\in \Hc} \inf\limits_{\zeta \in \mathcal{S}^\perp_\mu(x)} \biggl\{ j(\hk,\xi + \zeta) \biggr\} = \sup_{\hk\in \Hc} \biggl\{ j_\mu(\hk,\xi,x) \biggr\},
\end{equation*}
namely the operations $\hat{(\argu)}$ and $(\argu)_\mu$ commute thus the symbol $\jh_\mu$ is justified. It is
straightforward to show that for each $H \in \Hs$ and $\mu$-a.e. $x$ the function $j_\mu(H,\argu,x)$ inherits the properties of convexity and positive $p$-homogeneity enjoyed by the function $j(H,\argu)$ and therefore its convex conjugate $j^*_{\mu}(H,\argu,x)$ (with respect to the second argument) is meaningful and moreover the repartition of energy analogous to \eqref{eq:repartition} holds whenever $\sig \in \partial\, j_\mu(H,\xi,x)$. On top of that one easily checks that $j^*_{\mu}(H,\sig,x) = j^*(H,\sig)$ whenever $\sig \in \mathcal{S}_\mu(x)$ and $j^*_{\mu}(H,\sig,x) = \infty$ if $\sig \notin \mathcal{S}_\mu(x)$.
 
By $P_\mu(x)$ for $\mu$-a.e. $x$ we will understand an orthogonal projection onto $T_\mu(x)$. Next we introduce an operator $e_\mu : \Uc \rightarrow L^\infty_\mu(\Ob;\Sdd)$ such that for $u \in \Uc$
\begin{equation*}
	e_\mu(u) := P^\top_\mu \, \xi \,  P_\mu \quad \text{for any } \xi \in L^\infty_\mu(\Ob;\Sdd) \text{ such that } \exists u_n \in \U_1 \text{ with } u_n \rightrightarrows u, \ e(u_n) \stackrel{\ast}{\rightharpoonup} \xi \text{ in } L^\infty_\mu(\Ob).
\end{equation*}
The function $\xi$ always exists since the set $e(\U_1)$ is weakly-* precompact in $L^\infty_\mu(\Ob;\Sdd)$ and, although $\xi$ may be non-unique, the field $P^\top_\mu \, \xi \,  P_\mu$ is, see \cite{bouchitte2007}. The following lemma inscribes point (ii) of Theorem \ref{thm:rho_drho} into the frames of theory of space tangent to measure $\mu$: 

\begin{lemma}
	\label{lem:tengantial_constiutive_law}
	Let us take any $u \in \Uc$, $\mu \in \Mes_+(\Ob)$. Then, for $\mu$-a.e. $x \in \Ob$, any non-zero $\sig \in \mathcal{S}_\mu(x)$ ($\sig$ \nolinebreak is a tensor, not a tensor function) and $\breve{\hk} \in \Hc$ the following conditions are equivalent:
	\begin{enumerate}[(i)]
		\item there hold extremality conditions:
		\begin{equation*}
			\pairing{\,e_\mu(u)(x)\,,\,\sig\,} = \dro(\sig) \qquad \text{and} \qquad \breve{\hk} \in \Hcc(\sig);
		\end{equation*}
		\item the constitutive law is satisfied: 
		\begin{equation}
		\label{eq:tangential_constiutive_law}
		\frac{1}{\dro(\sig)} \, \sig \in \partial j_\mu\biggl(\breve{\hk},e_\mu(u)(x),x \biggr),
		\end{equation}
		with subdifferential intended with respect to the second argument of $j_\mu$.
	\end{enumerate}
	\begin{proof}
	 Thanks to Lemma 1 in \cite{bouchitte2007} for a function $u \in \Uc$ we have $\jh_\mu\bigl(e_\mu(u)(x),x\bigr)\leq 1/p$ for $\mu$-a.e. $x$. or in other words for every $x$ in some Borel set $A \subset \Ob$ such that $\mu(\Ob \backslash A) =0$. In the sequel of the proof we fix $x \in A$ for which we treat $\mathcal{S}_\mu(x)$ as a well defined linear subspace of $\Sdd$.
	 
	 Since the minimization problem in \eqref{eq:jh_mu} always admits a solution we find $\zeta \in \mathcal{S}^\perp_\mu(x)$ such that $\jh_\mu\bigl(e_\mu(u)(x),x\bigr) = \jh\bigl(e_\mu(u)(x)+\zeta \bigr)\leq 1/p$ or equivalently $\ro\bigl(e_\mu(u)(x)+\zeta \bigr) \leq 1$ or alternatively $j\bigr(\hk,e_\mu(u)(x)+\zeta\bigl)\leq 1/p$ for each $\hk\in \Hc$. Next we notice that $j^*_{\mu}(\breve{\hk},\sig,x) = j^*(\breve{\hk},\sig)$ due to $\sig \in \mathcal{S}_\mu(x)$. Further we will assume that $\dro(\sig) =1$, which is not restrictive.
	 
	 First we prove the implication (i) $\Rightarrow$ (ii). We shall denote $\xi:=e_\mu(u)(x)+\zeta$ where $\zeta$ is chosen as above. Since $\zeta \in \mathcal{S}^\perp_\mu(x)$ and $\sig \in \mathcal{S}_\mu(x)$ we see that $\pairing{\xi,\sig} = \pairing{e_\mu(u)(x),\sig} = 1$. Since in addition $\rho(\xi) \leq 1$ we see that the triple $\xi,\sig,\breve{\hk}$ satisfies the condition (1) in point (ii) of Theorem \ref{thm:rho_drho} and therefore (2) follows, i.e. $\sig \in \partial j\bigl(\breve{\hk},\xi\bigr)$ or alternatively $\pairing{\xi,\sig} = j(\breve{\hk},\xi ) + j^*\bigl(\breve{\hk},\sig\bigr)$. Due to the remarks above there also must hold $\pairing{e_\mu(u)(x),\sig} = j_\mu\bigr(\breve{\hk},e_\mu(u)(x),x \bigr) + j^*_{\mu}\bigl(\breve{\hk},\sig,x\bigr)$ furnishing \eqref{eq:tangential_constiutive_law} and thus establishing the first implication.
	 
	 For the second implication (ii) $\Rightarrow$ (i) we shall modify the proof of implication (2) $\Rightarrow$ (1) in Theorem \ref{thm:rho_drho}. The constitutive law \eqref{eq:tangential_constiutive_law} implies repartition of energy $\pairing{e_\mu(u)(x),\sig} = p \, j_\mu\bigl(\breve{\hk},e_\mu(u)(x),x\bigr)$ and $\pairing{e_\mu(u)(x),\sig} = p' \, j^*_{\mu}(\breve{\hk},\sig,x)$. In addition we observe that $j^*\bigl(\breve{\hk},\sig\bigr) \geq \jc\bigl(\sig\bigr) = \frac{1}{p'}\bigl(\dro(\sig) \bigr)^{p'} = \frac{1}{p'}$ and we may readily write down a chain
	 \begin{equation*}
	 	1  \leq p'\, j^*\bigl(\breve{\hk},\sig\bigr) = p'\, j^*_{\mu}\bigl(\breve{\hk},\sig,x\bigr) = \pairing{e_\mu(u)(x),\sig} = p \, j_\mu\bigl(\breve{\hk},e_\mu(u)(x),x\bigr) \leq p\, \jh_\mu\bigl(e_\mu(u)(x),x\bigr) \leq 1
	 \end{equation*}
	 being in fact a chain of equalities furnishing $\pairing{e_\mu(u)(x),\sig} = 1$ and $\breve{\hk} \in \Hcc(\sig)$, which completes the proof.
	\end{proof}
\end{lemma}
The optimality condition for the Linear Constrained Problem may readily be given:
\begin{theorem}
	\label{thm:optimality_conditions}
	Let us consider a quadruple ${u}\in C(\Ob;\Rd), {\mu} \in \Mes_+(\Ob), {\sig} \in L^\infty_\mu(\Ob;\Sdd)$, ${\hf} \in L^\infty_\mu(\Ob;\Hs)$ with $\dro({\sig}) = 1 $ and $ \hf \in \Hc \ $ ${\mu}$-a.e. The quadruple solves (LCP) if and only if the following optimality conditions are met:
	\begin{enumerate}[(i)]
		\item $-\DIV ({\sig} {\mu}) = \Fl $;
		\item ${u} \in \Uc$;
		\item $\pairing{e_{\mu}({u})(x),\sig(x)} = 1\quad $ and $ \quad\hf(x) \in \Hcc\bigl( \sig(x) \bigr) \quad$ for $\mu$-a.e. $x$.
	\end{enumerate}
	Moreover, condition (iii) may be equivalently put as a constitutive law of elasticity:
	\begin{enumerate}[(i)']
		\setcounter{enumi}{2}
		\item ${\sig}(x) \in \partial j_{{\mu}} \bigl( {\hf}(x), e_{{\mu}}({u})(x),x \bigr)\quad $ for ${\mu}$-a.e. $x$.
	\end{enumerate}

	\begin{proof}
		Since the form of the duality pair $\relProb$ and $\dProb$ is identical to the one from \cite{bouchitte2007} we may quote the optimality conditions given in Theorem 3 therein: for the triple $(u,\mu,\sig)$ with $u \in \Uc$, $\mu \in \Mes_+(\Ob)$ and $\dro({\sig}) =1$ the following conditions are equivalent:
		\begin{enumerate}[(1)]
			\item $u$ solves the problem $\relProb$ and $\TAU = \sig \mu$ solves the problem $\dProb$;
			\item conditions (i), (ii) hold and moreover $\pairing{e_{\mu}({u}),\sig} = 1\ $ $\mu$-a.e.
		\end{enumerate}
		By Definition \ref{def:LCP_solution} we see that the quadruple $(u,\mu,\sig,\hf)$ satisfying the assumptions of the theorem solves (LCP) if and only if: (1) holds and moreover $\hf(x) \in \Hcc\bigl( \sig(x) \bigr)$ for $\mu$-a.e. $x$. Thus we infer that the quadruple $(u,\mu,\sig,\hf)$ solves (LCP) if and only if: conditions (i), (ii), (iii) hold. The "moreover" part of the theorem follows directly from Lemma \ref{lem:tengantial_constiutive_law}.
	\end{proof}
\end{theorem}

\section{Case study and examples of optimal structures}
\label{sec:examples}

\subsection{Other examples of Free Material Design settings}
\label{sec:example_material_settings}

In Example \ref{ex:rho_drho_AMD} we have computed: $\rho$, $\dro$ together with the extremality conditions for $\xi,\sig$ and the sets of optimal Hooke tensors $\Hch(\xi), \Hcc(\sig)$ in the setting of Anisotropic Material Design (AMD) problem which assumed that $\Hs = \Hs_0$ (all Hooke tensors are admissible) and $j(H,\xi) = \frac{1}{2} \pairing{H \xi, \xi}$ (linearly elastic material). The computed functions and sets virtually define the (LCP) (in AMD setting) which, in accordance with sections above, paves the way to solution of the original (FMD) problem. In the present section we will compute $\rho, \dro$ and $\Hch(\xi), \Hcc(\sig)$ for other settings. Although we will mostly vary the set $\Hs$ of admissible Hooke tensors, we shall also give two alternatives for the energy function $j$ so that the fairly general assumptions (H1)-(H5) are worthwhile. We start with the first one, while the other will be presented at the end of this subsection (cf. Example \ref{ex:power_law}):

\begin{example}[\textbf{Constitutive law of elastic material that is dissymmetric in tension and compresion}]
	\label{ex:dissymetru_tension_compresion}
	For a chosen convex closed cone $\Hs$ let $j: \Hs \times \Sdd \rightarrow \R$ be any elastic potential that meets the assumptions (H1)-(H5). We propose two functions $j_+,j_-: \Hs \times \Sdd \rightarrow \R$ such that for any $\hk \in \Hs$
	\begin{equation}
	\label{eq:j_plus_j_minus_defintions}
	j_+(\hk,\argu) := \bigl(j^*(\hk,\argu)+\mathbbm{I}_{\mathcal{S}^{d\times d}_+}\bigr)^*\qquad \text{and} \qquad j_-(\hk,\argu) := \bigl(j^*(\hk,\argu)+\mathbbm{I}_{\mathcal{S}^{d\times d}_-}\bigr)^* 
	\end{equation}
	which are proposals of elastic potentials of materials that are incapable of withstanding compressive and, respectively, tensile stresses. The sets $\mathcal{S}^{d\times d}_+$ and $\mathcal{S}^{d\times d}_-$ are the convex cones of positive and negative semi-definite symmetric tensors;   $\mathbbm{I}_A$ for $A \subset X$ and any vector space $X$ denotes the indicator function, i.e.  $\mathbbm{I}_A(x) = 0$ for $x\in A$ and $\mathbbm{I}_A(x) = \infty$ for $x \in X \backslash A$. For any $\xi \in \Sdd$ we obtain by introducing a Lagrange multiplier $\zeta$: 
	\begin{alignat*}{1}
	j_+(\hk,\xi) &= \sup\limits_{\sig \in \Sdd} \biggl\{\pairing{\xi,\sig} - j^*(\hk,\sig) - \mathbbm{I}_{\mathcal{S}^{d\times d}_+}(\sig) \biggr\} =  \sup\limits_{\sig \in \mathcal{S}^{d\times d}_+} \biggl\{\pairing{\xi,\sig} - j^*(\hk,\sig) \biggr\}\\
	\nonumber
	&= \sup\limits_{\sig \in \Sdd} \inf\limits_{\zeta \in \mathcal{S}^{d\times d}_+} \biggl\{\pairing{\xi+\zeta,\sig} - j^*(\hk,\sig) \biggr\} =  \inf\limits_{\zeta \in \mathcal{S}^{d\times d}_+} \sup\limits_{\sig \in \Sdd} \biggl\{\pairing{\xi+\zeta,\sig} - j^*(\hk,\sig) \biggr\},
	\end{alignat*}
	where in order to swap the order of $\inf$ and $\sup$ we again used Proposition 1 in \cite{bouchitte2007} (from the beginning we may restrict $\sig$ to some ball in $\Sdd$, which is due to ellipticity $j^*(\hk,\sig) \geq C(\hk) \abs{\sig}^{p'}$ for any $\hk$). By repeating the same argument for $j_-$ we obtain formulas
	\begin{equation}
	\label{eq:j_plus_j_minus_formulas}
	j_+(\hk,\xi) =  \inf\limits_{\zeta \in \mathcal{S}^{d\times d}_+} j(\hk,\xi +\zeta) \qquad \text{and} \qquad j_-(\hk,\xi) =\inf\limits_{\zeta \in \mathcal{S}^{d\times d}_-} j(\hk,\xi +\zeta).
	\end{equation}
	It is now easy to see that the functions $j_+, j_-$ satisfy assumptions (H1)-(H4). Conditions (H\ref{as:convex}) and (H\ref{as:p-hom}) follow directly from definitions \eqref{eq:j_plus_j_minus_defintions} and properties of Fenchel transform. Condition (H\ref{as:concave}) can be easily inferred from \eqref{eq:j_plus_j_minus_formulas}, where functions $j_+(\argu,\xi),j_-(\argu,\xi)$ are point-wise infima of concave u.s.c. functions $j(\argu,\xi)$; one similarly shows (H\ref{as:1-hom}). It is clear, however, that the assumption (H\ref{as:elip}) is not satisfied for either of functions $j_+,j_-$: indeed, there for instance holds $j_+(H,\xi)=0$ for any $H\in \Hs$ and $\xi \in \mathcal{S}^{d\times d}_-$.
	
	In order to restore the condition (H\ref{as:elip}) we define a function $j_\pm: \Hs \times \Sdd \rightarrow \R$ that shall model a composite material that is dissymmetric for tension and compresion:
	\begin{equation*}
		j_\pm(H,\xi) = (\kappa_+)^p \,j_+(H,\xi) + (\kappa_-)^p \,j_-(H,\xi)
	\end{equation*}
	where $\kappa_+,\kappa_-$ are positive reals and $p$ is the homogeneity exponent of $j(H,\argu)$. To show that the condition (H\ref{as:elip}) is met for $j_\pm$ it will suffice to show that $\bar{j}_\pm(\xi) = \max_{\hk \in \Hc} j_\pm(\hk,\xi)$ is greater than zero for any non-zero $\xi$. This amounts to verifying whether $\bar{j}^*_\pm(\sig)$ is finite for any $\sig \in \Sdd$. Using the formula \eqref{eq:conjugate_of_jhat} (its proof did not utilize property (H\ref{as:elip})) and by employing the inf-convolution formula for convex conjugate of sum of functions we obtain
	\begin{equation*}
	\bar{j}^*_\pm(\sig) = \inf\limits_{\hk\in \Hc} j^*_\pm(H,\sig) = \inf\limits_{\hk\in \Hc} \inf\limits_{\substack{\sig_+\in \mathcal{S}^{d\times d}_+ \\ \sig_-\in \mathcal{S}^{d\times d}_- } } \biggl\{ \frac{1}{(\kappa_+)^{p'}}\,j^*(\hk,\sig_+) + \frac{1}{(\kappa_-)^{p'}}\,j^*(\hk,\sig_-) \,  : \,  \sig_+ + \sig_- = \sig\biggr\}.
	\end{equation*}
	Next, since $\jc$ is a real function on $\Sdd$, for any $\sig_+,\sig_-\in \Sdd$ there exist $\hk_+,\hk_-\in \Hc$ such that $j^*(\hk_+,\sig_+) < \infty$, $j^*(\hk_-,\sig_-) < \infty$. We set $\hk_\pm = (\hk_+ + \hk_-)/2\in \Hc$ to discover that \eqref{eq:sub_additivity_j_star} gives $j^*(\hk_\pm,\sig_+) <\infty$ and $j^*(\hk_\pm,\sig_-) <\infty$ and therefore the RHS of the above is finite proving (H\ref{as:elip}).
	
	In summary, the function $j_\pm:\Hs \times \Sdd \rightarrow \R$ satisfies the conditions (H1)-(H5) and thus the Free Material Design problem is well posed for the material that $j_\pm$ models. In particular the function $j_\pm$ is an example of a function which in general non-trivially meets the concavity condition (H\ref{as:concave}): even in the case when $j(H,\xi) = \frac{1}{2} \pairing{\hk\,\xi,\xi}$ the function $j_\pm$ may be non-linear with respect to argument $\hk$. In fact, in the paper \cite{giaquinta1985} in Equation (3.19) the authors construct an explicit formula for energy function that happens to coincide with $j_-(H,\xi)$. The point of departure therein is 2D linear elasticity with an isotropic Hooke tensor $\hk$. We quote their result below (we use bulk and shear constants $K$ and $G$ instead of $E,\nu$, see \eqref{eq:Young_and_Poisson}):
	\begin{equation*}
		j_-(\hk,\xi) = \left\{
		\begin{array}{ccl}
			\frac{1}{2} \pairing{\hk\,\xi,\xi} & \quad \text{if} \quad  \xi \in& \Sigma_1(K,G),\\
			\frac{1}{2}\, \frac{4\, K\, G}{K+G}\, \bigl(\min\{\lambda_1(\xi),\lambda_2(\xi)\} \bigr)^2  & \quad \text{if} \quad  \xi \in& \Sigma_2(K,G),\\
			0 & \quad \text{if} \quad  \xi \in& \mathcal{S}^{d\times d}_+
		\end{array}
		\right.
	\end{equation*}
	where $\Sigma_1(K,G), \Sigma_2(K,G)$ are subregions of $\Sdd$, see \cite{giaquinta1985} for details; we note that the quotient $\frac{4\, K\, G}{K+G}$ above is  the Young modulus $E$, see \eqref{eq:Young_and_Poisson}. If for the cone of admissible Hooke tensors $\Hs$ we choose $\Hs_{iso}$ (see \eqref{eq:iso_K_G} and Example \ref{ex:rho_drho_IMD} below) for a fixed $\xi$ we see that, provided the moduli $K,G$ vary such that $\xi \in \Sigma_2(K,G)$, the function $j_-(\argu,\xi)$ is not linear, i.e. the energy $j_-$ does not depend linearly on $K,G$ and the same will apply to $j_\pm$. This example justifies the need for a fairly general assumption (H\ref{as:concave}) which allows energy functions that does not vary linearly with respect to $\hk$. 
\end{example}

We move on to present another three settings of the Free Material Design problem:

\begin{example}[\textbf{Fibrous Material Design problem}]
	\label{ex:rho_drho_FibMD}
	We present the setting of the \textit{Fibrous Material Design} problem (FibMD) which differs from AMD problem in Example \ref{ex:rho_drho_AMD} only by the choice of admissible family of Hooke tensors: 
	\begin{equation}
	\label{eq:FibMD_setting}
	\Hs = \HM, \qquad j(\hk,\xi) =   \frac{1}{2} \pairing{\hk \,\xi,\xi}, \qquad \cost(\hk) = \tr\, \hk
	\end{equation}
	where $\Hax$ was defined in Example \ref{ex:Hs_Michell} as a closed, yet non-convex (for the case $d>1$) cone $\Hax$ of uni-axial Hooke tensors $a \ \eta \otimes \eta \otimes \eta \otimes \eta $ with $a \geq 0$ and $\eta\in S^{d-1}$. We first observe that for each $\hk \in \Hax$ with $\cost(\hk) \leq 1$, i.e. with $\tr \, \hk = a \leq 1$, there holds 
	\begin{equation}
	\label{eq:est_uniaxial}
	j(\hk,\xi) = \frac{1}{2} \pairing{\hk \,\xi,\xi} = \frac{a}{2} \, \bigl(\pairing{\xi,\eta \otimes \eta} \bigr)^2 \leq \frac{1}{2} \,\left(\max\limits_{i \in \{1,\ldots,d\}} \abs{\lambda_i(\xi)} \right)^2
	\end{equation}
	and at the same time
	\begin{equation}
	\label{eq:max_uniaxial}
	j(\bar{\hk}_\xi,\xi) = \frac{1}{2} \,\left(\max\limits_{i \in \{1,\ldots,d\}} \abs{\lambda_i(\xi)} \right)^2 \qquad \text{for} \qquad  \bar{\hk}_\xi = \bar{v}(\xi) \otimes\bar{v}(\xi) \otimes \bar{v}(\xi) \otimes \bar{v}(\xi)
	\end{equation}
	where $\bar{v}(\xi)$ is any unit eigenvector of $\xi$ corresponding to an eigenvalue of maximal absolute value.  
	
	Let us now take any $\tilde{\hk} \in \mathrm{conv} \bigl(\Hax \bigr)$, namely, since $\Hax$ is a cone, $\tilde\hk = \sum_{i=1}^m \alpha_i \hk_i$ for some $\alpha_i \geq 0$ and $\hk_i \in \Hax$ with $\cost(\hk_i) > 0$. Since both $\cost = \tr$ and $j(\argu,\xi)$ are linear there holds $\cost(\tilde{\hk}) = \nolinebreak \sum_{i=1}^m \alpha_i\, \cost\left(\hk_i\right)$ and thus
	\begin{alignat*}{1}
	j(\tilde{\hk},\xi) = \sum_{i=1}^m \alpha_i\, j(\hk_i,\xi) =  \sum_{i=1}^m \alpha_i \, \cost(\hk_i)\ j\left(\frac{\hk_i}{\cost(\hk_i)},\xi\right) &\leq \biggl(\sup\limits_{\substack{\hk \in \Hax \\ \cost(\hk)\leq 1 }} j(\hk,\xi) \biggr) \sum_{i=1}^m \alpha_i \, \cost(\hk_i) \nonumber \\ 
	& = \biggl(\sup\limits_{\substack{\hk \in \Hax \\ \cost(\hk)\leq 1 }} j(\hk,\xi) \biggr)\ \cost(\tilde{\hk}).
	\end{alignat*}
	By recalling \eqref{eq:est_uniaxial} and \eqref{eq:max_uniaxial} we arrive at
	\begin{equation}
	\label{eq:Hax_instead_of_HM}
	\jh(\xi) = \max\limits_{\substack{\hk \in \HM \\ \cost(\hk)\leq 1 }} j(\hk,\xi) = \max\limits_{\substack{\hk \in \Hax \\ \cost(\hk)\leq 1 }} j(\hk,\xi) = \frac{1}{2} \,\left(\max\limits_{i \in \{1,\ldots,d\}} \abs{\lambda_i(\xi)} \right)^2,
	\end{equation}
	where the first equality is by definition of $\jh$; moreover
	\begin{equation}
	\label{eq:Hch_Michell}
	\Hch(\xi) = \mathrm{conv} \biggl\{\bar{v}(\xi) \otimes\bar{v}(\xi) \otimes \bar{v}(\xi) \otimes \bar{v}(\xi)  \ : \bar{v}(\xi) \text{ is an eigenvector } v_i(\xi) \text{ with maximal } \abs{\lambda_i(\xi)}   \biggr\}.
	\end{equation}
	As a consequence $\ro$ becomes the spectral norm on the space of symmetric matrices $\Sdd$; we display it next to the well-established formula for its polar:
	\begin{equation}
		\label{eq:spectral_rho}
		\ro(\xi) = \max\limits_{i \in \{1,\ldots,d\}} \abs{\lambda_i(\xi)}, \qquad \dro(\sig) = \sum_{i =1}^d \abs{\lambda_i(\sig)}.
	\end{equation}
	The extremality condition for the pair $\rho,\dro$ may be characterized as follows
	\begin{equation}
	\label{eq:ext_cond_Michell}
	\pairing{\xi,\sig} = \ro(\xi) \, \dro(\sig) \qquad \Leftrightarrow \qquad
	\left\{
	\begin{array}{l}
	\text{every eigenvector of } \sig  \text{ is an eigenvector of } \xi \text{ and }\\
	\lambda_i(\sig) \neq 0 \quad \Rightarrow \quad  \lambda_i(\xi) = \sign\bigl(\lambda_i(\sig)\bigr)\, \ro(\xi).
	\end{array}
	\right.	   
	\end{equation}
	
	It is thus only left to characterize the set $\Hcc(\sig)$; we see that this time around we are forced to search the set $\Hs = \HM$, instead of just $\Hax$, which was the case while maximizing $j(\argu,\xi)$ (see \eqref{eq:Hax_instead_of_HM}): indeed, any $\sig$ of at least two non-zero eigenvalues yields $j^*(\hk,\sig)=\infty$ for each $\hk \in \Hax$. According to point (ii) of Theorem \ref{thm:rho_drho} for a given non-zero $\sigma$ the Hooke tensor $\hk \in \Hc$ is an element of $\Hcc(\sig)$ if and only if the constitutive law \eqref{eq:extreme_const_law} holds for any $\xi = \xi_\sig$ that satisfies: $\rho(\xi_\sig)=1$ and the extremal relation \eqref{eq:ext_cond_Michell} with $\sig$. Since the function $j$ was chosen as quadratic form (see \eqref{eq:FibMD_setting}) the constitutive law reads
	\begin{equation}
		\label{eq:criteria_for_optimality_of_H}
		\frac{\sig}{\dro(\sig)} = H \, \xi_\sig.
	\end{equation}
	With $v_i(\sig)$ denoting unit eigenvectors of $\sig$  for a non-zero stress $\sig$ we propose the Hooke tensor
	\begin{equation}
		\label{eq:optimal_H_for_Michell}
		\bar{\hk}_\sig = \sum_{i=1}^{d} \frac{\abs{\lambda_i(\sig)}}{\dro(\sig)} \ v_i(\sig) \otimes v_i(\sig) \otimes v_i(\sig) \otimes v_i(\sig) 
	\end{equation}
	that is an element of $\Hc$, i.e. $\bar{\hk}_\sig \in \HM$ and $\tr \, \bar{\hk}_\sig = 1$. Since the pair $\xi_\sig,\sig$ satisfies \eqref{eq:ext_cond_Michell} each $v_i(\sig)$ is an eigenvector for $\xi_\sig$ and moreover $\pairing{\xi_\sig, v_i(\sig) \otimes v_i(\sig) } = \sign{\bigl(\lambda_i(\sig)\bigr)}$, therefore
	\begin{equation*}
		 \bar{\hk}_\sig  \, \xi_\sig = \sum_{i=1}^{d} \frac{\abs{\lambda_i(\sig)}}{\dro(\sig)} \, \sign{\bigl(\lambda_i(\sig)\bigr)} \ v_i(\sig) \otimes v_i(\sig) = \frac{\sig}{\dro(\sig)},
	\end{equation*}
	which proves that $\bar{\hk}_\sig  \in \Hcc(\sig)$. The full characterization of the set $\Hcc(\sig)$ is difficult to write down for arbitrary $d$ hence further we shall proceed in dimension $d=2$, where three cases must be examined:
	
	\noindent\underline{Case a) the determinant of $\sig$ is negative}
	
	In this case $\sig$ has two non-zero eigenvalues of opposite sign, let us say: $\lambda_1(\sig)<0$ and $\lambda_2(\sig)>0$. Therefore there exists a unique $\xi = \xi_\sig$ that satisfies $\rho(\xi_\sig) \leq 1$ and is in the extremal relation \eqref{eq:ext_cond_Michell} with $\sig$: there must hold $\xi_\sig = - v_1(\sig) \otimes v_1(\sig) + v_2(\sig) \otimes v_2(\sig)$ where $v_1(\sig),v_2(\sig)$ are the respective eigenvectors of $\sigma$. According to point (iii) of Theorem \ref{thm:rho_drho} there must hold $\Hcc(\sig) \subset \Hch(\xi_\sig)$ and thus from \eqref{eq:Hch_Michell} we deduce that each $\hk\in \Hch(\sig)$ satisfies $H = \sum_{i=1}^2 \alpha_i \,v_i(\sig) \otimes v_i(\sig) \otimes v_i(\sig) \otimes v_i(\sig)$ for $\alpha_1+\alpha_2 = 1$. Then the constitutive law \eqref{eq:criteria_for_optimality_of_H} enforces $\sig/\dro(\sig) = -\alpha_1\, v_1(\sig) \otimes v_1(\sig) + \alpha_2\, v_2(\sig) \otimes v_2(\sig)$ and we immediately obtain that $\alpha_i = \abs{\lambda_i(\sig)}/\dro(\sig)$ and therefore $\hk$ must coincide with $\bar{\hk}_\sig$ from \eqref{eq:optimal_H_for_Michell}. In summary, in the case when $d = 2$ and $\mathrm{det}\,\sig <0$ the set $\Hcc(\sig)$ is a singleton:
	\begin{equation}
		\label{eq:Hcc_char_Michell_negative}
		\Hcc(\sig) =  \biggl\{\sum_{i=1}^{2} \frac{\abs{\lambda_i(\sig)}}{\dro(\sig)} \ v_i(\sig) \otimes v_i(\sig) \otimes v_i(\sig) \otimes v_i(\sig) \biggr\},
	\end{equation} 
	while $\Hch(\xi_\sig)$ is the convex hull of $\bigl\{  v_i(\sig) \otimes v_i(\sig) \otimes v_i(\sig) \otimes v_i(\sig) \,:\, i\in \{1,2\} \bigr\}$.
	
	\noindent\underline{Case b) the determinant of $\sig$ is positive}
	
	Without loss of generality we may assume that $\lambda_1(\sig),\lambda_2(\sig) > 0$. Once again there is unique $\xi_\sig$ with $\rho(\xi_\sig) = 1$ and satisfying \eqref{eq:ext_cond_Michell}: necessarily $\xi_\sig = \mathrm{I}$. Therefore any unit vector $\eta$ is an eigenvector of $\xi_\sig$ (but not necessarily of $\sig$) with eigenvalue equal to one and thus $\Hch(\xi_\sig) = \mathrm{conv}\bigl\{ \eta \otimes \eta \otimes \eta \otimes \eta \,:\, \eta \in S^{d-1} \bigr\}$. Therefore the inclusion $\Hcc(\sig) \subset \Hch(\xi_\sig)$ merely indicates that for $H \in \Hcc(\sig)$ there must hold $H = \sum_{i=1}^m \alpha_i\, \eta_i \otimes \eta_i \otimes \eta_i \otimes \eta_i$ where $m \in \mathbbm{N}$, $\eta_i\in S^{d-1}$ and $\sum_{i=1}^m \alpha_i = 1$. By plugging this form of $H$ into \eqref{eq:criteria_for_optimality_of_H} we obtain the characterization for a positive definite $\sig$ (recall that $\pairing{\eta_i \otimes \eta_i,\xi_\sig}\!=\!1$ for each $i$)
	\begin{equation}
		\label{eq:Hcc_char_Michell_positive}
		\Hcc(\sig) = \biggl\{ \sum_{i=1}^m \alpha_i\, \eta_i \otimes \eta_i \otimes \eta_i \otimes \eta_i \ : \ \eta_i\in S^{d-1},\ \alpha_i\geq 0,\ \sum_{i=1}^m \alpha_i = 1, \ \frac{\sig}{\tr \,\sig}= \sum_{i=1}^m \alpha_i\, \eta_i \otimes \eta_i \biggr\},
 	\end{equation}
 	where we used the fact that $\dro(\sig) = \tr\,\sig$ for any positive semi-definite $\sig$.
 	
 	With the following example we show that optimal Hooke tensor for positive definite $\sig$ is highly non-unique and the characterization above cannot be sensibly simplified. With $e_1,e_2$ denoting a Cartesian base of $\Rd$ we consider $\sig = \frac{4}{5} \,e_1\otimes e_1+\frac{1}{5}\, e_2\otimes e_2$, we see that $\tr\,\sig = 1$. By \eqref{eq:Hcc_char_Michell_positive} it is clear that $H_1 =\frac{4}{5} \,e_1\otimes e_1\otimes e_1\otimes e_1\otimes e_1+\frac{1}{5}\, e_2\otimes e_2 \otimes e_2\otimes e_2\otimes e_2 $ is optimal for $\sig$ and it is the expected solution: it is the universally optimal tensor $\bar{\hk}_\sig$ given in \eqref{eq:optimal_H_for_Michell}. Next we choose non-orthogonal vectors $\eta_1 = \frac{2}{\sqrt{5}}\, e_1 + \frac{1}{\sqrt{5}}\, e_2$ and $\eta_2 = \frac{2}{\sqrt{5}}\, e_1 - \frac{1}{\sqrt{5}}\, e_2$ and we may check that the tensor $H_2 = \sum_{i=1}^2 \frac{1}{2} \,\eta_i\otimes \eta_i\otimes \eta_i\otimes \eta_i\otimes \eta_i$ is an element of $\Hcc(\sig)$ according to \eqref{eq:Hcc_char_Michell_positive}. Since $H_1 \neq H_2$ it becomes clear that elements of $\Hcc(\sig)$ for positive definite $\sig$ may be constructed in many ways.
 	
 	\noindent\underline{Case c) $\sig$ is of rank one}
 	 
 	It is not restrictive to assume that $\lambda_1(\sig) = 0$, $\lambda_2(\sig)>0$ and so $\sig = \lambda_2(\sig) \,v_2(\sig) \otimes v_2(\sig)$. In this case there are infinitely many $\xi_\sig$ such that $\rho(\xi_\sig) = 1$ and \eqref{eq:ext_cond_Michell} holds. We can, however, test \eqref{eq:criteria_for_optimality_of_H} with only one: $\xi_\sig := v_2(\sig) \otimes v_2(\sig)$ for which $\Hch(\xi_\sig) = \big\{ v_2(\sig) \otimes v_2(\sig) \otimes v_2(\sig) \otimes v_2(\sig) \bigr\}$ which is necessarily equal to $\Hcc(\sig)$ due to point (iii) of Theorem \ref{thm:rho_drho}. Eventually, for a rank-one stress $\sig$ the set of optimal Hooke tensors may be written as a singleton
 	\begin{equation}
 		\label{eq:Hcc_char_Michell_rank_one}
 		\Hcc(\sig) = \left\{ \frac{\sig}{\abs{\sig}} \otimes \frac{\sig}{\abs{\sig}}\right\},
 	\end{equation}
	where we used the fact that $\dro(\sig) = \abs{\sig}$ for $\sig$ of rank one. By comparing to Example \ref{ex:rho_drho_AMD} we learn that the AMD and FibMD problems furnish the same optimal Hooke tensor at points where $\sig$ is rank-one.

	\begin{remark}
	\label{rem:FibMD_Michell}
	The pair of variational problems $\Prob$ and $\dProb$ with $\ro$ and $\dro$ specified above are well known to constitute the \textit{Michell problem} which is the one of finding the least-weight  truss-resembling structure in $d$-dimensional domain $\Ob$, cf. \cite{strang1983} and \cite{bouchitte2008}. An extensive coverage of the Michell structures may be found in \cite{Lewinski2019}. Typically one poses the Michell problem in the so-called plastic design setting, namely the structure is not a body that undergoes elastic deformation, it is merely a body made of perfectly rigid-plastic material and is being designed to work under given stress regime. Herein the Michell problem is recovered as a special case of the Free Material Design problem for elastic body: we start with the set $\Hax$ of uni-axial Hooke tensors that is supposed to mimic the truss-like behaviour of the design structure. Mathematical argument requires that $\Hax$ be convexified to $\HM$ and eventually the optimal structure is made of a fibrous-like material. Another work where a link between the Michell problem and optimal design of elastic body was made is \cite{bourdin2008} where the Michell problem was recovered as the asymptotic limit for structural topology design problem in the high-porosity regime.
	\end{remark}
\end{example}

\begin{example}[\textbf{Fibrous Material Design problem with dissymmetry in tension and compression}]
	\label{ex:rho_drho_FibMD_plus_minus}
	We revisit the problem of Fibrous Material Design with the linear constitutive law replaced by the constitutive law for material that responds differently in tension and compression (the design problem will be further abbreviated by FibMD$\pm$), i.e. we take
	\begin{equation}
		\label{eq:FibMD_setting_plus_minus}
		\Hs = \HM, \qquad j_\pm(\hk,\xi) = (\kappa_+)^2 \, j_+(\hk,\xi) + (\kappa_-)^2 \, j_-(\hk,\xi), \qquad \cost(\hk) = \tr\, \hk
	\end{equation}
	where $j_+, j_-$ are computed for $j(\hk,\xi) = \frac{1}{2} \pairing{\hk\,\xi,\xi}$, i.e. $p=2$, see Example \ref{ex:dissymetru_tension_compresion}. In contrast to Example \ref{ex:rho_drho_FibMD} we have no linearity of $j_\pm$ with respect to $\hk$ and therefore for given $\xi \in \Sdd$ we must test $j_\pm(\hk,\xi)$ with tensors $\hk$ in the whole $\HM$ instead of just $\Hs_{axial}$. We start with a remark: for every $\xi \in \Sdd$ there exist $\zeta_1 \in \mathcal{S}^{d\times d}_+$ and $\zeta_2 \in \mathcal{S}^{d\times d}_-$ such that $\xi +\zeta_1 = \xi_+$ and $\xi +\zeta_2 = \xi_-$ where $\xi_+ = \sum_i \max\{\lambda_i(\xi),0\}\, v_i(\xi) \otimes v_i(\xi)$ and $\xi_- = \sum_i \min\{\lambda_i(\xi),0\}\, v_i(\xi) \otimes v_i(\xi)$ are, respectively, the positive and negative part of the tensor $\xi$. Then, for any $\hk \in \Hc$, i.e. for $\hk = \sum_{i =1}^m \alpha_i\,\eta_i \otimes \eta_i \otimes \eta_i \otimes \eta_i$ with $\sum_{i=1}^m \alpha_i=1$, we estimate
	\begin{equation*}
		j_+(H,\xi) = \inf\limits_{\zeta \in \mathcal{S}^{d\times d}_+} \biggl\{ \sum_{i =1}^m \frac{\alpha_i}{2} \bigl(\pairing{\xi + \zeta,\eta_i \otimes \eta_i}\bigr)^2 \biggr\} \leq \sum_{i =1}^m \frac{\alpha_i}{2} \bigl(\pairing{\xi_+,\eta_i \otimes \eta_i}\bigr)^2
	\end{equation*}
	and by repeating an analogous estimate for $j_-$ we obtain
	\begin{alignat*}{1}
		j_\pm(H,\xi) \leq \sum_{i =1}^m \frac{\alpha_i}{2} \biggl(\bigl(\kappa_+\pairing{\xi_+,\eta_i \otimes \eta_i}\bigr)^2 + \bigl( \kappa_- \pairing{\xi_-,\eta_i \otimes \eta_i}\bigr)^2 \biggr) \leq \sum_{i =1}^m \frac{\alpha_i}{2} \biggl(\pairing{\bigl(\kappa_+\xi_+ - \kappa_- \xi_-\bigr),\eta_i \otimes \eta_i} \biggr)^2
	\end{alignat*}
	where we used the fact that $\bigl(\kappa_+\pairing{\xi_+,\eta_i \otimes \eta_i}\bigr) \, \bigl( \kappa_- \pairing{\xi_-,\eta_i \otimes \eta_i}\bigr) \leq 0$. We see that for any $\hk\in \Hc$ we have
	\begin{equation*}
		j_\pm(H,\xi) \leq \sum_{i =1}^m \frac{\alpha_i}{2} \bigl( \rho_\pm(\xi)\bigr)^2 \leq \frac{1}{2} \bigl( \rho_\pm(\xi)\bigr)^2
	\end{equation*}
	where, upon denoting by $\rho$ the spectral norm from \eqref{eq:spectral_rho}, we introduce
	\begin{equation*}
		\rho_\pm(\xi) := \rho\bigl(\kappa_+\,\xi_+ - \kappa_- \,\xi_-\bigr) =  \max\limits_{i \in \{1,\ldots,d\}} \biggl\{ \max \bigl\{ \kappa_+ \lambda_i(\xi),-\kappa_- \lambda_i(\xi) \bigr\} \biggr\}.
	\end{equation*}
	By choosing $\eta$ parallel to a suitable eigenvector of $\xi$ we easily obtain $\jh_\pm(\xi) \geq j_\pm(\eta\otimes \eta\otimes \eta\otimes \eta,\xi) = \frac{1}{2} \bigl( \rho_\pm(\xi)\bigr)^2$. The two estimates furnish $\jh_\pm(\xi) = \frac{1}{2} \bigl( \rho_\pm(\xi)\bigr)^2$ hence $\rho_\pm$ is the gauge function for the FibMD$\pm$ problem. We observe that
	\begin{equation*}
		\rho_\pm(\xi) \leq 1 \qquad \Leftrightarrow \qquad -\frac{1}{\kappa_-} \leq \lambda_i(\xi) \leq \frac{1}{\kappa_+} \quad \forall\, i \in \{1,\ldots,d\}.
	\end{equation*}
	For $\sig \in \Sdd$ the polar $\rho_\pm^0$ reads
	\begin{equation*}
		\rho_\pm^0(\sig) = \sum_{i=1}^{d} \max \left\{ \frac{1}{\kappa_+} \lambda_i(\sig),-\frac{1}{\kappa_-} \lambda_i(\sig) \right\}= \frac{1}{2} \left( \frac{1}{\kappa_+} - \frac{1}{\kappa_-}\right) \tr\,\sig + \frac{1}{2} \left( \frac{1}{\kappa_+} + \frac{1}{\kappa_-}\right) \dro(\sig)
	\end{equation*}
	where $\dro$ is the polar to the spectral norm, see \eqref{eq:spectral_rho}; it is worth to note that $\tr\, \sig$ enters the formula with a sign. The formula for $\rho_\pm^0$ was already reported in Section 3.5 in \cite{Lewinski2019}. The extremality conditions between $\xi$ and $\sig$ for $\rho_\pm$ and $\rho_\pm^0$ are very similar to those displayed for FibMD problem (see \eqref{eq:ext_cond_Michell}) thus we shall neglect to write them down. The same goes for characterizations of the sets $\Hch(\xi)$ and $\Hcc(\sig)$; we merely show a formula for
	\begin{equation*}
		\bar{\hk}_\sig = \sum_{i=1}^{d} \frac{\max \bigl\{ \frac{1}{\kappa_+} \lambda_i(\sig),-\frac{1}{\kappa_-} \lambda_i(\sig) \bigr\}}{\rho_\pm^0(\sig)} \ v_i(\sig) \otimes v_i(\sig) \otimes v_i(\sig) \otimes v_i(\sig) 
	\end{equation*}
	being a universal (but in general non-unique) element of the set $\Hcc(\sig)$.
	
	\begin{remark}
		In Remark \ref{rem:FibMD_Michell} the FibMD problem, characterized by the spectral norm $\rho$ and its polar $\dro$ and posed for an elastic body, was recognized it as equivalent to the Michell problem of designing a truss-like plastic structure of minimum weight -- this observation was valid under the condition that the permissible stresses in the second model are equal in tension and compression, i.e. $\bar{\sigma}_+ = \bar{\sigma}_- \in \R_+$. The theory of plastic Michell structures is developed in the case $\bar{\sigma}_+ \neq \bar{\sigma}_-$ as well, see Section 3.4 in \cite{Lewinski2019}. If one chooses $\kappa_+/\kappa_- = \bar{\sigma}_+ / \bar{\sigma}_-$ in the FibMD$\pm$ problem then again the duality pair $\Prob$, $\dProb$ with gauges $\rho_\pm$, $\rho_\pm^0$ is the very same as the one appearing in the Michell problem with permissible stresses $\bar{\sigma}_+ \neq \bar{\sigma}_-$. To the knowledge of the present authors the FibMD$\pm$ problem is the first formulation for elastic structure design known in the literature that is directly linked to the Michell problem for uneven permissible stresses in tension and compression.
	\end{remark}

\end{example}

\begin{example}[\textbf{Isotropic Material Design problem}]
	\label{ex:rho_drho_IMD}
	The following variant of (FMD) problem is known as the \textit{Isotropic Material Design} problem (IMD), see \cite{czarnecki2015a}:
	\begin{equation}
	\label{eq:IMD_setting}
	\Hs = \Hs_{iso}, \qquad j(\hk,\xi) =   \frac{1}{2} \pairing{\hk \,\xi,\xi}, \qquad \cost(\hk) = \tr\, \hk,
	\end{equation}
	where $\Hs_{iso} = \bigl\{d  K \bigl( \frac{1}{d}\, \mathrm{I} \otimes \mathrm{I} \bigr) + 2\, G\, \bigl( \mathrm{Id}- \frac{1}{d}\, \mathrm{I} \otimes \mathrm{I} \bigr)\, : \, K,G\geq 0  \bigr\}$ is a two-dimensional closed convex cone of isotropic Hooke tensors in a $d$-dimensional body, $d\in \{2,3\}$. For any $\hk \in \Hs_{iso}$ and $\xi \in \Sdd$ we have $j(H,\xi) = \frac{1}{2}\,\bigl( K \abs{\tr\,\xi}^2 +  2G\, \abs{\mathrm{dev} \,\xi}^2 \bigr)$ where $\mathrm{dev} \,\xi = \xi - \frac{1}{d} \, (\tr\, \xi)\, \mathrm{I} = \bigl( \mathrm{Id}- \frac{1}{d}\, \mathrm{I} \otimes \mathrm{I} \bigr) \,\xi$ and $\abs{\mathrm{dev}\,\xi}$ denotes the Euclidean norm. It is well established that $\hk$ has a single eigenvalue $d K$ and $N(d)-1$ eigenvalues $2G$ (we recall that $N(d)=d\,(d+1)/2$) therefore $\tr\,\hk = d K + \bigl(N(d)-1\bigr)\,2G$. Upon introducing auxiliary variables $A_1 = d K$ and $A_2 = \bigl(N(d)-1\bigr)\,2G$ we obtain $\tr\,\hk = A_1+A_2$ and
	\begin{equation*}
		j(\hk,\xi) = \frac{1}{2} \left( A_1  \biggl(\frac{\abs{\tr \,\xi}}{\sqrt{d}} \biggr)^2 + A_2  \biggl(\frac{\abs{\mathrm{dev} \,\xi}}{\sqrt{N(d)-1}} \biggr)^2\right).
	\end{equation*} 
	Thus we have
	\begin{equation*}
		\jh(\xi) = \max_{\hk \in \Hc} j(\hk,\xi) = \max_{A_1,A_2 \geq 0} \bigl\{j(\hk,\xi) \, :\, A_1+A_2 \leq 1 \bigr\} = \frac{1}{2} \bigl( \rho(\xi)\bigr)^2 
	\end{equation*}
	where
	\begin{equation}
		\label{eq:rho_IMD}
		\rho(\xi) = \max\left\{ \frac{\abs{\tr \,\xi}}{\sqrt{d}}, \frac{\abs{\mathrm{dev} \,\xi}}{\sqrt{N(d)-1}} \right\},
	\end{equation}
	while, for a non-zero $\xi \in \Sdd$
	\begin{equation*}
		\Hch(\xi) \!= \!\biggl\{ \hk\in \Hs_{iso}  :  d K+ \bigl(N(d)-1\bigr)2G = 1,\, \biggl(\frac{\abs{\tr \,\xi}}{\sqrt{d}}-\rho(\xi)\biggr) K =0,\ \biggl(\frac{\abs{\mathrm{dev} \,\xi}}{\sqrt{N(d)-1}} - \rho(\xi) \biggr) G=0  \biggr\}.
	\end{equation*}
	By using the fact that $\pairing{\xi,\sig} = \frac{1}{d} (\tr\,\xi)(\tr\,\sig) + \pairing{\mathrm{dev}\,\xi,\mathrm{dev}\,\sig}$ we arrive at the polar
	\begin{equation}
		\label{eq:drho_IMD}
		\dro(\sig) = \frac{1}{\sqrt{d}} \, \abs{\tr\,\sig} + \sqrt{N(d)-1}\, \abs{\mathrm{dev}\,\sig}
	\end{equation}
	and the extremality conditions for non-zero $\xi,\sig$ follow:
	\begin{equation}
		\label{eq:ext_cond_IMD}
		\pairing{\xi,\sig} = \ro(\xi) \, \dro(\sig) \qquad \Leftrightarrow \qquad
		\left\{
		\begin{array}{l}
		\tr\, \sig = \abs{\tr\,\sig}\, \frac{\tr\,\xi}{\sqrt{d} \,\rho(\xi)}, \\
		\mathrm{dev}\,\sig = \abs{\mathrm{dev}\,\sig}\, \frac{\mathrm{dev}\,\xi}{\sqrt{N(d)-1} \,\rho(\xi)}.
		\end{array}
	\right.	   
	\end{equation}
	In order to  characterize optimal Hooke tensors for non-zero $\sig$ we use point (ii) of Theorem \ref{thm:rho_drho}: $\hk$ is an element of $\Hcc(\sig)$ if and only if, for any $\xi = \xi_\sig$ satisfying $\rho(\xi_\sig)=1$ and the extremality conditions above, the constitutive law $\sig/\dro(\sig) = H\,\xi_\sig$ holds, which, considering \eqref{eq:ext_cond_IMD}, may be rewritten as:
	\begin{equation*}
		\frac{1}{\dro(\sig)} \left( \frac{1}{d}\,\abs{\tr\,\sig} \,\frac{\tr\,\xi_\sig}{\sqrt{d}} \, \mathrm{I} + \abs{\mathrm{dev}\,\sig} \, \frac{\mathrm{dev}\,\xi_\sig}{\sqrt{N(d)-1}}   \right) = K (\tr\,\xi_\sig) \, \mathrm{I} + 2G \, \mathrm{dev}\,\xi_\sig.
	\end{equation*}
	It is easy to see that for any $\sig$ the tensor $\xi_\sig$ may be chosen so that both $\tr\,\xi_\sig \neq 0$ and $\mathrm{dev}\,\xi_\sig \neq 0$  and then comparing the left and right-hand side above yields
	\begin{equation}
		\label{eq:Hcc_char_IMD}
		\Hcc(\sig) = \biggl\{ \hk\in \Hs_{iso} \ : \ K = \frac{1}{d\sqrt{d}}\, \frac{\abs{\tr\,\sig}}{\dro(\sig)}, \ G = \frac{1}{2 \sqrt{N(d)-1}}\, \frac{\abs{\mathrm{dev}\,\sig}}{\dro(\sig)}  \biggr\}.
	\end{equation}
	We notice that $\Hcc(\sig)$ is always a singleton for non-zero $\sig$, while $\Hch(\xi)$ may be a one dimensional affine subset of $\Hs_{iso}$, provided $\abs{\tr\,\xi}/\sqrt{d} = \abs{\mathrm{dev}\,\xi} /\sqrt{N(d)-1} = \rho(\xi) \neq 0$.
	
	\begin{example}[\textbf{Isotropic Material Design in the case of the power-law }]
		\label{ex:power_law}
		For $p\in (1,\infty)$ different than 2 one may propose a generalization of the constitutive law of linear elasticity. The conditions (H1)-(H5) can be easily satisfied if one assumes admissible Hooke tensors to be isotropic, i.e. again $\hk \in \Hs = \Hs_{iso}$. For instance we may choose
		\begin{equation*}
			j(\hk,\xi) = \frac{1}{p}\,\biggl( K \abs{\tr\,\xi}^p +  2G\, \abs{\mathrm{dev} \,\xi}^p \biggr)
		\end{equation*}
		where once more the moduli $K,G \geq 0$ identify an isotropic Hooke tensor $\hk \in \Hs_{iso}$ by means of \eqref{eq:iso_K_G}. A similar potential is proposed in \cite{castaneda1998} and referred to as the \textit{power-law} potential. The authors therein, however, allow to choose different exponents for the two tensor invariants $\abs{\tr\,\xi}$ and  $\abs{\mathrm{dev} \,\xi}$, whilst here the assumption (H\ref{as:p-hom}) obligates us to apply a common exponent $p$. Naturally the results from Example \ref{ex:rho_drho_IMD} hold here with only slight modifications, for instance 
		\begin{equation*}
			\rho(\xi) = \max\left\{ \frac{\abs{\tr \,\xi}}{d^{\,1/p}} , \frac{\abs{\mathrm{dev} \,\xi}}{\bigl(N(d)-1\bigr)^{1/p}}\right\}, \qquad 	\dro(\sig) = \frac{1}{d^{\,1/p'}} \, \abs{\tr\,\sig} + \bigl(N(d)-1\bigr)^{1/p}\, \abs{\mathrm{dev}\,\sig}.
		\end{equation*}
	\end{example} 
	
\end{example}

\begin{remark}
	The \textit{Cubic Material Design} problem (CMD) considered in the paper \cite{czubacki2015} \textit{a priori} lies outside the scope of the present contribution. The set $\Hs_{cubic}$ of all the Hooke tensors of cubic symmetry is not a convex set, which is due to distinction of anisotropy directions. Thus $\Hs_{cubic}$ cannot be directly chosen as $\Hs$ herein. Nevertheless, in case when $j(\hk,\xi) = \frac{1}{2} \pairing{\hk \,\xi,\xi}$ and $\cost(\hk) = \tr \hk$, it turns out that the set of solutions of the problem $\max_{\hk \in \Hs_{cubic}, \ \cost(\hk) \leq 1} j(\hk,\xi)$ is convex for any $\xi \in \Sdd$, see \cite{czubacki2015} for details. This implies that the original CMD problem can be recovered as a special case of the (FMD) problem provided we set $\Hs = \mathrm{conv} \bigl(\Hs_{cubic}\bigr)$. We shall not formulate this result rigorously herein.
\end{remark}

\subsection{Examples of solutions of the  Free Material Design problem in settings: AMD, FibMD, FibMD$\pm$ and IMD}

For one load case $\Fl$ that simulates the uni-axial tension we are to solve a family of Free Material Design problems in several settings listed in this paper. Thanks to Theorem \ref{thm:FMD_LCP} we may solve the corresponding (LCP) problem instead, for which we have at our disposal the optimality conditions from Theorem \ref{thm:optimality_conditions}. Our strategy will be to first put forward a competitor $u,\mu,\sig,\hf$ for which we shall validate the optimality conditions. While solutions in Cases a) and b) are fairly easy to guess, it is clear that solution (the exact coefficients) in Case c), i.e. for IMD problem, had to be derived first. We stress that displacement solutions $u$ are given up to a rigid body displacement function $u_0\in \U_0$. It is also worth explaining that the Hooke functions $\hf$ and their underlying moduli are given without physical units as they are normalized by the condition $\hf(x) \in \Hc$: one can see that the ultimate Hooke field is $\lambda = \hf \mu$ and the suitable units are included in the "elastic mass distribution" $\mu$; an analogous comment concerns the stress function $\sigma$.

\begin{example}[\textbf{Optimal material design of a plate under uni-axial tension test}]
	For a rectangle being a closed set $R = A_1 A_2 B_2 B_1 \subset \Rd = \R^2$ (we set $d=2$ the throughout this example) with $A_1 = (-a/2,-b/2)$, $A_2 = (-a/2,b/2)$, $B_1 = (a/2,-b/2)$ and $B_2 = (a/2,b/2)$ we consider a load 
	\begin{equation*}
		\Fl = F_q +F_Q, \qquad F_q=- q\,e_1\, \Ha^1\mres[A_1,A_2] +  q\,e_1\, \Ha^1\mres[B_1,B_2], \quad  F_Q=- Q\, e_1 \, \delta_{A_0} + Q \, e_1\, \delta_{B_0}
	\end{equation*}
	where $A_0 = (-a/2,0)$, $B_0 = (a/2,0)$ and $q$ and $Q$ are non-negative constants that represent, respectively, loads diffused along segments and point loads, see Fig. \ref{fig:optimal_structure}. It is straightforward to check that $\Fl$ is balanced. For $\Omega$ we can take any bounded domain such that $R \subset \Ob$.
	\begin{figure}[h]
		\centering
		\includegraphics*[width=0.65\textwidth]{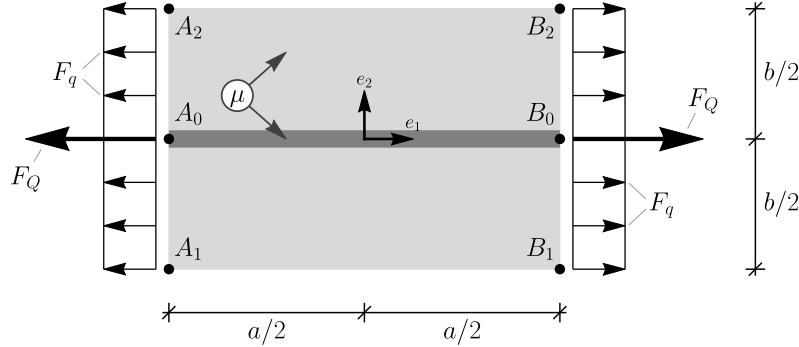}
		\caption{Graphical representation of load $\Fl = F_q +F_Q$ and optimal mass distribution $\mu$.}
		\label{fig:optimal_structure}       
	\end{figure}

	\noindent\underline{Case a) the Anisotropic Material Design}
	
	In the AMD setting where $\rho = \dro = \abs{\argu}$, see \eqref{eq:AMD_setting} in Example \ref{ex:rho_drho_AMD}, we propose the following quadruple
	\begin{alignat}{2}
		\label{eq:AMD_quadruple_1}
		u(x) &= x_1\,e_1, \qquad \qquad &\mu &= q \, \mathcal{L}^2 \mres R + Q \, \Ha^1 \mres [A_0,B_0],\\
		\label{eq:AMD_quadruple_2}
		\sig &= e_1 \otimes e_1, \qquad \qquad &\hf &= e_1 \otimes e_1 \otimes e_1 \otimes e_1.
	\end{alignat}
	We see that $\dro(\sig) =\abs{\sig} = 1$ and $\tr\,\hf = 1$, which are the initial assumptions in Theorem \ref{thm:optimality_conditions}. An elementary computation shows that $-\DIV(\sig \mu) = \Fl$, which gives the optimality condition (i) in Theorem \ref{thm:optimality_conditions}. The function $u$ is smooth and thus checking the condition $u\in \overline{\U}_1$ boils down to verifying  whether $\rho\bigl( e(u) \bigr) = \abs{e(u)} \leq 1$. We have $e(u) = e_1 \otimes e_1$ and clearly the optimality condition (ii) follows. Next we can choose which of the conditions (iii) or (iii)' in Theorem \ref{thm:optimality_conditions} we shall check. First we list essential elements of theory of space tangent to measure $\mu$ for $\mu$-a.e. $x$:
	\begin{equation*}
		\mathcal{S}_\mu(x) = \left\{
		\begin{array}{cl}
		\mathcal{S}^{2 \times 2}&  \text{for }  \mathcal{L}^2\text{-a.e. } x \in R,\\
		\mathrm{span}\,\{e_1 \otimes e_1\}  &  \text{for }  \Ha^1\text{-a.e. } x \in [A_0,B_0],
		\end{array}
		\right. \
			P_\mu(x) = \left\{
		\begin{array}{cl}
		\mathrm{I}&  \text{for }  \mathcal{L}^2\text{-a.e. } x \in R,\\
		e_1 \otimes e_1  &  \text{for }  \Ha^1\text{-a.e. } x \in [A_0,B_0],
		\end{array}
		\right.
	\end{equation*}
	see e.g. \cite{bouchitte2001}. Since $u$ is smooth we simply compute $e_\mu(u)(x) = P^{\top}_\mu(x)\, e(u)(x) \,P_\mu(x)$; having $e(u) = e_1 \otimes e_1$ we clearly obtain that $e_\mu(u) = e_1 \otimes e_1$ for $\mu$-a.e. $x$ as well. We check that $\pairing{\sig(x),e_\mu(u)(x)} = 1$ for $\mu$-a.e. $x$. In addition, since $\hf = \sig \otimes \sig$, we have $\hf \in \Hcc(\sig)$ (see \eqref{eq:Hcc_char_AMD}) and the last optimality condition (iii) follows. We have thus already proved that the quadruple $u,\mu,\sig,\hf$ is an optimal solution for the (LCP) problem and Theorem \ref{thm:FMD_LCP} furnishes a solution for the original Free Material Design problem in the AMD setting.
	
	For the sake of demonstration we will in addition check the condition (iii)' as well: to this purpose we must compute the formula for $j_\mu\bigl(\hf(x),\argu\bigr)$. For $\mathcal{L}^2$-a.e. $x\in R$ clearly $j_\mu\bigl(\hf(x),\xi \bigr) = j\bigl(\hf(x),\xi \bigr)$ since for such $x$ we have $\mathcal{S}_\mu^\perp(x) = \{0\}$. For $\Ha^1$-a.e. $x \in [A_0,B_0]$ we have $\pairing{e_1 \otimes e_1,\zeta} = 0$ whenever $\zeta \in \mathcal{S}_\mu^\perp(x)$ hence for any $\xi \in \mathcal{S}^{2 \times 2}$
	\begin{equation*}
		j_\mu\bigl(\hf(x),\xi \bigr) = \inf\limits_{\zeta \in \mathcal{S}^\perp_\mu(x)}  j\bigl(\hf(x),\xi +\zeta \bigr)  = \inf\limits_{\zeta \in \mathcal{S}^\perp_\mu(x)} \frac{1}{2} \bigl(\pairing{e_1 \otimes e_1,\xi +\zeta}\bigr)^2 = \frac{1}{2} \bigl(\pairing{e_1 \otimes e_1,\xi}\bigr)^2
	\end{equation*}
	and ultimately we obtain $j_\mu\bigl(\hf(x),\argu \bigr) = j\bigl(\hf(x),\argu \bigr)$ for $\mu$-a.e. $x$. Therefore, verifying the condition (iii)' boils down to checking if $\mu$-a.e. $\sig = \hf \, e_\mu(u)$ and this is straightforward.
	
	\noindent\underline{Case b) the Fibrous Material Design}
	
	In the case of Fibrous Material Design it is enough to shortly note that the quadruple $u,\mu,\sig,\hf$ proposed in Case a) is also optimal in the setting of FibMD problem: indeed, both $e(u)$ and $\sig$ are of rank one, thus spectral norm $\rho\bigl(e(u) \bigr)$ and its polar $\dro(\sig)$ (see \eqref{eq:spectral_rho}) coincide with $\abs{e(u)}$ and $\abs{\sig}$ respectively. Moreover, again for a rank-one field $\sig$, the sets $\Hcc(\sig)$ are identical for AMD and FibIMD, see \eqref{eq:Hcc_char_Michell_rank_one} and the comment below. An additional comment is that the field $\tilde{u} = \tilde{u}(x) = x_1\, e_1 + \beta(x_2) \, e_2$ will also be optimal for FibMD problem provided that $\beta:\R\rightarrow \R$ is 1-Lipschitz; note that this was not the case for AMD problem where $u$ was uniquely determined up to a rigid body displacement function. 
	
	Further, the same solution \eqref{eq:AMD_quadruple_1},\eqref{eq:AMD_quadruple_2} of (LCP) will be shared by the FibMD$\pm$ provided that one assumes $\kappa_+ = 1$. This is a consequence of $\sigma$ being positive definite $\mu$-a.e. 
	
	\noindent\underline{Case c) the Isotropic Material Design}
	
	For the IMD problem the norms $\rho$ and $\dro$ are given in \eqref{eq:rho_IMD} and \eqref{eq:drho_IMD} respectively. We put forward a quadruple that shall be checked for optimality in the IMD problem:
	\begin{alignat*}{2}
		u(x) &= \frac{2+\sqrt{2}}{2}\,x_1\,e_1 - \frac{2-\sqrt{2}}{2}\,x_2\,e_2, \qquad \qquad &\mu& = \frac{2+\sqrt{2}}{2}\ \biggl( q \, \mathcal{L}^2 \mres R + Q \, \Ha^1 \mres [A_0,B_0] \biggr),\\
		\sig &= \frac{2}{2+\sqrt{2}}\ e_1 \otimes e_1, \qquad\qquad &\hf& = 2K\,\biggl( \frac{1}{2}\,\mathrm{I} \otimes \mathrm{I}\biggr) + 2G\left(\mathrm{Id} - \frac{1}{2}\mathrm{I} \otimes \mathrm{I} \right)
	\end{alignat*}
	with
	\begin{equation}
		\label{eq:optimal_K_G}
		K = \frac{1}{2+2\sqrt{2}}, \qquad  G = \frac{1}{4+2\sqrt{2}}.
	\end{equation}
	First we check that $\tr \,\hf = 2K+2\cdot2G = 1$, thus $\hf \in \Hc$ as assumed in Theorem \ref{thm:optimality_conditions}. Since the force flux $\tau = \sig \mu$ is identical to the one from Case a) the optimality condition (i) in Theorem \ref{thm:optimality_conditions} clearly holds. The function $u$ is again smooth so we compute $e(u) = \frac{2+\sqrt{2}}{2}\,e_1 \otimes e_1 - \frac{2-\sqrt{2}}{2}\,e_2 \otimes e_2$ and
	\begin{equation*}
		\tr\bigl(e(u)\bigr) = \sqrt{2}, \quad \abs{\mathrm{dev}\bigl(e(u) \bigr)} = \abs{e_1 \otimes e_1 - e_2 \otimes e_2} = \sqrt{2}, \quad \tr\,\sig = \frac{2}{2+\sqrt{2}}, \quad  \abs{\mathrm{dev}\,\sig} = \frac{1}{1+\sqrt{2}}
	\end{equation*}
	yielding
	\begin{equation*}
		\rho\bigl( e(u) \bigr) = \max\left\{ \abs{\tr\bigl(e(u)\bigr)}/\sqrt{2}\,,\, \abs{\mathrm{dev}\bigl(e(u) \bigr)}/\sqrt{2} \right\} =1, \qquad \dro(\sig) = \frac{1}{\sqrt{2}} \, \abs{\tr\,\sig} + \sqrt{2}\, \abs{\mathrm{dev}\,\sig} = 1
	\end{equation*}
	and therefore $u \in \overline{\U}_1$, which validates the optimality condition (ii); moreover $\dro(\sig) = 1$ as required in Theorem \ref{thm:optimality_conditions}. We move on to check the last optimality condition in version (iii). Since $\mu$ above is coincides with the one from Case a) (up to multiplicative constant), the formulas for $\mathcal{S}_\mu$ and $P_\mu$ derived therein are also correct here. Due to smoothness of $u$ we have $\mu$-a.e. $\pairing{\sig,e_\mu(u)} = \pairing{\sig, P^{\top}_\mu\, e(u) \,P_\mu} = \pairing{\sig,e(u)}$, where we used the fact that $\sig \in \mathcal{S}_\mu\ $ $\mu$-a.e. We easily check that $\pairing{\sig,e(u)} = 1$ and the extremality condition $\pairing{\sig,e_\mu(u)} = 1$ follows. Then one may easily check that the moduli $K, G$ agree with the characterization of the set $\Hcc(\sig)$ in \eqref{eq:Hcc_char_IMD}, hence $\hf \in \Hcc(\sig)\ $ $\mu$-a.e. and the optimality condition (iii) follows proving that the quadruple $u,\mu,\sig,\hf$ is indeed optimal for (LCP) problem in the IMD setting.
	
	In order to be complete we will show that the optimality condition (iii)' holds as well. It is clear that for $\mathcal{L}^2$-a.e. $x \in R$, where $\mathcal{S}_\mu(x) = \mathcal{S}^{2 \times 2}$, we have $j_\mu\bigl(\hf(x), \argu \bigr) =j\bigl(\hf(x), \argu \bigr)$. Meanwhile for $\Ha^1$-a.e. $x \in [A_0,B_0]$ the tensors $\zeta \in \mathcal{S}^\perp_\mu(x)$ are exactly those of the form $\zeta = e_2 \diamond \eta$ where $\eta \in \R^2$ and $\diamond$ denotes the symmetrized tensor product. Hence, for $\Ha^1$-a.e. $x \in [A_0,B_0]$ after performing the minimization (being non-trivial here) we obtain
	\begin{equation}
		\label{eq:uniaxial_constitutive_law}
		j_\mu\bigl(\hf(x),\xi \bigr) = \inf\limits_{\eta \in \R^2}  j\bigl(\hf(x),\xi + e_2 \diamond \eta \bigr)  = \frac{1}{2} \frac{4 K G}{K+G} \ \pairing{e_1 \otimes e_1 \otimes e_1 \otimes e_1,\xi\otimes \xi},
	\end{equation}
	where the constant $\frac{4 K G}{K+G}$ can be readily recognized as Young modulus $E$, cf. \eqref{eq:Young_and_Poisson}. For chosen $\xi$ the minimizer $\eta = \eta_\xi$ above is exactly the one for which $\hf(x)\,(\xi + e_2 \diamond \eta_\xi) = s\,e_1\otimes e_1 $ for $s \in \R$. The potential $j_\mu$ in \eqref{eq:uniaxial_constitutive_law} induces the well-known uni-axial constitutive law in the bar $[A_0,B_0]$ that spontaneously emerges as a singular (with respect to $\mathcal{L}^2$) part of $\mu$. Upon computing: $e_\mu(u) (x)= e(u)(x)$ for $\mathcal{L}^2$-a.e. $x \in R$ and $e_\mu(u(x)) = \frac{2+\sqrt{2}}{2}\, e_1\otimes e_1$ for $\Ha^1$-a.e. $x\in [A_0,B_0]$, we see that eventually verifying condition (iii)' boils down to checking if
	\begin{equation*}
		\sig(x) = \left\{
		\begin{array}{cl}
		\hf(x) \, e(u)(x) &  \text{for }  \mathcal{L}^2\text{-a.e. } x \in R,\\
		\left(\frac{4 K G}{K+G} \ e_1 \otimes e_1 \otimes e_1 \otimes e_1 \right)\bigl(\frac{2+\sqrt{2}}{2} e_1\otimes e_1 \bigr) &  \text{for }  \Ha^1\text{-a.e. } x \in [A_0,B_0].
		\end{array}
		\right.
	\end{equation*}
	The equations above are verified after elementary computations; in particular using formulas \eqref{eq:optimal_K_G} for optimal $K,G$ gives the Young modulus and the Poisson ratio:
	\begin{equation}
		\label{eq:optimal_E_nu}
		E = \frac{4 K G}{K+G} = \left( \frac{2}{2+\sqrt{2}}\right)^2 = 6-4\sqrt{2}, \qquad \nu = \frac{K-G}{K+G}=3-2\sqrt{2}.
	\end{equation}
	
	We finish the example with an observation: the computed value of Young modulus $E$ turns out to be maximal among all pairs $K,G\geq 0$  satisfying $\tr \,\hf = 2K+2\cdot2G \leq 1$. This is not surprising, since the plate under tension test has minimum compliance whenever its relative elongation along direction $e_1$, which here equals $\check\eps := \pairing{\check{u}(a/2,0)-\check{u}(-a/2,0)\,,\,e_1/b}$, is minimal. It must be carefully noted that $\check{u}$ is a solution of (FMD) problem and not (LCP) problem, cf. Definitions \ref{def:LCP_solution}, \ref{def:FMD_solution} and Theorem \ref{thm:FMD_LCP}.
	For the Hooke law $\check\sig = \check\hf\,e(\check{u})$ with isotropic $\check\hf$ and $\check\sig = \check{s}\, e_1 \otimes e_1$ (representing uni-axial tensile stress)  it is well established that $\check\eps = \pairing{e(\check{u}),e_1 \otimes e_1} = \check{s} / \check{E}$. Since the stress coefficient $\check{s}$ is predetermined by the load we see that minimizing $\check{\eps}$ (or minimizing the compliance of the plate) reduces here to maximizing the Young modulus $\check{E}$. Since the cost assumed in the IMD problem was $c = \tr$ maximizing Young modulus is non-trivial and furnishes \eqref{eq:optimal_E_nu}, which includes the optimal Poisson ratio $\nu \cong 0.172$.
	
\end{example}

\section{The scalar settings of the Free Material Design problem}
\label{sec:outlook}

On many levels the presented paper has built upon the work \cite{bouchitte2007} on the optimal design of mass $\mu \in \Mes_+(\Ob)$. In some sense we have rigorously shown that the simultaneous design of the mass $\mu$ and the material's anisotropy described by Hooke tensor function $\hf \in L^\infty_\mu(\Ob;\Hs)$ consists of three steps:
\begin{enumerate}[(i)]
	\item computing functions $\rho=\rho(\xi)$ and $\dro = \dro(\sig)$ that, respectively, are maximum strain energy $j(\hk,\xi)$ and minimum stress energy $j^*(\hk,\sig) $ with respect to Hooke tensors $\hk \in \Hs$ of unit $c$-cost;
	\item finding the solutions $\hat{u}$ and $\hat{\tau}$ of the problems $\relProb$ and $\dProb$, formulated with the use of $\rho$ and $\dro$ respectively, and retrieving the optimal mass $\check{\mu} = \frac{\Totc}{Z} \dro(\hat{\tau})$;
	\item with $\check{\sig} = \frac{d \hat{\tau}}{d\check{\mu}} \in L^\infty_{\check{\mu}}(\Ob;\Sdd)$ finding point-wise the optimal Hooke tensor $\mathscr{C}(x) \in \Hc$ that for $\check{\mu}$-a.e. $x$ minimizes the stress energy $j^*\bigl(\mathscr{C}(x),\check{\sig}(x) \bigr)$, which may be done in a $\check{\mu}$-measurable fashion.
\end{enumerate}
The step (ii) alone is the essence of the approach for the optimal mass design presented in \cite{bouchitte2007}, where the functions $\rho$, $\dro$ are in fact data. At the same time it is the most difficult step here since the steps (i) and (iii) involve finite dimensional programming problems.

The present work concerns the problem of elasticity in two or three dimensional bodies, where the state function $u$ is vectorial and the differential operator is $e=e(u)$ being the symmetric part of the gradient. The framework of the paper \cite{bouchitte2007} is, however, far more general as \textit{a priori} it allows to choose any linear operator $A$, while the function $u$ may be either scalar or vectorial. The particular interest of the authors of \cite{bouchitte2007} is the case of $u:\Omega \rightarrow \R$ and $A = \nabla^2$ (the Hessian operator) that reflects the theory of elastic Kirchhoff plates (thin plates subject to bending). It appears that the theory of the Free Material Design problem herein developed is also easily transferable to problems other than classical elasticity and this last section shall serve as an outline of FMD theory in the context of two scalar problems: the aforementioned Kirchhoff plate problem and the stationary heat conductivity problem. 

\subsection{The Free Material Design problem for elastic Kirchhoff plates (second order scalar problem)}
\label{sec:FMD_for_plates}
For a plane bounded domain $\Omega\subset \R^2$ with Lipschitz boundary let there be given a first order distribution $f \in \D'(\R^2)$ with its support contained in $\Ob$. We assume that $f$ is balanced, i.e. $\pairing{u_0,f} = 0$ for any $u_0$ of the form $u_0(x) = \pairing{a,x}+b$ with $a \in \R^2$, $b\in \R$ ($u_0$ are functions of rigid plate out-of-plane displacements). With the cone of admissible Hooke tensors $\Hs$ and energy function $j:\Hs \times \Sdd \rightarrow \R$ defined as in Section \ref{sec:elasticity_problem}, for a Hooke tensor field given by a measure $\lambda \in \Mes(\Ob;\Hs)$ (the term \textit{bending stiffness field} would be more suited) we define compliance of an elastic Kirchhoff plate:
\begin{equation}
	\label{eq:compliance_def_plate}
	\Comp(\lambda) = \sup \left\{ f(u) - \int j\bigl(\lambda,\nabla^2 u\bigr) \ : \ u \in \D(\R^2) \right\},
\end{equation}
where the scalar function $u$ represents the plate deflection. With the compliance expressed as above the Free Material Design problem for Kirchhoff plates is formulated exactly as in the case of elasticity, i.e.  $\Cmin = \min \bigl\{ \Comp(\lambda) \ : \ \lambda \in \MesHH, \ \int \cost(\lambda) \leq \Totc \bigr\}$.

Since the elastic potential $j$ remains unchanged with respect to classical elasticity, the energy functional $J_\lambda$ from \eqref{eq:J_lambda} is identical as well and therefore Propositions \ref{prop:Carath}, \ref{prop:usc_j}, \ref{prop:usc_J} follow directly. Next it is straightforward to observe that in Theorem \ref{thm:problem_P} we do not utilize the structure of the operator $e$ and a counterpart of the result for operator $\nabla^2$ instead yields a pair of mutually dual problems:
\begin{alignat*}{2}
	&\relProb \qquad \qquad Z &&= \max \biggl\{ f(u) \ : \ u \in \overline\V_1 \biggr\} \qquad \qquad\\
	&\dProb\qquad \qquad &&=\min \biggl\{ \int \dro(\chi) \ : \ \chi \in \MesT, \ \DIV^2  \chi = f \biggr\}  \qquad \qquad
\end{alignat*}
where the functions $\rho$ and its polar $\dro$ are defined exactly as in Section \ref{sec:FMD_problem}, see \eqref{eq:jh_rho}. Above for the maximization problem we have already given its relaxed version where (see \cite{bouchitte2007} for details) $\overline\V_1$ is the closure of the set $\V_1 = \bigl\{ u \in \D(\R^2)\, :\, \rho(\nabla^2 u) \leq 1 \ \text{ in }\Omega \bigr\}$ in the norm topology of $C^1(\Ob)$. Proposition \nolinebreak 6 in \cite{bouchitte2007} offers a characterization
\begin{equation*}
	\overline\V_1 = \biggl\{ u \in W^{2,\infty}(\Omega) \ : \ \rho(\nabla^2 u) \leq 1 \ \text{ a.e. in } \Omega \biggr\},
\end{equation*}
which tells us that the problem $\relProb$ above admits a solution $\hat{u}$ whose first derivative is Lipschitz continuous (note that no analogous characterization was available for the elasticity case). The second order equilibrium equation $\DIV^2 \chi = f$ in $\dProb$ renders the tensor valued measure $\chi$ a \textit{bending moment field}. 

It is clear that Section \ref{sec:anisotropy_at_point} on the point-wise maximization and minimization of energy functions $j$ and $j^*$ respectively remains valid here since the definitions of $j$, $\Hs$ and $c$ did not change. Consequently Lemma \ref{lem:measurable_selection} on the existence of an optimal measurable Hooke tensor function $\hf$ still holds true. Eventually, with Linear Constrained Problem defined for the pair $\relProb$ and $\dProb$ above, the analogue of Theorem \ref{thm:FMD_LCP} paves a way to constructing the solution of the (FMD) problem for Kirchhoff plates based on the solution of (LCP). Thereby we have sketched how the Sections \ref{sec:FMD_problem}, \ref{sec:FMD_LCP} on the (FMD) theory for elasticity can be translated to the setting of Kirchhoff plates; of course the contribution \cite{bouchitte2007} played a key factor. The Section \ref{sec:optimality_conditions} on the optimality conditions could be adjusted as well, yet this would be more involved as it requires more insight on the theory of the $\mu$-intrinsic counterpart of the second order operator $\nabla^2$; the reader is referred to \cite{bouchitte2003} for details.

\subsection{The Free Material Design problem for heat conductor (first order scalar problem)}
\label{sec:FMD_for_heat_cond}

In this section $\Om \subset \Rd$ is any bounded domain in $d$-dimensional space ($d$ may equal 2 or 3) with Lipschitz boundary. The heat inflow and the heat outflow shall be given by two positive, mutually singular Radon measures $f_+\in \Mes_+(\Ob)$ and $f_- \in \Mes_+(\Ob)$ respectively; we assume the measures to be of equal mass: $f_+(\Ob) = f_-(\Ob)$.

Next, let $\mathscr{A}$ be a set of admissible conductivity tensors being any closed convex cone contained in the set of symmetric positive semi-definite 2nd-order tensors $\mathcal{S}^{d \times d}_+$. The constitutive law of conductivity will be determined by the energy $j_1:\mathscr{A} \times \Rd \rightarrow  \R$ that for some $p \in (1,\infty)$ meets assumptions analogous to (H1)-(H5) for function $j:\Hs \times \Sdd \rightarrow \R$. The compliance or the potential energy of the conductor given by a tensor-valued measure $\alpha \in \Mes(\Ob;\mathscr{A})$ may be defined as 
\begin{equation}
\label{eq:compliance_def_cond}
\Comp(\alpha) = \sup \left\{ \int u \, df - \int j_1\bigl(\alpha,\nabla u\bigr) \ : \ u \in \D(\R^d) \right\}
\end{equation}
where we put $f = f_+ - f_- \in \Mes(\Ob)$; the function $u$ plays a role of the temperature field. 

The Free Material Design problem for heat conductor may be readily posed:
\begin{equation}
\label{eq:FMD_heat_cond_def}
\FMD \qquad \quad \Cmin = \min \biggl\{ \Comp(\alpha) \ : \ \alpha \in \Mes(\Ob;\mathscr{A}), \ \int \cost_1(\alpha) \leq \Totc \biggr\} \qquad \quad
\end{equation}
where the cost function $c_1$ is the restriction to $\mathscr{A}$ of any norm on the space of symmetric tensors $\Sdd$; for instance $c_1$ may be taken as $\tr$. Below we shall also shortly use the name: \textit{the scalar $\FMD$ problem}.

Upon studying Sections \ref{sec:elasticity_problem}, \ref{sec:FMD_problem}, \ref{sec:FMD_LCP} we may observe that in the main results we did not make use of the structure of the space $\LSdd$ (being isomorphic to a subspace of 4-th order tensors) and neither of the fact that $\Hs$ contained positive semi-definite tensors only. In fact $\LSdd$ could be replaced by any finite dimensional linear space, while $\Hs$ by any convex closed cone $K \subset V$. In other words, the well-posedness of the $(\mathrm{FMD})$ problem stemmed from assumptions (H1)-(H5) alone and the set $\Hs \subset \LSdd$ was chosen merely to stay within natural framework of elasticity. The other choice could be precisely $V = \Sdd$ and $K = \mathscr{A}$. The argumentation for switching from vectorial $u$ to scalar one and from operator $e$ to $\nabla$ runs similarly to the one outlined for Kirchhoff plates and in addition it is again not an issue that the second argument of the function $j_1$ lies in $\Rd$ instead of $\Sdd$ in case of $j$: it could as well be any other finite dimensional linear space $W$. In summary, the conductivity framework presented above is well suited to the theory developed in this paper.

We are now in a position to quickly run through the main results for the scalar $(\mathrm{FMD})$ problem. We start by analogous definitions of mutually polar gauges $\rho_1,\rho_1^0 :  \Rd \rightarrow \R$: for any $v, q \in \Rd$
\begin{equation*}
\frac{1}{p}\bigl(\rho_1(v) \bigr)^p = \max\limits_{\substack{A\in \mathscr{A} \\ c_1(A)\leq 1 }} j_1(A,v), \qquad \frac{1}{p'}\bigl(\rho_1^0(q) \bigr)^{p'} = \min\limits_{\substack{A\in \mathscr{A} \\ c_1(A)\leq 1 }} j^*_1(A,q),
\end{equation*}
while by $\bar{\mathscr{A}}_1(v)$ and $\ubar{\mathscr{A}}_1(q)$ we will denote the sets of, respectively, maximizers and minimizers above. The counterpart of Theorem \ref{thm:problem_P} for the scalar case furnishes the pair of mutually dual problems:
\begin{alignat*}{2}
&\relProb \qquad \qquad Z &&= \max \biggl\{ \int u \, df \ : \ u \in W^{1,\infty}(\Om), \ \rho_1(\nabla u) \leq 1 \ \text{ a.e. in } \Om \biggr\} \qquad \qquad\\
&\dProb\qquad \qquad &&=\min \biggl\{ \int \rho_1^0(\vartheta) \ : \ \vartheta \in \Mes(\Ob;\Rd), \ -\DIV\,  \vartheta = f \biggr\},  \qquad \qquad
\end{alignat*}
where $\vartheta$ plays the role of the \textit{heat flux}. The problem $\relProb$ is already in its relaxed form, i.e. the set of admissible functions $u$ is the closure of the set $\bigl\{ u \in \D(\R) \, : \, \rho_1(\nabla u) \leq 1 \text{ in } \Om \bigr\}$ in the topology of uniform convergence. Recall that respective characterization via vector-valued functions $u \in W^{1,\infty}(\Om;\Rd)$ was not available for the $(\mathrm{FMD})$ problem in elasticity, see the comment below \eqref{eq:relProb}.

Theorem \ref{thm:FMD_LCP} adjusted for the scalar setting states that the conductivity tensor field $\check{\alpha} \in \nolinebreak \Mes(\Ob;\mathscr{A})$ solves the scalar $\FMD$ problem if and only if it is of the form
\begin{equation}
\label{eq:FMD_LCP_scalar}
\check{\alpha} = \check{A} \,\check\mu, \quad \check\mu = \frac{\Totc}{Z}\, \hat{\mu}, \quad \check{A} \in L^\infty_{\check{\mu}}(\Ob;\mathscr{A}) \text{ is any } \check{\mu} \text{-meas. selection of } x \mapsto \ubar{\mathscr{A}}_1\bigl( \hat{q}(x) \bigr)
\end{equation}
where $\hat{\mu} = \rho_1^0(\hat{\theta})$ and $\hat{q} = \frac{d \hat{\theta}}{d \hat{\mu}}$
for some solution $\hat\vartheta$ of the problem $\dProb$ above. Existence of the measurable selection referred to above follows from an adapted version of Lemma  \ref{lem:measurable_selection}.

In the sequel we shall consider the AMD version of the design problem along with the Fourier constitutive law, more precisely
\begin{equation}
\label{eq:scalar_AMD_setting}
\mathscr{A} = \mathcal{S}_+^{d \times d}, \qquad j_1(A,v)= \frac{1}{2} \pairing{A\,v , v}, \qquad c_1(A) = \tr\,A.
\end{equation}
Following the argument in Example \ref{ex:rho_drho_AMD} on the AMD setting in the case of elasticity, for non-zero vectors $v,q \in \Rd$ we arrive at
\begin{equation*}
\rho_1 = \rho_1^0 = \abs{\argu}, \qquad \bar{\mathscr{A}}_1(v) = \left\{ \frac{v}{\abs{v}} \otimes \frac{v}{\abs{v}}  \right\}, \qquad \ubar{\mathscr{A}}_1(q) = \left\{ \frac{q}{\abs{q}} \otimes \frac{q}{\abs{q}}  \right\}
\end{equation*}
with $\abs{\argu}$ being Euclidean norm on $\Rd$. Therefore, owing to \eqref{eq:FMD_LCP_scalar}, the tensor valued measure $\check{\alpha} \in \Mes(\Ob;\mathcal{S}^{d \times d}_+)$ is a solution of the scalar $\FMD$ problem in the AMD setting if and only if
\begin{equation}
\label{eq:alpha_theta}
\check{\alpha} = \frac{\Totc}{Z} \biggl(  \frac{d\hat{\vartheta}}{d\lvert\hat{\vartheta}\rvert} \otimes \frac{d\hat{\vartheta}}{d\lvert\hat{\vartheta}\rvert} \biggr)  \lvert\hat{\vartheta}\rvert
\end{equation} 
for some solution $\hat\vartheta$ of the problem $\dProb$ with $\rho_1^0 = \abs{\argu}$. An important feature of the optimal conductivity field readily follows:

\begin{proposition}
	Let $\check{\alpha}$ be any solution of the scalar $\FMD$ problem in the AMD setting \eqref{eq:scalar_AMD_setting}. Then $\check{\alpha}$ is rank-one, namely $\frac{d \check{\alpha}}{d\abs{\check{\alpha}}}$ is a rank-one matrix $\abs{\check{\alpha}}$-a.e.
\end{proposition}

\begin{remark}
	Up to a multiplicative constant, the only isotropic gauge function $\rho$ on $\Rd$ is the Euclidean norm $\abs{\argu}$. We thus find that every "isotropic scalar $(\mathrm{FMD})$ problem" reduces to the pair $\relProb$, $\dProb$ above with $\rho =b \abs{\argu}$, $b>0$ (note that no similar conclusion was true for the vectorial $(\mathrm{FMD})$ problem, see Examples \ref{ex:rho_drho_FibMD}, \ref{ex:rho_drho_FibMD_plus_minus}, \ref{ex:rho_drho_IMD} where all the gauges $\rho$ are isotropic). For instance in \eqref{eq:scalar_AMD_setting} we could instead  take $\mathscr{A} = \mathscr{A}_{iso}$, i.e. the set of all isotropic conductivity tensors while of course: $A \in \mathscr{A}_{iso}$ if and only if $A = a \,\mathrm{I}$, for $a \geq 0$. Then the scalar $(\mathrm{FMD})$ problem is equivalent to the Mass Optimization Problem from \cite{bouchitte2001} and, since in $\Rd$ space $c_1(\mathrm{I})=\tr \,\mathrm{I} = d$, we have for any $v\in \Rd$
	\begin{equation*}
	\frac{1}{2}\bigl(\rho_1(v)\bigr)^2 =  \max\limits_{\substack{A\in \mathscr{A}_{iso} \\ c_1(A)\leq 1 }} j_1(A,v) =  \max\limits_{\substack{a \geq 0 \\ d \cdot a \leq 1 }} \ \frac{1}{2}\pairing{a \, \mathrm{I}, v \otimes v} = \frac{1}{d} \ \frac{1}{2} \abs{v}^2
	\end{equation*}
	yielding $\rho_1(v) = \frac{1}{\sqrt{d}} \abs{v}$ and $\rho_1^0(q) = \sqrt{d}\, \abs{q}$. In dimension $d=2$, once $f$ charges the boundary $\partial\Om$ only, the problem $\dProb$ can be reformulated as the Least Gradient Problem, see \cite{gorny2017}. Upon acknowledging this equivalence, a study of $\dProb$ for anisotropic functions $\rho_1^0$ can be found in \cite{gorny2018planar}. 
\end{remark}

In the remainder of this section we shall assume that $\Om$ is convex, which, upon recalling that $\rho_1=\abs{\argu}$, allows to rewrite the condition $\rho_1(\nabla u) \leq 1$ a.e. in $\Om$ by a constraint on the Lipschitz constant: $\mathrm{lip}(u) \leq  1$. Then the Rubinstein-Kantorovich theorem combined with a duality argument allows to replace the problem $\dProb$ with the \textit{Optimal Transport Problem} $(\mathrm{OTP})$, see \cite{villani2003topics}: 
\begin{alignat*}{2}
Z &= \max \biggl\{ \int u \, d(f_+-f_-) \ : \ u \in C(\Ob), \ \ u(x)-u(y) \leq \abs{x-y} \ \  \forall\, (x,y)\in \Ob \times \Ob \biggr\}\\
&=\min \biggl\{ \int_{\Ob \times \Ob} \abs{x-y}\, \gamma(dx dy)  \ : \ \gamma \in \Mes_+(\Ob \times \Ob), \
\begin{array}{rl}
\pi_{\#,1} \gamma =& f_+,\\
\pi_{\#,2} \gamma =& f_- 
\end{array}
\biggr\}; \qquad \quad (\mathrm{OTP})
\end{alignat*}
in $(\mathrm{OTP})$ we enforce the left and the  right marginals of the \textit{transportation plan} $\gamma$ to be $f_+$ and $f_-$ respectively. Whenever $\hat\gamma$ is a solution of $(\mathrm{OTP})$ the measure $\hat{\vartheta}$ defined via acting on $v \in  C(\Ob;\Rd)$ by
\begin{equation}
\label{eq:theta_gamma}
\bigl(v;\hat{\vartheta} \bigr) := \int_{\Ob \times \Ob} \left( \int_{[x,y]} \pairing{v(z),\frac{x-y}{\abs{x-y}}} \, \Ha^1(dz) \right)  \hat\gamma(dx dy) 
\end{equation}
solves the problem $\dProb$ with $\rho_1^0=\abs{\argu}$. The passage from the problem $\Prob$ to the Optimal Transport Problem along with validation of the formula \eqref{eq:theta_gamma} may be found in \cite{bouchitte2001}. Upon plugging a solution $\hat\vartheta$ of the form \eqref{eq:theta_gamma} into formula \eqref{eq:alpha_theta}, however, it is not clear whether $\check{\alpha}$ enjoys a characterization of the type \eqref{eq:theta_gamma}. We conclude the paper with a result showing that there indeed exists an optimal tensor field $\check{\alpha}$ that decomposes to segments on which $\check{\alpha}$ is uni-axial (rank-one):
\begin{theorem}
	For a convex bounded design domain $\Om\subset \Rd$ let $\hat{\gamma} \in \Mes_+(\Ob \times \Ob)$ denote any solution of $(\mathrm{OTP})$; then the conductivity tensor field $\check{\alpha} \in \Mes(\Ob;\mathcal{S}^{d \times d}_+)$ defined as a linear functional for any $M \in C(\Ob;\Sdd)$ by
	\begin{equation}
	\label{eq:alpha_gamma}
	\bigl(M;\check{\alpha}\bigr) = \frac{\Totc}{Z}\, \int_{\Ob \times \Ob} \left( \int_{[x,y]} \pairing{M(z),\frac{x-y}{\abs{x-y}} \otimes \frac{x-y}{\abs{x-y}}} \, \Ha^1(dz) \right)  \hat\gamma(dx dy) 
	\end{equation}
	is a solution of the scalar $(\mathrm{FMD})$ problem in the AMD setting \eqref{eq:scalar_AMD_setting}.
\end{theorem}
\begin{proof}
	First we must verify whether $\check{\alpha}$ is a competitor for $(\mathrm{FMD})$. The functions on the space of symmetric tensors $g,g^0:\Sdd \rightarrow \R$ given by formulas $g(M) = \max_i \abs{\lambda_i(M)}$ and $g^0(A) = \sum_i \abs{\lambda_i(A)}$ ($\lambda_i$ are the eigenvalues) are mutually dual norms. We note that for any $A$ that is positive semi-definite we have $g^0(A) = \tr\,A = c_1(A)$. Based on the paper \cite{bouchitte1988} we discover
	\begin{align*}
	&\int c_1(\check{\alpha}) = \int g^0(\check{\alpha}) = \sup\biggl\{\int \pairing{M,\check{\alpha}} \ : \ M \in C(\Ob;\Sdd), \ g(M)\leq 1 \text{ in } \Omega   \biggr\} \\
	\leq&   \frac{\Totc}{Z}\, \int_{\Ob \times \Ob} \left( \int_{[x,y]} g^0 \!\left(\frac{x-y}{\abs{x-y}} \otimes \frac{x-y}{\abs{x-y}} \right) \, \Ha^1(dz) \right)  \hat\gamma(dx dy) =  \frac{\Totc}{Z}\,\int_{\Ob \times \Ob} \abs{x-y}\,  \hat\gamma(dx dy) = \Totc,
	\end{align*}
	where we used the fact that $\hat{\gamma}$ solves $(\mathrm{OTP})$; the feasibility of $\check{\alpha}$ in $(\mathrm{FMD})$ is validated.
	
	Let $\hat{\vartheta} \in \Mes(\Ob;\Rd)$ be given by the formula \eqref{eq:theta_gamma}; since $\hat{\vartheta}$ solves the problem $\dProb$ in particular it satisfies $-\DIV \, \hat{\vartheta} = f$ or equivalently $\int u \,df = \int \langle\nabla u, \hat{\vartheta} \rangle$ for any $\,u \in \D(\Rd)$.
	The compliance $\Comp(\check{\alpha})$ in \eqref{eq:compliance_def_cond} can be thus rewritten and then estimated as follows 
	\begin{align*}
	&\Comp(\check{\alpha}) = \sup \left\{ \int \langle\nabla u, \hat{\vartheta}\rangle - \frac{1}{2} \int \pairing{\check{\alpha},\nabla u \otimes \nabla u}
	\ : \ u \in \D(\R^d) \right\} \\
	\leq & \sup \left\{ \int \langle v, \hat{\vartheta}\rangle - \frac{1}{2} \int \pairing{\check{\alpha},v \otimes v} \ : \ v \in \bigl(\D(\R^d) \bigr)^d \right\}\\
	= &  \sup \left\{ \int_{\Ob \times \Ob}  \int_{[x,y]} \left( \pairing{v(z),\frac{x-y}{\abs{x-y}}} -\frac{\Totc}{2Z}\biggl(\pairing{v(z),\frac{x-y}{\abs{x-y}}}\biggr)^2 \right)\!\Ha^1(dz) \,  \hat\gamma(dx dy)  :  v \in \bigl(\D(\R^d) \bigr)^d \right\}\\
	\leq & \int_{\Ob \times \Ob} \left( \int_{[x,y]} \frac{Z}{2\Totc} \,  \Ha^1(dz) \right)  \hat\gamma(dx dy) = \frac{Z}{2\Totc} \,\int_{\Ob \times \Ob} \abs{x-y}\,  \hat\gamma(dx dy) = \frac{Z^2}{2\Totc},
	\end{align*}
	where in the last inequality we substituted $t:= \langle v(z),\frac{x-y}{\abs{x-y}} \rangle$ and used the fact that $\sup_{t\in \R} \bigl\{t-  \frac{\Totc}{2 Z}\, t^2\bigr\} = \frac{Z}{2\Totc}$. The last term in the chain above: $\frac{Z^2}{2\Totc}$ is precisely the value of minimum compliance $\Cmin$ (see Theorem \ref{thm:problem_P} for the vectorial case) proving $\check{\alpha}$ is a solution of the scalar $(\mathrm{FMD})$ problem.
\end{proof}

	\bigskip
	\footnotesize
	\noindent\textbf{Acknowledgments.}
	The authors would like to thank the National Science Centre (Poland) for the financial support: the first author would like to acknowledge the Research Grant no 2015/19/N/ST8/00474 (Poland), entitled "Topology optimization of thin elastic shells - a method synthesizing shape and free material design"; the second author would like to acknowledge the Research Grant no 2019/33/B/ST8/00325 entitled "Merging the optimum design problems of structural topology and of the optimal choice of material characteristics. The theoretical foundations and numerical methods".

\end{document}